\documentclass[11pt,a4paper,reqno]{amsart}
\usepackage[francais]{babel}
\usepackage[latin1]{inputenc}
\usepackage[T1]{fontenc}
\usepackage{latexsym,amsfonts,amssymb,amsmath,amsthm,textcomp,mathtools}
\usepackage{longtable}
\usepackage{graphicx}
\usepackage{tikz}
\usepackage{fancyhdr}
\usepackage[hidelinks]{hyperref}
\usepackage{calc}
\usepackage{chngcntr}

\theoremstyle{plain}
\newtheorem{thm}{Th\'eor\`eme}[subsection]
\newtheorem{definition}[thm]{D\'efinition}
\newtheorem{proposition}[thm]{Proposition}

\newtheorem{thmm}{Théorème}[section]
\newtheorem{Propositionn}[thmm]{Proposition}
\newtheorem{lemmee}[thmm]{Lemme}
\newtheorem{corollaire}[thm]{Corollaire}

\newtheorem{théorème}[thm]{Théorème}
\newtheorem{lemme}[thm]{Lemme}

\newtheorem{Construction}[thm]{Construction}
\newcommand*{\bigboxplus}{\DOTSB\mathop{\mathpalette\big@boxplus\relax}\slimits@}
\newcommand{\op}{
	\mathop{
		\vphantom{\bigoplus} 
		\mathchoice
		{\vcenter{\hbox{\resizebox{\widthof{$\displaystyle\bigoplus$}}{!}{$\boxplus$}}}}
		{\vcenter{\hbox{\resizebox{\widthof{$\bigoplus$}}{!}{$\boxplus$}}}}
		{\vcenter{\hbox{\resizebox{\widthof{$\scriptstyle\oplus$}}{!}{$\boxplus$}}}}
		{\vcenter{\hbox{\resizebox{\widthof{$\scriptscriptstyle\oplus$}}{!}{$\boxplus$}}}}
	}\displaylimits 
}

\makeatletter
\@ifpackageloaded{hyperref}%
{\newcommand{\mylabel}[2]
	{\protected@write\@auxout{}{\string\newlabel{#1}{{#2}{\thepage}%
				{\@currentlabelname}{\@currentHref}{}}}}}%
{\newcommand{\mylabel}[2]
	{\protected@write\@auxout{}{\string\newlabel{#1}{{#2}{\thepage}}}}}
\makeatother

\theoremstyle{definition}
\newtheorem{remarque}[thm]{Remarque}
\newtheorem{notation}[thm]{Notation}

\newtheorem*{demn}{Démonstration du théorème principal}

\newtheorem{exemple}[thm]{Exemple}

\counterwithin{equation}{subsection}

\def \e {\underline{e}}

\def \N {\mathbb{N}}
\def \Q {\mathbb{Q}}
\def \Z {\mathbb{Z}}
\def \R {\mathbb{R}}
\def \C {\mathbb{C}}

\def \F {\mathbb{F}}

\def \G {\mathbb{G}}
\def \Gm {\mathbb{G}_m}

\def \X {\mathbb{X}}

\def \A {\mathbb{A}}

\def \Cent{\mathop{\mathrm{Cent}}\nolimits}

\def \hom{\mathop{\mathrm{Hom}}\nolimits}
\def \End{\mathop{\mathrm{End}}\nolimits}

\def \dim{\mathop{\mathrm{dim}}\nolimits}
\def \diag{\mathop{\mathrm{diag}}\nolimits}

\def \Fr{\mathop{\mathrm{Fr}}\nolimits}

\def \Lie{\mathop{\mathrm{Lie}}\nolimits}
\def \det{\mathop{\mathrm{det}}\nolimits}
\def \Sh{\mathop{\mathrm{Sh}}\nolimits}
\def \Aut{\mathop{\mathrm{Aut}}\nolimits}
\def \Id{\mathop{\mathrm{Id}}\nolimits}
\def \Res{\mathop{\mathrm{Res}}\nolimits}
\def \J{\mathop{\mathrm{J}}\nolimits}
\def \Spf{\mathop{\mathrm{Spf}}\nolimits}
\def \Lie{\mathop{\mathrm{Lie}}\nolimits}

\def \diag{\mathop{\mathrm{diag}}\nolimits}
\def \Tr{\mathop{\mathrm{Tr}}\nolimits}
\def \ker{\mathop{\mathrm{ker}}\nolimits}

\def \Ext{\mathop{\mathrm{Ext}}\nolimits}
\def \Ind{\mathop{\mathrm{Ind}}\nolimits}
\def \Im{\mathop{\mathrm{Im}}\nolimits}
\def \mod{\mathop{\mathrm{mod}}\nolimits}

\def \Tr{\mathop{\mathrm{Tr}}\nolimits}

\begin{document}
	\setcounter{tocdepth}{2}
	\title{Un cas PEL de la conjecture de Kottwitz}
	
	\author{NGUYEN Kieu Hieu}
\maketitle
\newtheorem*{Résumé}{Résumé}
	\begin{Résumé}
		La conjecture de Kottwitz décrit la cohomologie des espaces de Rapoport-Zink basiques à l'aide des correspondances de Langlands locales. Dans cet article, par voie globale via l'étude de la géométrie de certaines variétés de Shimura de type Kottwitz, on prouve cette conjecture pour des espaces de Rapoport-Zink de type PEL unitaires non ramifiés simples basiques de signature $ (1, n-1)$.   
	\end{Résumé}
\newtheorem*{Abstract}{Abstract}
\begin{Abstract}
	The Kottwitz conjecture describes the cohomology of basic Rapoport-Zink spaces using local Langlands correspondences. In this paper, via geometrical studies of some Kottwitz-type Shimura varieties, we prove this conjecture for basic simple unramified unitary PEL type Rapoport-Zink spaces of signature $(1, n-1)$.	
\end{Abstract}

\thispagestyle{fancy}
\fancyhf{}
\fancyfoot[L]{\textbf{Classification mathématique par sujet (2010)} 11F70, 11F80, 11F85, 11G18, 20C08}
	
	
	

	\tableofcontents
	\section*{Introduction}
	Les variétés de Shimura jouent un rôle important dans le programme de Langlands global qui prédit un lien entre représentations automorphes des groupes linéaires et représentations galoisiennes. Rapoport et Zink ont introduit des analogues $p$-adiques définis comme des espaces de modules de groupes $p$-divisibles munis de structures additionnelles  (\cite{RZ96}). La cohomologie $\ell$-adique $(\ell \neq p)$ de ces espaces devrait fournir l'incarnation locale des correspondances de Langlands, c'est le sujet de la conjecture de Kottwitz \cite{Rap94}.
	
	Selon la terminologie de \cite{RZ96}, les cas abordés jusqu'à présent sont tous de type EL : le cas Lubin-Tate a été entièrement traité dans \cite{Boyer99}, \cite{HT01} et \cite{Boyer09}; par dualité \cite{Fal} \cite{FGL} \cite{SW17}, on obtient aussi le cas de Drinfeld et le cas EL général est démontré par \cite{Far04}, \cite{Shin}.
	
	Dans cet article on traite un cas PEL unitaire non ramifié simple basique où une donnée de Rapoport-Zink est un uplet $ \mathcal{D}_{\Q_p} = (F_p, *, V, \langle \cdot | \cdot \rangle, G, \mu, b)$ auquel on associe un groupe $p$-divisible muni de structures additionnelles. On suppose de plus
	\begin{enumerate}
		\item[$\bullet$] $[F_p : \Q_p] = 2d$ avec $d$ impair,
		\item[$\bullet$] $ \dim_{F_p} V = n$ impair et $\mu = (1, n-1), (0, n), \cdots, (0, n)$.
	\end{enumerate}
	
	Via la notion de structure de niveau, à $\mathcal{D}_{\Q_p}$ est associée une tour d'espaces rigides $(\mathcal{M}_{K_p})_{K_p}$ indexée par les sous-groupes compacts ouverts de $G(\Q_p)$. Le groupe $J_b(\Q_p)$ des quasi-isogénies du groupe $p$-divisible avec structures additionnelles est, dans le cas où $b$ est basique, une forme intérieure du groupe unitaire quasi-déployé $p$-adique $G$. La tour $(\mathcal{M}_{K_p})_{K_p}$ est alors munie d'une action de $J_b(\Q_p) \times G(\Q_p)$ et pour $\ell \neq p$ 
	
	\[
	H^i_c (\mathcal{M}, \overline{\Q}_{\ell}) := \mathop{\mathrm{lim}}_{\overrightarrow{K_p}} H^i_c (\mathcal{M}_{K_p}, \overline{\Q}_{\ell})
	\]
	est une $\overline{\Q}_{\ell}$-représentation de $ G(\Q_p) \times J_b(\Q_p) \times W_{E_p} $, où $W_{E_p}$ est le groupe de Weil du corps de définition $E_p$ de $\mu\footnote{Le corps de définition du cocaractère $\mu = (1, n-1), (0, n), \cdots, (0, n)$ est $F_p$.}$. Pour l'expliciter, on utilise les paramètres de Langlands discrets $ \varphi : W_{\Q_p} \times SU(2) \longrightarrow \ \prescript{L}{}{G} $.
	
	Pour un tel $\varphi$ on note $\Pi_{\varphi}(G(\Q_p))$ (respectivement $\Pi_{\varphi}(J_b(\Q_p))$) le paquet de représentations de $G(\Q_p)$ (respectivement de $J_b(\Q_p)$) associé (cf \ref{itm: local} et \ref{itm : passage local}). Les éléments de $\Pi_{\varphi}(G(\Q_p))$ sont en bijection avec l'ensemble $ \text{Irr} (S^{\natural}_{\varphi}, \chi) $ des caractères de $S^{\natural}_{\varphi}$ dont le tiré en arrière via $Z(\widehat{G}^{Gal(\overline{\Q}_p / E_p)}) \hookrightarrow S_{\varphi} \twoheadrightarrow S^{\natural}_{\varphi}$ induit le caractère $ \chi $ (où $S_{\varphi}$ est le centralisateur de $\varphi$ dans $\widehat{G}$ et $\chi$ est le caractère de $Z(\widehat{G}^{Gal(\overline{\Q}_p / E_p)})$ associé à $G(\Q_p)$). Contrairement au cas EL, ces paquets ne sont pas en général des singletons ce qui sera à la source des difficultés techniques nouvelles de cet article. Rappelons que $\varphi$ est dit cuspidal s'il est trivial sur le facteur $SU(2)$  auquel cas le paquet $\Pi_{\varphi}(G(\Q_p))$ ne contient que des représentations supercuspidales (cf. \ref{itm : C.Moeglin}).
	
	D'après \cite{Lang79}, à $\mu$ est associé une représentation $ r_{\mu} $ de $\widehat{G} \rtimes W_{E_p} $ et on note $r_{\mu} \circ \varphi_{E_p} $ la représentation de $ S_{\varphi} \times W_{E_p} $ définie par la formule 
	\[
	(s, w) \in S_{\varphi} \times W_{E_p} \longmapsto r_{\mu}(s \cdot \varphi(w))
	\] 
	\newtheorem*{Conjec}{Conjecture}
	\begin{Conjec}(\textbf{Kottwitz}) 
		Fixons un $L$-paramètre $\varphi$ cuspidal. Soit $ \pi_p' \otimes \pi_p \otimes \sigma $ une représentation irréductible de $G(\Q_p) \times J_b(\Q_p) \times W_{E_p}$ qui contribue de manière non triviale dans $H^{*}_c(\mathcal{M}, \overline{\Q}_{\ell})$. Alors $\pi_p'$ appartient au $L$-paquet $\Pi_{\varphi} (G(\Q_p)) $ si et seulement si $\pi_p$ appartient au $L$-paquet $\Pi_{\varphi} (J_b(\Q_p)) $.
		
		De plus la contribution (au signe près) du $L$-paquet associé à $\varphi$ est donnée par la formule suivante :
		\[
		\sum_{(\pi_p', \pi_p) \in \Pi_{\varphi} (G(\Q_p)) \times \Pi_{\varphi} (J_b(\Q_p)) } \pi_p' \otimes {\pi}^{\vee}_p \otimes \hom_{S_{\varphi}}( \tau_{\pi_p'} \otimes \tau_{\pi_p} ,r_{\mu} \circ \varphi_{E_p} )
		\]
		où $ \tau_{\pi_p'} $ et $ \tau_{\pi_p} $ sont respectivement des représentations de $S_{\varphi}$ correspondant à $\pi_p'$ et $\pi_p$  et où ${\pi}^{\vee}_p$ signifie la représentation contragrédiente.
		
	\end{Conjec}
	
	\newtheorem*{remarqu}{Remarque}
	\begin{remarqu}
		Puisque $n$ est impair on a $G(\Q_p) = \J_b(\Q_p)$.
	\end{remarqu}
	Pour $\pi_p$ une représentation irréductible supercuspidale de $J_b(\Q_p)$, on écrit
	\begin{equation*} 
		\mathop{\mathrm{lim}}_{\overrightarrow{K_p}}  \hom_{\J_b(\Q_p)} \left( H^{i}_c(\mathcal{M}_{K_p}, \overline{\Q}_{\ell}(n-1)), \pi_p \right)_{cusp}   = \sum_{\pi_p'} \pi_p' \otimes \sigma^i_{\pi_p, \pi_p'}	
	\end{equation*} 
	où $\pi_p'$ parcourt l'ensemble des classes d'équivalences de représentations supercuspidales de $G(\Q_p)$. On note également $ \sigma_{\pi_p, \pi_p'} = \sigma^{n-1}_{\pi_p, \pi_p'} $.
	
	Considérons un paramètre de Langlands cuspidal $ \varphi : W_{\Q_p}  \longrightarrow \ \prescript{L}{}{G} $. La restriction de $\varphi$ sur $W_{E_p}$ se décompose en une somme de représentations irréductibles : $ \varphi_{E_p} = \varphi^{n_1}_1 \oplus \cdots \oplus \varphi^{n_r}_r $. On a donc $S_{\varphi} = S_{\varphi}^{\natural} = (\Z / 2\Z)^r$ (cf. \ref{itm : centra}). Parmi les $2^{r-1}$ éléments de $\text{Irr} (S^{\natural}_{\varphi}, \chi)$, il y a $r$ éléments $\{ \tau_1, \cdots, \tau_r \}$ (cf. \ref{itm : rep}) qui jouent un rôle privilégié. Le théorème suivant est alors une version explicite de la conjecture de Kottwitz dans le cas considéré. 
	
	\newtheorem*{théorème*}{Théorème A}
	\begin{théorème*} \phantomsection \label{itm : principal}
		\mylabel{text : thm}{Théorème A}
		Pour $[F_p : \Q_p] = 2$, $n$ impair et $\mu = (1, n-1)$, soit $\varphi$ un paramètre de Langlands cuspidal et $\pi_p' \in \Pi_{\varphi}(G(\Q_p))$\footnote{On rappelle que $\varphi$ étant cuspidal alors $\pi_p'$ est nécessairement supercuspidale}. Alors pour $\pi_p$ supercuspidale, la représentation $\sigma^i_{\pi_p, \pi_p'}$ est nulle dans chacun des cas suivants
		\begin{enumerate}
			\item[i)]  $ i \neq n-1 $,
			\item[ii)] $\pi_p \notin \Pi_{\varphi}(\J_b(\Q_p))$
			\item[iii)] Pour $\pi_p \in \Pi_{\varphi}(\J_b(\Q_p)) = \Pi_{\varphi}(G(\Q_p))$ (puisque $G(\Q_p) = \J_b(\Q_p)$) et $\tau_{\pi_p} \cdot \tau_{\pi_p'} \notin\{ \tau_1, \cdots, \tau_r \} $. 
		\end{enumerate}
		De plus pour $\pi_p \in \Pi_{\varphi}(\J_b(\Q_p))$ et $\tau_{\pi_p} \cdot \tau_{\pi_p'} = \tau_i$, on a
		\[
		\sigma_{\pi_p, \pi_p'} = \left( r_{\mu_i} \circ \varphi^{n_i}_{i} \right) \otimes |\cdot|^{ - \frac{n-1}{2}}.
		\]
		
		où $\mu_i = (1, \dim \varphi^{n_i}_i - 1)$ pour $ 1 \leq i \leq r $.
	\end{théorème*}
	
	\newtheorem*{Remarque*}{Remarque}
	\begin{Remarque*}
		\begin{enumerate}
			\item[$\bullet$] Les représentations du groupe de Weil ci-dessus sont Frobenius semi-simples. 
			\item[$\bullet$] Pour chaque $1 \leq i \leq r$, il y a $2^{r-1}$ couples $(\pi_p, \pi_p') \in \Pi_{\varphi}(\J_b(\Q_p)) \times \Pi_{\varphi}(G(\Q_p)) $ de sorte que $\tau_{\pi_p} \cdot \tau_{\pi_p'} = \tau_i$ pour lesquels on a $ \sigma_{\pi_p, \pi_p'} = \left( r_{\mu_i} \circ \varphi^{n_i}_{i} \right) \otimes |\cdot|^{ - \frac{n-1}{2}} $.
			\item[$\bullet$] Il y a $r2^{r-1}$ couples $(\pi_p, \pi_p')$ parmi $(2^{r-1})^2$ couples pour lesquels $\sigma_{\pi_p, \pi_p'}  \neq 0$.  
		\end{enumerate}	
	\end{Remarque*}
	Le principe de la démonstration repose sur l'étude des variétés de Shimura $\Sh_{ / E}$ définies sur leur corps reflex $E = F$ où $F$ est quadratique imaginaire et plus particulièrement sur la géométrie de la fibre spéciale en une place $p$ inerte dans $E$ d'un modèle $\Sh_{/ \mathcal{O}_p}$ où $\mathcal{O}_p$ est l'anneau des entiers de $E_p$. 
	
	Dans un premier temps, on considère une variété de Shimura \textit{non compacte} $\Sh$ associée à un groupe algébrique $\overset{\bullet}{G}$ unitaire quasi-déployé en toutes les places finies avec $\overset{\bullet}{G}(\Q_p) = G(\Q_p)$ et où la signature associée est de la forme $(1, n-1)$. On note aussi $I$ la forme intérieure de $\overset{\bullet}{G}$ qui est quasi-déployé en toutes les places finies et où la signature associée est $(0, n)$. D'après \cite{Far04}, on a une uniformisation du lieu basique de $\Sh$ par des $\mathcal{M}_K$ avec une suite spectrale :
	$$ E_2^{tq} = | \ker^1 (\Q, \overset{\bullet}{G}) | \sum_{\substack{\Pi \in \mathcal{A}(I) \\ \Pi_{\infty} = \breve{\rho}}} \left( \mathop{\mathrm{lim}}_{\overrightarrow{K_p}} \Ext^t_{\J_{b}(\Q_p)} \left( H^q_c(\mathcal{M}_{K_p}, \overline{\Q}_{\ell}(n-1)), \Pi_p \right) \right) \otimes (\Pi^p) $$
	dont l'aboutissement est 
	$$ \mathop{\mathrm{lim}}_{\overrightarrow{K}} H^{t+q}((\Sh)_K^{an}(basic), \mathcal{L}_{\rho}^{an}). $$
	
	Lorsqu'on considère la partie $G(\Q_p)$-supercuspidale, d'après \cite{Man08}, \cite{Shen}, \cite{LS} on a :
	$$ H^{t+q}((\Sh)(basic), \mathcal{L}_{\rho})_{p-cusp} = H^{t+q}((\Sh), \mathcal{L}_{\rho})_{p-cusp}. $$
	
	En outre lorsque $\Pi_p$ est $J_b(\Q_p)$-supercuspidale, on a : 
	$$ \Ext^i_{\J_{b}(\Q_p)} \left( H^q_c(\mathcal{M}_{K_p}, \overline{\Q}_{\ell}(n-1)), \Pi_p \right) = 0 $$
	dès que $i > 0$.
	
	La suite spectrale ci-dessus dégénère donc
	\[
	| \ker^1 (\Q, \overset{\bullet}{G}) | \sum_{\substack{\Pi \in \mathcal{A}(I) \\ \Pi_{\infty} = \breve{\rho}}} \left( \mathop{\mathrm{lim}}_{\overrightarrow{K}} \hom_{\J_{b}(\Q_p)} \left( H^q_c(\mathcal{M}_{K_p}, \overline{\Q}_{\ell}(n-1)), \Pi_p \right)_{cusp}  \right) \otimes (\Pi^{\infty, p}) = \left( H^{q}(\Sh, \mathcal{L}_{\rho}) \right)_{p-cusp}.
	\]
	
	Par un argument de globalisation, on en déduit que 
	\begin{equation} \label{eqn:*}
		\sum_{\pi_p \in \Pi_{\varphi}(J_b(\Q_p))} \sigma_{\pi_p, \pi_p'} = \left( r_{\mu} \circ \varphi_{E_p} \right)  \otimes | \cdot |^{-\frac{n-1}{2}} = \sum_{i = 1}^r \left( r_{\mu_i} \circ \varphi^{n_i}_{i} \right)  \otimes | \cdot |^{-\frac{n-1}{2}}. \tag{$\star$}
	\end{equation}
	
	
	Dans le but d'identifier chacun des termes de la somme à gauche de \eqref{eqn:*} avec l'un des termes dans la somme à droite, on devra considérer des formes automorphes provenant d'un groupe endoscopique de $\overset{\bullet}{G}$. Afin de mener à bien cette stratégie, on utilise la formule de multiplicité pour les groupes unitaires \cite{KMSW} pour construire des formes automorphes satisfaisant des propriétés particulières ainsi que les résultats de \cite{Mo} pour obtenir des informations plus fines sur la cohomologie de variétés de Shimura.
	\begin{Remarque*}
		Un ingrédient important de la démonstration est le théorème $7.2.2$ de \cite{Mo} sur la composante isotypique dans la cohomologie d'intersection de variétés de Shimura de type PEL unitaires non compactes sur un corps CM $F$. Dans cet article l'auteur a supposé $F^{+} = \Q$ et c'est la raison pour laquelle on suppose $d = 1$. Toutefois si on disposait d'une formule analogue à \cite{Mo} 7.2.2, notre méthode permettrait d'obtenir le théorème précédent pour $d$ quelconque.
		
		Il est en revanche possible, en utilisant les variétés de type Kottwitz-Harris-Taylor, de prouver la formule \eqref{eqn:*}, cf. appendice \ref{itm : appendice}, théorème \ref{itm : thm partiel}. 
	\end{Remarque*} 
	\newtheorem*{remerciements*}{Remerciements}
	\begin{remerciements*}
		Je remercie profondément Pascal Boyer et Laurent Fargues tant pour leur aide mathématique déterminante que pour leurs constants encouragements. Je remercie chaleureusement Sophie Morel, Colette Moeglin, Tasho Kaletha et Sug Woo Shin de m'avoir expliqué leurs travaux.	
	\end{remerciements*}
	\section{Données géométriques}
	\subsection{Espaces de Rapoport-Zink d'après \cite{RZ96}}
	Fixons un nombre premier $p$. Soit $ \breve{\Q}_p := \widehat{\Q_p^{\text{nr}}} = \text{Frac}W(\overline{\F}_p)$ le complété de l'extension maximale non ramifiée de $\Q_p$ et $\sigma$ l'automorphisme de Frobenius géométrique de $\breve{\Q}_p / \Q_p$.
	\begin{definition}
		Étant donné un groupe réductif $G$ défini sur $\Q_p$, deux éléments $b_1$, $b_2$ sont dits $\sigma$-conjugués s'il existe $g \in G(\breve{\Q}_p)$ tel que $ b_1 = gb_2g^{-\sigma} $. On note $B(G)$ l'ensemble des classes de $\sigma$-conjugaisons dans $G(\breve{\Q}_p)$.
	\end{definition}
	\begin{remarque}
		D'après Kottwitz \cite{Kot97} section 6.2, on s'intéressera dans la suite à un sous-ensemble $B(G, \mu)$ de $B(G)$ associé à un cocaractère minuscule $ \mu : \G_{m/\overline{\Q}_p} \longrightarrow G_{\overline{\Q}_p}$ (un cocaractère qui ne possède que des poids $0$ et $1$). Il existe un ordre partiel sur $B(G, \mu)$.
	\end{remarque}

	\begin{definition} \label{itm: defPEL} Une donnée de Rapoport-Zink de type PEL unitaire non ramifiée simple $ (F_p, *, V, \langle \cdot | \cdot \rangle, GU, \mu, b)$ consiste en la donnée:
		\begin{enumerate}
			\item[-] d'une extension $F_p$ de degré $2d$ de $\Q_p$ non ramifiée munie d'une involution non triviale $*$,
			\item[-] d'un $F_p$-espace vectoriel de dimension finie $V$,
			\item[-] d'un produit hermitien symplectique $ \langle \cdot | \cdot \rangle : V \times V \longrightarrow \Q_p $ pour lequel il existe un réseau auto-dual $\Lambda$,  
			\item[-] d'une classe de conjugaison de cocaractères minuscules $\mu : \G_{m / \overline{\Q}_p} \longrightarrow GU_{\overline{\Q}_p} $ où $GU$ est le groupe des similitudes unitaires défini sur $\Q$ par
			\[
			GU(R) = \big\{ g \in GL(V \otimes R) | <gv, gw> = \nu (g) <v,w>, \ v,w \in V \otimes R \big\}
			\] 
			pour toute $\Q$-algèbre $R$ et $\nu(g) \in R^{\times}$.
			\item[-] d'une classe de $\sigma$-conjugaison $b \in B(GU, \mu) $. On suppose de plus que $ c \circ \mu(z) = z $ où $c$ est le facteur de similitude. Avec ces hypothèses, un tel $\mu$ est déterminé par des couples $(p_{\tau}, q_{\tau})_{\tau \in \Phi \amalg \Phi^{*}}$ où $(p_{\tau}, q_{\tau}) = (q_{\tau*}, p_{\tau*})$ et où $\Phi$ est un type CM $p$-adique de $F_p / \Q_p$.		
		\end{enumerate}
	\end{definition}
	A une telle donnée, on associe  l'isocristal $ N = \Big( V \otimes_{\Q_p} \breve{\Q}_p, b \circ ( Id \otimes \sigma)\Big) $ muni d'une action $ \iota: \mathcal{O}_{F_p} \longrightarrow End(N) $  et une forme alternée non dégénérée $\langle \cdot | \cdot \rangle : N \times N \longrightarrow \breve{\Q}_p(n) $ où $n = val_p(c(b))$. Par la théorie de Dieudonné, l'isocristal $N$ correspond à un groupe $p$-divisible $(\mathbb{X}, \iota, \lambda)$ défini sur $\overline{\F}_p$ muni d'une action de $\mathcal{O}_{F_p}$ et d'une polarisation $\lambda$.
	
	\begin{théorème} (\cite{RZ96}, theo. 3.25)
		Soit $\mathcal{M}$ le foncteur qui associe à chaque $\mathcal{O}_{\breve{\Q}_p}$-schéma $S$ sur lequel $p$ est localement nilpotent l'ensemble des couples $(X, \rho)$ où :
		\begin{enumerate}
			\item [-] $X$ est un groupe $p$-divisible sur $S$ muni d'une polarisation $p$-principale $\lambda_X$ et d'une action $\iota_X$ telles que l'involution de Rosati induite par $\lambda_X$ induit $*$ sur $\mathcal{O}_{F_p}$. 
			\item [-] Une quasi-isogénie $\mathcal{O}_{E_p}$-linéaire $ \rho : X \times_S \overline{S} \longrightarrow \mathbb{X} \times_{Spec(\overline{\F}_p)} \overline{S} $ telle que $\rho^{V} \circ \lambda_X \circ \rho $ est un $\Q_p$-multiple de $\lambda_X$ dans $\mathop{\mathrm{Hom}}\nolimits_{\mathcal{O}_{E_p}}(X, X^{V})\otimes_{\Z}\Q$. (Ici, $\overline{S}$ est la réduction modulo $p$ de $S$).	
		\end{enumerate}
		
		On demande également que $(X, \iota_X)$ satisfasse la condition de déterminant de Kottwitz. Plus précisément, sous l'action de $F_p$, on a une décomposition: $ \Lie(X) = \bigoplus_{\tau} \Lie(X)_{\tau} $ alors $\Lie(X)_{\tau}$ est localement libre de rang $p_{\tau}$. Ce foncteur est alors représenté par un schéma formel $\mathcal{M}(\mu, b)$ défini sur $\Spf (\mathcal{O}_{\breve{\Q}_p})$. 
	\end{théorème}
	\begin{remarque}
		Dans \cite{RZ96}, les auteurs considèrent également les espaces de Rapoport-Zink de type EL. Les lecteurs intéressés pourront consulter loc. cit. pour plus de détails.
	\end{remarque}
	\begin{notation}
		On pose $ C_0 = \{ g \in G(\Q_p) \ \vert \ g\Lambda = \Lambda \} $, un sous-groupe compact maximal de $G(\Q_p)$.
	\end{notation}
	
	Afin d'introduire les structures de niveau usuelles comme dans le cas $GL(2)$ on travaille avec les espaces analytiques $\mathcal{M}^{\text{an}}$ de $\mathcal{M}$ sur $\breve{\Q}_p$.
	
	\begin{definition}
		Soit $\mathcal{T} / \mathcal{M}^{\text{an}} $ le système local défini par le module de Tate $p$-adique du groupe $p$-divisible universel sur $\mathcal{M}$. Pour $K \subset C_0$ on définit $\mathcal{M}_K$ comme le revêtement étale de $\mathcal{M}^{\text{an}}$ qui classifie les $\mathcal{O}_{F_p}$ trivialisations modulo $K$ de $\mathcal{T}$ par $\Lambda$. Dans le cas $PEL$ on demande de plus que les trivialisations préservent la forme alternée à $\Q_p^{\times}$ près. 
	\end{definition}
	
	On a $\mathcal{M}^{an} = \mathcal{M}_{C_0}$. D'autre part il y a une tour $(\mathcal{M}_{K_p})_{K_p}$ d'espaces analytiques sur $\breve{\Q}_p$ munis de morphismes de transition étales finis pour $K_p^{'} \subset K_p $ :
	\[
	\Phi_{K_p^{'}, K_p} : \ \mathcal{M}_{K^{'}_p} \ \longrightarrow \ \mathcal{M}_{K_p} 
	\]
	d'oubli de la structure de niveau. Le morphisme $\Phi_{K_p^{'}, K_p}$ est galoisien de groupe de Galois $ K_p / K_p^{'} $ si $K_p^{'}$ est normal dans $K_p$.
	
	\begin{proposition} (\cite{Far04} p.16)
		La dimension $d_{K_p}$ de $\mathcal{M}_{K_p}$ est donnée par la formule $ d_{K_p} = \frac{1}{2}\sum_{\tau \in I_F} p_{\tau}q_{\tau} $.
	\end{proposition}
	
	\begin{definition}
		Soit $J(\Q_p)$ le groupe des quasi-isogénies $\mathcal{O}_{F_p}$-linéaires $g$ de $\X$ telles que $\lambda \circ g$ est une $\Q^{\times}$-multiple de $g^{\vee} \circ \lambda $. Le groupe $J(\Q_p)$ agit à gauche sur $\mathcal{M}$ (dans le cas EL et PEL) par la formule
		\[
		\forall g \in J(\Q_p) \ \forall (X, \rho) \in \ \mathcal{M} \quad (X, \rho) \cdot g = (X, \rho \circ g^{-1}). 
		\]
	\end{definition}
	
	\begin{definition} \label{itm: basique}
		Une donnée de Rapoport-Zink non ramifiée simple $ (F_p, *, V, \langle \cdot | \cdot \rangle, G, \mu, b)$ est basique si le groupe $J(\Q_p)$ associé est une forme intérieure de $G$. La donnée ci-dessus est basique si et seulement si $b$ est l'élément minimal dans $B(G, \mu)$. Dans ce cas, on dit également que $b$ est basique.
	\end{definition}
	
	
	
	Soit $ \ell \neq p$ un nombre premier.
	\begin{notation}
		Soit $K_p \subset C_0 $ un niveau. On pose:
		\[
		H_c^{\bullet}(\mathcal{M}_{K_p}, \Q_{\ell}) := \mathop{\mathrm{lim}}_{\overrightarrow{V}} \Big( \mathop{\mathrm{lim}}_{\overleftarrow{n}} H_c^{\bullet}(V \otimes_{\breve{\Q}_p} \C_p, \Z / \ell^n \Z ) \otimes \Q_{\ell} \Big)
		\] 
		où $V$ parcourt les ouverts relativement compacts de $\mathcal{M}_{K_p}$.
	\end{notation}
	
	Le groupe $\J_b(\Q_p)$ agit sur $\mathcal{M}_{C_0}$ et cette action s'étend à $\mathcal{M}_{K_p}$ de sorte que $\J_b(\Q_p)$ agit sur les $H_c^{\bullet}(\mathcal{M}_{K_p}, \Q_{\ell})$. On peut aussi définir une action du groupe de Weil $W_{E_p}$ sur ces groupes de cohomologie grâce à la donnée de descente de Rapoport-Zink définie comme suit. 
	
	Soit $\sigma_{E_p} : \breve{\Q}_p \xrightarrow{\sim} \breve{\Q}_p$ l'automorphisme de Frobenius relatif au corps de définition $E_p$ de $\mu$ (où $\breve{\Q}_p = \widehat{E_p^{nr}}$). On note $\overline{\sigma}_{E_p}$ le morphisme de Frobenius induit sur $\overline{\F}_p$. Pour $\mathbb{X}$ un groupe $p$-divisible défini sur $\overline{\F}_p$, on note $ F_{E_p} : \mathbb{X} \longrightarrow \overline{\sigma}_{E_p}^*\mathbb{X}$ le morphisme de Frobenius relatif. On construit un isomorphisme de foncteurs : $ \alpha : \mathcal{M} \longrightarrow \sigma_{E_p}^*\mathcal{M} $ comme suit.
	
	Pour $S$ un $\mathcal{O}_{\breve{\Q}_p}$-schéma sur lequel $p$ est nilpotent ainsi qu'un point $(X, \rho) \in \mathcal{M}(S)$, le point $(X^{\alpha}, \rho^{\alpha})$ associé dans $\sigma_{E_p}^*\mathcal{M}(S)$ est défini de la manière suivante:
	\begin{enumerate}
		\item[-] $ X^{\alpha} := X $ avec l'action de $\iota_{X^{\alpha}} := \iota_X$ (et avec la polarisation $\lambda_{X^{\alpha}} := \lambda_X$ dans le cas PEL)
		\item[-] $\rho^{\alpha} := \rho \circ F_{E_p}^{-1}.$
	\end{enumerate}
	
	L'isomorphisme de foncteurs $ \alpha : \mathcal{M} \longrightarrow \sigma_{E_p}^*\mathcal{M} $ est la donnée de descente de Rapoport-Zink associée à $\mathcal{M}$. Étant donné que la donnée de descente commute à l'action de $\J_b(\Q_p)$, les groupes $H_c^{\bullet}(\mathcal{M}_{K_p}, \Q_{\ell})$ sont munis d'une action de $\J_b(\Q_p) \times W_{E_p}$. De plus, lorsque $K_p$ varie, le système $(H_c^{\bullet}(\mathcal{M}_{K_p}, \Q_{\ell}))_{K_p}$ est muni d'une action de $ G(\Q_p) \times \J_b(\Q_p) \times W_{E_p} $.
	\begin{proposition} (\cite{Man08} theorem 8) \phantomsection \label{itm: RZ}
		Soit $\rho$ une représentation $\ell$-adique admissible de $\J_b(\Q_p)$. 
		\begin{enumerate}
			\item[-] Les groupes $$H^{i,j}(\mathcal{M}^{\infty})_{\rho} := \mathop{\mathrm{lim}}_{\overrightarrow{K}} \Ext^j_{\J_b(\Q_p)}(H^i(\mathcal{M}_K, \Q_{\ell}(d_K)), \rho)$$ sont nuls pour presque tous $i,j \geq 0 $.
			\item[-] Les représentations $ H^{i,j}(\mathcal{M}^{\infty})_{\rho} $ sont admissibles.
		\end{enumerate}
	\end{proposition}
	\subsection{Variétés de Shimura de type PEL unitaire} \label{itm : Sh}
	On considère la donnée de Shimura de type PEL simple suivante:  $\mathcal{D} = \Big(F, B, *, V, \langle \cdot | \cdot \rangle , G, \Lambda, h \Big) $ où:
	\begin{enumerate}
		\item[-] $F$ est un corps $CM$ non ramifié au-dessus de $p$.
		\item[-] $B$ est une algèbre semi-simple sur $F$, déployée en toutes les places de $F$ au-dessus de $p$. 
		\item[-] $*$ est une involution positive sur $B$ : $ \forall b \in B \ tr(bb^*) > 0 $. 
		\item[-] $ \Big( V, \langle \cdot | \cdot \rangle \Big)$ est un $B$-module hermitien où $\langle \cdot | \cdot \rangle$ est une $\Q$-forme alternée telle que $ \langle xv, w \rangle = \langle v, x^*w \rangle $ pour tous $v, w \in V $ et $ x \in B $.
		\item[-] $G$ est le groupe algébrique défini sur $\Q$ par 
		$$ G(R) = \{ g \in GL_{B\otimes_{\Q} R}  (V \otimes_{\Q} R) \vert \langle gv, gw \rangle = c(g)\langle v, w \rangle; c(g) \in R^* \} $$
		On suppose de plus que $G_{\R}$ est isomorphe au groupe des similitudes unitaires de signature $(1, n-1), (0,n), \cdots, (0,n)$ où $n = [B : F]^{1/2} \text{rg}_BV$.
		\item[-] $\Lambda$ est un $\mathcal{O}_{F_p}$ réseau $\mathcal{O}_B$-invariant dans $V \otimes_{\Q}{\Q_p}$ tel que la forme $\langle \cdot | \cdot \rangle$ induit une $\Z_p$ forme non dégénérée sur $\Lambda$. On pose également $C_0 = Stab_{G(\Q_p)}(\Lambda) = \{ g \in G(\Q_p) \ \vert \ g \Lambda = \Lambda \}$.
		\item[-] Enfin, $h$ est un morphisme de groupes algébriques
		\[
		h : \mathop{\mathrm{Res}}\nolimits_{\C / \R}(\G_{m, \C}) \longrightarrow G_{\R}
		\] 
		qui définit une $\Q$-structure de Hodge $V = V_0 \oplus V_1$.
		
		On lui associe un morphisme
		$$ \mu_h : \G_{m, \C} \hookrightarrow \displaystyle \prod_{Gal(\C / \R)} \G_{m, \C} = \Big(\mathop{\mathrm{Res}}\nolimits_{\C / \R}(\G_{m, \C})\Big)_{\C} \xrightarrow{h_{\C}} G_{\C}. $$
	\end{enumerate}
	
	\begin{notation}
		Soit $E$ le corps reflex de cette donnée de type PEL, à savoir le corps de définition de la classe de conjugaison de $\mu_h$. Lorsque $n > 2$ on a $F = E$.
	\end{notation}
	
	Fixons $\Phi \subset \hom(E, \C) $ un type $CM$ de $E$. Chaque élément $\tau \in \Phi$ fournit un plongement $ \nu \circ \tau : E \hookrightarrow \overline{\Q}_p $ où $ \nu $ est un isomorphisme abstrait entre $ \C $ et $ \overline{\Q}_p $. Soient $ (w_i)_{i \in I} $, $(w_j)_{j \in J}$ les places de $E$ divisant $p$ associées à tous ces plongements et où on suppose que $ \forall i \in I \ w_i \neq w_i^c $ et $ \forall j \in J \ w_j = w_j^c $. On a alors :
	\begin{enumerate}
		\item[-] $
		G_{\Q_p} \simeq \prod_{i \in I} GL_n(E_{w_i}) \times G \left( \prod_{j \in J} GU(E_{w_j}, n) \right)
		$
		où le $G$ devant le produit signifie que l'on prend le sous-groupe du produit formé des uplets ayant le même facteur de similitude.
		\item[-] $
		B_{\Q_p} \simeq \prod_i \Big( M_d(E_{w_i}) \times M_d(E_{w_i})^{\text{opp}} \Big) \times \prod_j M_d(E_{w_j}).
		$
	\end{enumerate}
	
	L'équivalence de Morita permet de supposer qu'en chaque place on est dans l'un des cas suivant :
	\begin{enumerate}
		\item \textbf{Cas EL}
		\begin{enumerate}
			\item[-] $ B_{\Q_p} = E_{w_i} \times E_{w_i} $, \quad $\mathcal{O}_{B_{\Q_p}} = \mathcal{O}_{E_{w_i}}  \times \mathcal{O}_{E_{w_i}} $ et $(x, y)^* = (y, x)$.
			\item[-] $V_{\Q_p} = V_i \oplus V_i^{\vee}$ où $V_i^{\vee}$ est l'espace dual de $V_i$.
			\item[-] $ \langle x \oplus \phi, x' \oplus \phi'\rangle = \phi'(x) - \phi (x')$.
		\end{enumerate}
		\begin{remarque}
			Pour chaque $b \in B(GL_n(E_{w_i}), \mu_{\overline{\Q}_p})$, on obtient une donnée de Rapoport-Zink simple de type $EL$ de la forme $\{ E_{w_i}, *, V_i, \mu_{\overline{\Q}_p}, b \} $.
		\end{remarque}
		\item \textbf{Cas PEL}
		\begin{enumerate}
			\item[-] $B_{\Q_p} = E_{w_j}$, $\mathcal{O}_{B_{\Q_p}} = \mathcal{O}_{E_{w_j}}$ et l'involution $*$ est égale à $\sigma^{[E_{w_j} / \Q_p] / 2}$ où $\sigma$ est le morphisme de Frobenius de l'extension non ramifiée $E_{w_j} / \Q_p $.
			\item[-] $V$ est un $E_{w_j}$-espace vectoriel de dimension $n$.
			\item[-] Dans une base convenable de $V$, la forme $\langle \cdot | \cdot \rangle$ est donnée par la formule:
			$$ \forall \ X,Y \in V \quad \langle X,Y \rangle = \Tr_{E_{w_j}/\Q_p} (\alpha ^{t}X^{*}JY) $$
			où $\alpha \in E_{w_j}$ est tel que $\sigma(\alpha) = - \alpha$, et
			$$ J = \left( \begin{array}{cc}
				0 & I_{n/2} \\
				I_{n/2} & 0
			\end{array} \right) $$
			si $n$ est pair et
			$$ J = \left( \begin{array}{ccc}
				0 & 0 & I_{(n-1)/2} \\ 
				0 & 1 & 0 \\  
				I_{(n-1)/2} & 0 & 0
			\end{array} \right) $$
			si $n$ est impair.
		\end{enumerate} 
		\begin{remarque}
			Pour chaque $b \in B(GU(E_{w_j}, n), \mu_{\overline{\Q}_p})$, on obtient une donnée de Rapoport-Zink simple de type $PEL$ de la forme $\{ E_{w_j}, *, V, \langle \cdot | \cdot \rangle, \mu_{\overline{\Q}_p}, b \} $.
		\end{remarque}
	\end{enumerate}	
	
	Soit $K$ un sous-groupe compact ouvert suffisamment petit de $G(\A_f)$ et soit $\Sh_K$ le foncteur qui associe à chaque $F$-schéma $S$ l'ensemble des quadruplets $(A, \lambda, \iota, \overline{\kappa})$ où 
	\begin{enumerate}
		\item[-] $A$ est un $S$-schéma abélien à isogénie près.
		\item[-] $\lambda$ est une polarisation $\Q^{\times}$-homogène de $A$.
		\item[-] $\iota: B \longrightarrow \End(A)_{\Q}$ est un morphisme d'algèbres tel que $*$ correspond à l'involution de Rosati associée à $\lambda$.
		\item[-] $ \overline{\kappa} : V\otimes \A_f \longrightarrow H_1(A, \A_f) $ est un morphisme de $B \otimes \A_f$-modules symplectiques définissant une structure de niveau $K$ sur le module de Tate de $A$.
		\item[-] On suppose de plus la condition suivante:
		$$ \forall b \in B \det(b, \Lie(A)) = \det (b, V_0) $$
	\end{enumerate}
	
	\begin{théorème} (\cite{Kot92} sec.$5$, sec. $8$)
		Le foncteur $(\Sh_K)$ est représenté par une variété quasi-projective lisse définie sur $F$. De plus la tour de variétés $(\Sh_K)_K$ est munie d'une action  de $G(\A_f)$ par action sur la structure de niveau.	
	\end{théorème}
	
	Maintenant on s'intéresse aux modèles entiers des variétés de Shimura.
	
	Considérons un sous-groupe compact ouvert $K^p$ de $G(\mathbb{A}_f^p)$. Soit $S_{K^p}$ le foncteur qui associe à un $\mathcal{O}_{E_p}$-schéma $S$, l'ensemble des couples $(A, \overline{\lambda}, \iota, \overline{\kappa})$, où:
	\begin{enumerate}
		\item[-] $A$ est une variété abélienne sur $S$ de dimension $n$.
		\item[-] $\overline{\lambda}$ est une $\Q^{\times}$-classe de polarisations $p$-principales.
		\item[-] $ \iota : \mathcal{O}_{B} \longrightarrow End(A) \otimes_{\Z} \Z_{(p)} $ est tel que l'involution de Rosati définie par $\overline{\lambda}$ sur $End(A) \otimes_{\Z} \Z_{(p)}$ induit l'involution $*$ sur $\mathcal{O}_B$. On suppose de plus que $A$ satisfait la condition de signature de Kottwitz.
		\item[-] $\overline{\kappa} : H_1(A, \mathbb{A}_f^p) \longrightarrow V \otimes_{\Q} \mathbb{A}_f^p $ est une $K^p$-structure de niveau.
	\end{enumerate}
	\begin{théorème} (\cite{Kot92} sec.$5$) \phantomsection \label{itm:modele entier hors p}
		Le problème de module $S_{K^p}$ est représenté par un $\mathcal{O}_{E_p}$ schéma lisse, quasi-projectif. Le groupe $G(\A_f^p)$ opère sur la tour $(S_{K^p})_{K^p}$ par action sur la structure de niveau.
		
		Soit $C_0 = Stab_{G(\Q_p)}(\Lambda)$ le sous-groupe compact hyperspécial associé. Il y a alors des isomorphismes compatibles pour $K^p$ variable:
		$$ S_{K^p} \otimes_{\mathcal{O}_{E_p}} E_p \xrightarrow{\sim} \Sh_{C_0K^p} \otimes_E E_p. $$	
	\end{théorème}
	
	
	Supposons que la variété de Shimura est \textit{compacte}. On rappelle quelques constructions de modèles entiers de variétés de Shimura avec niveau en $p$ (cf \cite{Man05} section $6$).
	
	Pour $ K^p \subset G(\A_f^p) $, on définit pour tout $ m \geqslant 0$
	$$ K^p(m) = K^p \times V_m \subset G(\A_f) $$ où $ V_m = \{ g \in G(\Q_p) \ \vert \ g(\Lambda) = \Lambda, \ c(g) \in \Z_p^{\times} \ g_{\vert \Lambda} \equiv 1 \mod p^m \Lambda \}$. On remarque que $V_0 = G(\Z_p)$.
	
	Lorsque $K^p$ et $m$ varient, les $K^p(m)$ forment un système direct de sous-groupes compacts ouverts suffisamment petits de $G(\A_f)$, cofinal au système de tous les sous-groupes compacts ouverts. Pour tout niveau $K^p$, d'après \ref{itm:modele entier hors p}, on dispose d'un $\mathcal{O}_{E_p}$-schéma $S_{K^p(0)}$ classifiant des variétés abéliennes avec structures additionnelles. Soit $\mathcal{G} := \mathcal{A}[p^{\infty}]$ le groupe $p$-divisible muni de structures additionnelles associé à la variété abélienne universelle de $S_{K^p(0)}$. Pour $m > 0$, on définit le foncteur $S_{K^p(m)}$ qui associe à un $S_{K^p(0)}$-schéma $T$ l'ensemble des morphismes de groupes 
	\[
	\alpha : p^{-m} \Lambda / \Lambda \longrightarrow \mathcal{G}[p^m](T)
	\]
	satisfaisant les conditions 
	\begin{enumerate}
		\item[-] $\{ \alpha(x) | x \in p^{-m}\Lambda / \Lambda \}$ est un << full set of sections >> de $\mathcal{G}[p^m]_{T} / T$,
		\item[-] $\alpha$ est $\mathcal{O}_{B_{\Q_p}}$-équivariant, 
		\item[-] $\alpha$ envoie $\langle \cdot | \cdot \rangle$ sur l'accouplement de $Weil$ de $\mathcal{G}[p^m](T)$, à un facteur dans $(\Z / p^m \Z)^{\times}$ près.
	\end{enumerate}
	\begin{théorème} (\cite{Man05}  proposition 15)
		Le foncteur $S_{K^p(m)}$ est représentable par un $S_{K^p(0)}$-schéma fini.
	\end{théorème}
	
	Lorsque $K^p$ et $m$ varient, les schémas $S_{K^p(m)}$ forment un système projectif muni d'une action de $G(\A^p_f) \times V_0 \subset G(\A_f) $. En général, on ne peut pas étendre cette action en une action de $G(\A_f)$, faute de bon modèles entiers en $p$. Afin de contourner cette difficulté, Mantovan considère une classe plus grande de modèles entiers de variétés de Shimura telle que l'action de $G(\A^p_f) \times V_0$ s'étend en une action d'un sous-monoïde $G(\A_f)^{+}$ vérifiant $G(\A_f) = \langle G(\A_f)^{+}, p \rangle$. 
	\begin{remarque}
		Dans le cas où la variété de Shimura \textit{n'est plus compacte}, les modèles entiers avec niveau en $p$ existent encore grâce au travail de Kai Wen Lan et Benoît Stroh. Pour $K$ un niveau, pas forcément maximal en $p$, il existe un modèle entier $ S_K \longrightarrow \text{Spec}(\mathcal{O}_{E_p}) $ comme dans le cas (Nm) de 2.1 de \cite{LS} (consulter également la section 6 de \cite{Lan}). Lorsque le niveau $K$ varie, les schémas $S_K$ forment un système projectif muni d'une action de $G(\A_f)$. (Consulter \cite{LS} page 33-34 pour plus de détails). 	
	\end{remarque}
	
	
	
	
	À toute $\overline{\Q}_{\ell}$-représentation algébrique irréductible de dimension finie $\xi$ de $G$, on associe, cf \cite{HT01} p.96, un système local $\mathcal{L}_{\xi}$ sur la tour de variétés de Shimura $(\Sh_K)_K$. En particulier, lorsque $\xi = 1$, on retrouve le système local $\overline{\Q}_{\ell}$.
	
	Nous allons utiliser la stratification de Newton afin de calculer la cohomologie de la variété de Shimura. Nous ne considérons que la signature $\mu = (1, n-1), (0, n), \cdots, (0,n)$. Dans ce cas, $ B(G, \mu) = \{ b_0, b_1, \cdots, b_{[\frac{n}{2}]} \} $ et de plus $ b_0 \prec \cdots \prec b_{[\frac{n}{2}]} $ où $b_0$ est l'unique classe basique. Rappelons que l'on associe à $\mu$ un polygone de Hodge et à chaque $b \in B(G, \mu)$ un polygone de Newton, en particulier le polygone de Newton correspondant à $b_{[\frac{n}{2}]}$ coïncide avec le polygone de Hodge. \textit{Remarquons que tous les polygones de Newton non basiques touchent le polygone de Hodge en dehors des points extrémaux.} On a un diagramme des polygones avec $n = 5$.
	
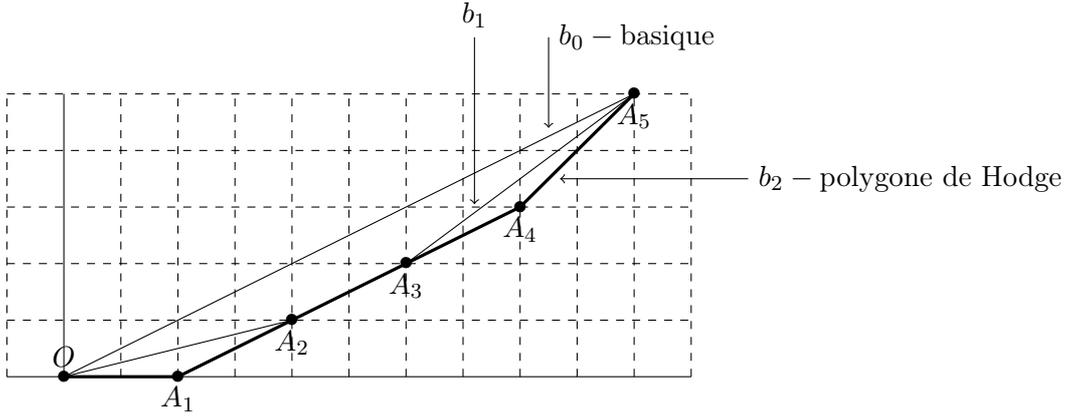
\begin{center}
		\begin{tikzpicture}[scale = 0.75]
			\draw (0,0) node [above] {$O$} node{$\bullet$};	
			\draw (2,0) node [below] {$A_1$} node{$\bullet$};
			\draw (4,1) node [below] {$A_2$} node{$\bullet$};
			\draw (6,2) node [below] {$A_3$} node{$\bullet$};
			\draw (8,3) node [below] {$A_4$} node{$\bullet$};
			\draw (10,5) node [below] {$A_5$} node{$\bullet$};
			\draw [dashed] (-1,1) -- (11,1);
			\draw [dashed] (-1,5) -- (11,5);
			\draw [dashed] (-1,2) -- (11,2);
			\draw [dashed] (-1,3) -- (11,3);
			\draw [dashed] (-1,4) -- (11,4);
			\draw [dashed] (1,0) -- (1,5);
			\draw [dashed] (2,0) -- (2,5);
			\draw [dashed] (3,0) -- (3,5);
			\draw [dashed] (4,0) -- (4,5);
			\draw [dashed] (5,0) -- (5,5);
			\draw [dashed] (6,0) -- (6,5);
			\draw [dashed] (7,0) -- (7,5);
			\draw [dashed] (8,0) -- (8,5);
			\draw [dashed] (9,0) -- (9,5);
			\draw [dashed] (10,0) -- (10,5);
			\draw [dashed] (11,0) -- (11,5);
			\draw [dashed] (-1,0) -- (-1,5);
			\draw [-] (-1,0) -- (11,0);
			\draw [-] (0,0) -- (0,5);
			\draw [-] (0,0) -- (4,1);
			\draw [-] (6,2) -- (10,5);
			\draw [very thick] (0,0) -- (2,0);
			\draw [very thick] (2,0) -- (8,3);
			\draw [very thick] (8,3) -- (10,5);
			\draw [-] (0,0) -- (10,5);
			\draw [->] (12,3.5) -- (8.7,3.5);
			\draw (12,3.5) node [above, right] {$b_2 - \text{polygone de Hodge} $};
			\draw [->] (8.5,6) -- (8.5,4.4);
			\draw (8.5,6) node [above, right] {$b_0 - \text{basique}$};
			\draw [->] (7.2,6) -- (7.2,3.05);
			\draw (7.2,6) node [above] {$b_1$};
		\end{tikzpicture}
	\end{center}
	\begin{notation}
		Pour tout entier $i \geq 0$ on notera
		\[
		H^i_c(\Sh, \mathcal{L}_{\xi}) = \mathop{\mathrm{lim}}_{\overrightarrow{K}} H^i_c(\Sh_K \times_{E_p} \overline{E}_p, \mathcal{L}_{\xi}).
		\]
		C'est un $\overline{\Q}_{\ell}$-espace vectoriel muni d'une action linéaire de $G(\A_f) \times W_{E_p}$. En notant $\Psi_{\eta}$ le foncteur cycles proches on a
		\[
		H^i_c( \overline{\Sh}, R\Psi_{\eta} (\mathcal{L}_{\xi})) := \mathop{\mathrm{lim}}_{\overrightarrow{K^p, m}} H^i R\Gamma_c( \overline{S}_{K^p(m)}, R\Psi_{\eta}(\mathcal{L}_{\xi})).
		\]
		Pour $b \in B(G, \mu)$ on notera
		\[
		H^i_c(\overline{\Sh}(b), R\Psi_{\eta} (\mathcal{L}_{\xi})) := \mathop{\mathrm{lim}}_{\overrightarrow{K^p, m}} H^i R\Gamma_c( \overline{S}_{K^p(m)}(b), R\Psi_{\eta}(\mathcal{L}_{\xi})_{\vert \overline{S}_{K^p(m)}(b) } ).
		\]
	\end{notation}
	
	\begin{théorème} (\cite{Man05}, \cite{Man11}, \cite{LS}) \phantomsection \label{itm : somme des strates}
		On a une suite spectrale $G(\A_f) \times W_{E_p}$-équivariante
		\[
		E_1^{p,q} = H_c^{p+q}(\overline{\Sh}(b_p), R\Psi_{\eta} (\mathcal{L}_{\xi})) \Longrightarrow H_c^{p+q}(\Sh, \mathcal{L}_{\xi}).
		\]
	\end{théorème}
	\begin{remarque}
		On ne suppose pas que la variété de Shimura soit compacte.
	\end{remarque}
	\begin{proof}
		La stratification de Newton de la fibre spéciale de la variété de Shimura induit une suite spectrale $G(\A_f) \times W_{E_p}$-équivariante (\cite{Man05}, theo. 22; \cite{Man11}, theo. 3.1 dans le cas où la variété de Shimura est compacte et \cite{LS}, theo. 6.32 lorsque la variété de Shimura n'est pas compacte) :
		\[
		E_1^{p,q} = H_c^{p+q}(\overline{\Sh}(b_p), R\Psi_{\eta} (\mathcal{L}_{\xi})) \Longrightarrow H_c^{p+q}(\overline{\Sh}, R\Psi_{\eta} (\mathcal{L}_{\xi})).
		\]   
		
		Or on a un isomorphisme $G(\A_f) \times W_{E_p}$-équivariant suivant
		\[
		R\Gamma_c(\Sh, \mathcal{L}_{\xi}) \simeq R\Gamma_c(\overline{\Sh}, R\Psi_{\eta} (\mathcal{L}_{\xi})).
		\]
		
		Lorsque la variété de Shimura est compacte, c'est un fait standard de la théorie de cycles proches; lorsque la variété de Shimura n'est plus compacte, c'est le corollaire 5.20 de \cite{LS}.
		
		Finalement on a une suite spectrale $G(\A_f) \times W_{E_p}$-équivariante :
		\[
		E_1^{p,q} = H_c^{p+q}(\overline{\Sh}(b_p), R\Psi_{\eta} (\mathcal{L}_{\xi})) \Longrightarrow H_c^{p+q}(\Sh, \mathcal{L}_{\xi}).
		\]
	\end{proof}
	
	\begin{théorème} \cite{Man05}, \cite{Man08}, \cite{Man11} \cite{LS}, \cite{Shen}. \phantomsection \label{itm : induction parabolique}
		On suppose que $b$ est une strate non basique telle que son polygone de Newton touche son polygone de Hodge à un point de rupture du polygone de Newton en dehors des points extrémaux. Alors les groupes de cohomologie $ \displaystyle   \sum _i (-1)^i H^i_c(\overline{\Sh}(b), R\Psi_{\eta} (\mathcal{L}_{\xi}))$ ne contiennent pas de représentation automorphe dont la composante en $p$ est une représentation supercuspidale de $G(\Q_p)$.
	\end{théorème}
	\begin{proof}
		Tout d'abord on exprime la cohomologie de la strate $b$ en fonction de celle de la variété d'Igusa et de l'espace de Rapoport-Zink associé :
		\[
		\sum_{i} (-1)^i H^i_c(\overline{\Sh}(b), R\Psi_{\eta} (\mathcal{L}_{\xi})) = \sum_{t,s,r} (-1)^{t+s+r}  \mathop{\mathrm{lim}}_{\overrightarrow{K}} \Ext_{J_b(\Q_p)}^t (H_c^s(\mathcal{M}_{K}, \Q_{\ell}(D_K)),H_c^r(Ig(b) )).
		\]	
		
		Lorsque la variété de Shimura est compacte, cette formule est démontrée dans \cite{Man05} (théorème 13 et page 28) et \cite{Man11} (théorème 3.2). Lorsque la variété de Shimura n'est pas compacte, la formule est le théorème 6.26 de \cite{LS}.
		
		La deuxième étape consiste à montrer que les représentations
		\[
		\mathop{\mathrm{lim}}_{\overrightarrow{K}} \Ext_{J_b(\Q_p)}^t (H_c^s(\mathcal{M}_{K}, \Q_{\ell}(D_K)),H_c^r(Ig(b) ))
		\]
		sont des induites paraboliques (\cite[Coro. 7.1]{Shen}  et \cite[théo. 41]{Man08}).	
	\end{proof}
	\begin{definition}
		Pour $H$ un groupe réductif $p$-adique et $V$ un $H$-module. On définit la partie supercuspidale de $V$ par la formule $ V_{cusp} = \bigoplus_{e} e \cdot V $ où $e$ parcourt les idempotents du centre de Bernstein de $H$ associés aux classes d'équivalences inertielles des représentations supercuspidales de $H$.
	\end{definition}
	Dans la suite, on notera $ H^i_c(\overline{\Sh}(b_0), R\Psi_{\eta} (\mathcal{L}_{\xi}))_{p-cusp}$ et $H^i_c(\Sh, \mathcal{L}_{\xi})_{p-cusp} $ pour les parties $G(\Q_p)$-supercuspidales de la cohomologie. Le corollaire suivant est une conséquence du théorème \ref{itm : induction parabolique} et du théorème \ref{itm : somme des strates}. 
	\begin{corollaire} \phantomsection \label{itm:cusp}
		On a des isomorphismes :
		\[
		H^i_c(\overline{\Sh}(b_0), R\Psi_{\eta}(\mathcal{L}_{\xi}))_{p-cusp} \simeq H^i_c(\Sh, \mathcal{L}_{\xi})_{p-cusp}
		\]	  
		où $b_0$ désigne la classe basique.
	\end{corollaire}
	
	\subsection{Uniformisation rigide} 
	
	Soit $\overline{S}_{K^p} = S_{K^p}\times_{\mathop{\mathrm{Spec}}\nolimits(\mathcal{O}_{E_p})} \F_{p^2}$ la fibre spéciale de $S_{K^p}$. Le schéma $\overline{S}_{K^p}$ est alors stratifié par le polygone de Newton du cristal muni de structures additionnelles. Plus précisément, pour chaque point $ x = (A_0, \overline{\lambda}_0, \iota_0, \overline{\eta}_0)$ dans $\overline{S}_{K^p}(\overline{\F}_p)$, on note $(X_x, \iota, \lambda)$ le groupe $p$-divisible muni d'une action de $\mathcal{O}_{E}$ et d'une polarisation $p$-principale, associé à $x$ ainsi que $(N_x, \iota, \langle \cdot | \cdot \rangle)$ l'isocristal associé. Si $b \in B(G_{\Q_p}, \mu_{\Q_p})$ on pose alors: 
	$$ \overline{S}_{K^p}(b) = \{ x \in \overline{S}_{K^p}(\overline{\F}_p) \vert (N_x, \iota, \langle \cdot | \cdot \rangle) \simeq (V_{\Q_p}\otimes_{\Q_p} \breve{\Q}_p, b\otimes \sigma, \langle \cdot | \cdot \rangle ) \} $$
	
	On a la stratification: $$ \overline{S}_{K^p} = \displaystyle \mathop{\mathrm{\coprod}}_{b \in B(G_{\Q_p}, \mu_{\Q_p}) } \overline{S}_{K^p}(b) $$ et si $y \in B(G_{\Q_p}, \mu_{\Q_p})$ est fixé alors
	$ \displaystyle \mathop{\mathrm{\coprod}}_{ b \prec y } \overline{S}_{K^p}(b) $ \ est un fermé dans $\overline{S}_{K^p}$.
	
	Soient maintenant $y = (A_y, \overline{\lambda}_y, \iota_y, \overline{\eta}_y ) $ un point géométrique dans la strate basique (cf \ref{itm: basique}) ainsi que $(\mathbb{X}, \iota, \lambda)$ le groupe $p$-divisible muni d'une action de $\mathcal{O}_{E}$ et d'une polarisation $p$-principale associé.
	
	Nous notons $\mathcal{M}$ le problème de module associé au groupe $p$-divisible $(\mathbb{X}, \iota, \lambda)$. Par définition, le groupe $J(\Q_p) := J_{b_y}(\Q_p)$ agit à gauche sur $\mathcal{M}$.
	
	On pose $\phi$ la classe d'isogénie du triplet $(A_y, \overline{\lambda}_y, \iota_y)$ et \ $ I^{\phi} = \Aut(A_y, \lambda_y, \iota_y) $ \ le groupe réductif sur $\Q$ associé. Le groupe $I^{\phi}(\Q)$ agit par quasi-isogénies sur le groupe $p$-divisible $(A_y[p^{\infty}], \lambda, \iota)$, ce qui donne une injection $ I^{\phi}(\Q) \hookrightarrow J(\Q_p) $, en particulier, $I^{\phi}(\Q)$ agit sur $\mathcal{M}$. De plus, l'action de $I^{\phi}(\Q)$ sur le module de Tate $H_1(A_y, \A_f^p)$ donne une injection $ I^{\phi}(\Q) \hookrightarrow G(\A_f^p) $. 
	
	A $\phi$ est associé l'ensemble
	$
	\tilde{S}(\phi)(\overline{\F}_p) = \big\{ z \in \overline{S}(b)(\overline{\F}_p) \vert \ \mathop{\mathrm{la \ classe \ d'isogénie \ de \ }}\nolimits (A_z, \lambda_z, \iota_z) \in \phi \big\}
	$. On peut munir $ \tilde{S}(\phi)(\overline{\F}_p) $ d'une structure de sous-schéma fermé réduit.
	
	\begin{théorème}(6.23 de \cite{RZ96})
		Il y a une uniformisation de schémas formels sur $\Spf (\mathcal{O}_{\breve{\Q}_p})$
		$$ I^{\phi}(\Q) \setminus ( \mathcal{M} \times G(\A_f^p) / K^p ) \quad \xrightarrow{\sim} \quad (S_{K^p} \otimes_{\mathcal{O}_{E_p}} \mathcal{O}_{\breve{\Q}_p} ) \widehat{/}_{\tilde{S}(\phi)}  $$
		
		Lorsque $K^p$ varie les différents isomorphismes d'uniformisation sont compatibles et commutent à l'action de $G(\A_f)$.	
	\end{théorème}
	
	En suivant \cite{Far04}, on travaille avec les espaces rigides dès qu'on n'est plus en niveau maximal. Pour chaque $K^p \subset G(\A_f^p)$ on note:
	\begin{enumerate}
		\item[-] $S_{K^p}^{\wedge}$ le complété $p$-adique du schéma $S_{K^p} \otimes_{\mathcal{O}_{E_p}} \mathcal{O}_{\breve{\Q}_p} $ ainsi que $(S^{\wedge}_{K^p})^{an}$ l'espace analytique associé.
		\item[-] $\Sh^{an}_{C_0K^p}$ l'espace analytique sur $\breve{\Q}_p$ associé à la variété algébrique $\Sh_{C_0K^p} \otimes \breve{\Q}_p$.
	\end{enumerate}
	\begin{definition}
		Nous noterons
		$ \Sh_{C_0K^p}^{an}(\phi) = ((S_{K^p})^{\wedge}_{/ \tilde{S}(\phi)})^{an} $
		la fibre  générique du schéma formel complété de $S$ le long de $\tilde{S}(\phi)$.
		
		Si $K_p \subset C_0, \ K = K_pK^p$ nous noterons
		$ \Sh_{K}^{an}(\phi) =  \Theta^{-1}_{C_0K^p, K}(\Sh_{C_0K^p}^{an}(\phi)) $
		(un ouvert analytique de $(\Sh_K)^{an}$) où $\Theta_{C_0K^p, K}$ est le morphisme de changement de niveau.
	\end{definition} 
	
	
	\begin{théorème}
		Pour $K = K_p K^p$ variable il y a des isomorphismes compatibles d'espaces analytiques sur $\breve{\Q}_p$:
		$$ I^{\phi}(\Q) \setminus (\mathcal{M}_{K_p} \times G(\A_f^p)/ K^p ) \xrightarrow{\sim} \Sh_{K}^{an}(\phi). $$
		
		Remarquons que la variété de Shimura n'est pas nécessairement compacte.
	\end{théorème}
	
	
	Soit comme précédemment $b_0$ la classe basique dans $B(G_{\Q_p}, \mu_{\overline{\Q_p}})$. Pour chaque classe d'isogénie $\phi = $ du triplet $(A_y, \overline{\lambda}_y, \iota_y)$, on note $b(\phi)$ l'unique élément dans $B(G_{\Q_p}, \mu_{\Q_p})$ correspondant à la classe d'isogénie du groupe $p$-divisible avec structures additionnelles de $(A_y, \overline{\lambda}_y, \iota_y)$. Rappelons les faits suivants (cf \cite{RZ96}, $(6.34)$): 
	\begin{enumerate}
		\item[-] l'ensemble $\{ \phi \ \vert \ b(\phi) = b_0 \}$ est fini,
		\item[-] $\forall \phi $ tel que $ \ b(\phi) = b_0 $, $I^{\phi}$ est une forme intérieure de $G$, plus précisément 
		\begin{enumerate}
			\item[-] $I^{\phi}(\Q_p) = J_b(\Q_p)$ et $\forall l \neq p $ on a $I^{\phi}(\Q_{\ell}) = G(\Q_{\ell})$,
			\item[-] $I^{\phi}(\R)$ est la forme intérieure compacte modulo le centre de $G(\R)$. 
		\end{enumerate}
	\end{enumerate} 
	
	Notons $\mathcal{A}(I^{\phi})$ l'espace des représentations automorphes de $I^{\phi}$ (on tient compte des multiplicités) ainsi que $\mathcal{A}_{\xi}^{\phi}$ l'espace des formes automorphes sur $I^{\phi}$ de type $\breve{\xi}'$ à l'infini au sens où 
	$$ \mathcal{A}_{\xi}^{\phi} = \hom_{I^{\phi}(\R)} (\breve{\xi}' ,\mathcal{A}(I^{\phi})) $$
	où $ \xi' : I^{\phi}(\R) \hookrightarrow I^{\phi}(\C) = G(\C) \xrightarrow{\xi} GL(V). $
	
	Supposons maintenant que $\xi$ est irréductible. Le groupe de Lie $I^{\phi}(\R)$ étant anisotrope modulo son centre, pour un sous-groupe compact ouvert $K^p$ de $G(\A_f^p)$ on a
	$$ (\mathcal{A}_{\xi}^{\phi})^{K^p} = \bigoplus_{\substack{\Pi \in \mathcal{A}(I^{\phi}) \\ \Pi_{\infty} = \breve{\xi}}} \Pi_p \otimes (\Pi^p)^{K^p} $$
	
	\begin{théorème} \cite{Far04} \phantomsection \label{itm : suite spectrale}
		Il y a une suite spectrale $G(\A_f) \times W_{E_p}$-équivariante
		$$ E_2^{tq} = | \ker^1 (\Q, G) | \sum_{\substack{\Pi \in \mathcal{A}(I^{\phi}) \\ \Pi_{\infty} = \breve{\xi}}} \left( \mathop{\mathrm{lim}}_{\overrightarrow{K}} \Ext^t_{\J_{b}(\Q_p)} \left( H^q_c(\mathcal{M}_{K_p}, \overline{\Q}_{\ell}(D)), \Pi_p \right)  \right)_{p-cusp} \otimes (\Pi^p) $$ 
		dont l'aboutissement est $\left( H^{t+q}_c(\Sh, \mathcal{L}_{\xi}) \right)_{p-cusp}$ et où $D$ est la dimension de la variété de Shimura. 
	\end{théorème}
	\begin{proof}
		Tout d'abord on a un isomorphisme pour tous $t$ et $q$ (\cite{Far04} page $75$)
		\[
		\Ext^t_{\J_{b}(\Q_p)} (H^q_c(\mathcal{M}_{K_p}, \overline{\Q}_{\ell}), (\mathcal{A}_{\xi}^{\phi})^{K^p}) \simeq \sum_{\substack{\Pi \in \mathcal{A}(I^{\phi}) \\ \Pi_{\infty} = \breve{\xi}}} \Ext^t_{\J_{b}(\Q_p)} (H^q_c(\mathcal{M}_{K_p}, \overline{\Q}_{\ell}), \Pi_p) \otimes (\Pi^p)^{K^p}.
		\]
		
		D'après le corollaire $4.3.15$ de \cite{Far04}, on a la suite spectrale $G(\A_f) \times W_{E_p}$-équivariante
		\[
		E_2^{tq} = | \ker^1 (\Q, G) | \sum_{\substack{\Pi \in \mathcal{A}(I^{\phi}) \\ \Pi_{\infty} = \breve{\xi}}} \left( \mathop{\mathrm{lim}}_{\overrightarrow{K}} \Ext^t_{\J_{b}(\Q_p)} \left( H^q_c(\mathcal{M}_{K_p}, \overline{\Q}_{\ell}(D)), \Pi_p \right) \right) \otimes (\Pi^p) 
		\]
		dont l'aboutissement est $\displaystyle \mathop{\mathrm{lim}}_{\overrightarrow{K}} H^{t+q}(\Sh_K^{an}(b_0), \mathcal{L}_{\xi}^{an})$.
		
		D'autre part la strate basique est propre, en utilisant la proposition $6.9.4$ de \cite{Far04} on a une égalité de groupes de cohomologie
		\[
		\mathop{\mathrm{lim}}_{\overrightarrow{K}} H^{t+q}(\Sh_K^{an}(b_0), \mathcal{L}_{\xi}^{an}) = \mathop{\mathrm{lim}}_{\overrightarrow{K}} H^{t+q}_c(\overline{S}_K(b_0), R\Psi(\mathcal{L}_{\xi})) = H^{t+q}_c(\Sh(b_0), \mathcal{L}_{\xi}).
		\]
		
		Or le corollaire \ref{itm:cusp} donne une égalité $ H^{t+q}_c(\Sh(b_0), \mathcal{L}_{\xi})_{p-cusp} = H^{t+q}_c(\Sh, \mathcal{L}_{\xi})_{p-cusp} $. On en déduit alors la suite spectrale voulue.
	\end{proof}
	
	\section{Classification des représentations automorphes pour les groupes unitaires}
	Le but de cette section est de rappeler, dans le cas des groupes unitaires, les formules de multiplicité locales et globales dans \cite{KMSW} et \cite{Mok}, c.f. les théorèmes \ref{itm: local} et \ref{itm: global}. Nous nous intéresserons aux paquets cuspidaux lesquels sont classifiés dans \cite{Moe}. 
	Notons que pour passer des groupes unitaires aux groupes de similitudes unitaires, on devra supposer que $n$ est impair. Le résultat nouveau de cette section est la proposition \ref{itm : globaliser auto} qui sera utilisée dans la preuve du théorème principal dans la section \ref{itm: cohomologie}.

	\subsection{Groupes unitaires et leurs formes intérieures}
	Soit $F$ un corps local ou global de caractéristique $0$ et $\overline{F}$ une clôture algébrique. On note $\Gamma$ le groupe de Galois de $\overline{F} / F$. On se donne un groupe réductif $G$ défini sur $F$.
	
	\textbf{Une forme intérieure} de $G$ est un groupe réductif $G_1$ défini sur $F$ et muni d'un isomorphisme $ \varrho : G \times \overline{F} \longrightarrow G_1 \times \overline{F} $ tel que pour tout $\sigma \in \Gamma$, l'automorphisme $ \varrho^{-1} \sigma (\varrho) = \varrho^{-1} \circ \sigma \circ \varrho \circ \sigma^{-1} $ est intérieur. 
	L'application $ \varrho \longmapsto \varrho^{-1}\sigma(\varrho) $ établit un isomorphisme entre l'ensemble des classes d'isomorphismes des formes intérieures de $G$ avec $H^1(\Gamma, G_{\text{ad}})$ où $G_{ad} = G/Z$ et $Z$ est le centre de $G$.

	Soit $E/F$ une extension quadratique de corps, notons $U_{E/F}(n)$ le groupe unitaire quasi-déployé en $n$ variables associé. On a :
	\begin{enumerate}
		\item[-] Lorsque $F$ est un corps $p$-adique alors
		$ H^{1}(\Gamma, U_{E/F}(n)_{ad}) = \Z / \delta \Z $ où $\delta = 1$ si $n$ est impair et $\delta = 2$ si $n$ est pair
		\item[-] Lorsque $F = \R$ alors
		$ H^{1}(\Gamma, U_{E/F}(n)_{ad}) = \Bigl\{ \{p,q\} | 0 \leq p,q \leq n, \ p+q = n \Bigr\} $
	\end{enumerate}
	
	\textbf{Une forme intérieure pure} est un couple $(\varrho, z) : G \longrightarrow G_1 $ où $\varrho : G \longrightarrow G_1$ est une forme intérieure et $z$ est un cocycle dans $Z^1(\Gamma, G)$ tels que $ \varrho^{-1}\sigma(\varrho) = \text{Ad}(z_{\sigma})$. 
	L'application $(\varrho, z) \longmapsto z $ établit alors une bijection entre l'ensemble des classes d'isomorphismes de formes intérieures pures et $ H^1(\Gamma, G)$.

	On aura besoin de la notion de \textbf{forme intérieure étendue}. La définition étant technique, on donne seulement une description explicite dans le cas qui nous intéresse. En fait, dans \cite{Kot14}, Kottwitz a construit un ensemble $B(F, G)$ pour $F$ un corps local ou global ainsi qu'un sous-ensemble $B(F, G)_{bs}$ contenant les éléments basiques. Les classes d'isomorphisme des \textbf{formes intérieures étendues} de $G$ sont en bijection avec $B(F,G)_{bs}$ (\cite{KMSW} page 16, ligne 23). Par la suite, on essaie de caractériser $B(F,G)_{bs}$.
	
	Pour $F$ un corps local il y a une application canonique :
	\begin{equation} \phantomsection \label{itm: kappa}
		\kappa_G : B(F, G)_{bs} \longrightarrow X^{*}(Z(\widehat{G})^{\Gamma})
	\end{equation}
	qui est une bijection si $F$ est un corps $p$-adique.
	
	Lorsque $F$ est un corps global, pour chaque place $v$ de $\overline{F}$, il y a un morphisme de localisation $ B(F, G) \longrightarrow B(F_v, G)$ qui préserve le caractère d'être basique.
	
	Si l'on choisit une place de $\overline{F}$ sur chaque place de $F$, les morphismes de localisation donnent un morphisme :
	\begin{equation} \phantomsection \label{itm: etendue ker}
		B(F, G)_{bs} \longrightarrow \prod_{v} B(F_v, G)_{bs}
	\end{equation}
	où $ \prod $ désigne l'ensemble des éléments dans le produit direct dont les composantes à presque toutes les places $v$ sont égales à l'élément neutre de $B(F_v, G)_{bs}$. On peut montrer que le noyau de (\ref{itm: etendue ker}) est en bijection avec $ \ker^1 (F, G)$. De plus, l'image de (\ref{itm: etendue ker}) est égale au noyau de la composition suivante :
	\begin{equation} \phantomsection \label{itm : etendue image }
		\prod_{v}  B(F_v, G)_{bs} \xrightarrow{(\ref{itm: kappa})} \bigoplus_v X^{*}(Z(\widehat{G})^{\Gamma_v}) \xrightarrow{\sum} X^{*}(Z(\widehat{G})^{\Gamma}). 
	\end{equation}

	\begin{exemple} \phantomsection \label{itm : étendre} (\cite{KMSW} sec. 0.3.3) \textbf{Formes intérieures étendues de groupes unitaires}
		
		Pour $U_{E/F}(n)$ un groupe unitaire en $n$ variables on a : $ X^{*}(Z(\widehat{U_{E/F}(n)})^{\Gamma}) \simeq \Z / 2\Z $.
		
		On suppose tout d'abord que $F$ est un corps local et dans ce cas-là on a une bijection  
		\[
		H^1(\Gamma, U_{E/F}(n)) \longrightarrow B(F, U_{E/F}(n))_{bs}. 
		\]
		\begin{enumerate}
			\item[$\bullet$] Lorsque $F$ est $p$-adique on a :
			\[
			H^1(\Gamma, U_{E/F}(n)) \simeq B(F, U_{E/F}(n))_{bs} \simeq \Z / 2\Z
			\]
			\item[$\bullet$] Lorsque $F = \R$, on a $ H^{1}(\Gamma, U_{\C/ \R}(n)) = \{ (p,q) | 0 \leq p,q \leq n, \ p+q = n \} $ et $ H^{1}(\Gamma, U_{\C / \R}(n)_{ad}) $ est le quotient de cet ensemble par la relation $(p, q) \sim (q, p)$. Le morphisme (\ref{itm: kappa}) est donné par $ H^{1}(\Gamma, U_{\C/ \R}(n)) \longrightarrow \Z / 2\Z $, $(p, q) \longmapsto ([\frac{n}{2}] + q) \ \text{mod} \ 2$.
		\end{enumerate}
		
		De plus pour le groupe linéaire général, l'application (\ref{itm: kappa}) devient un isomorphisme 
		\[
		B(F, GL_n)_{bs} \xrightarrow{\sim} \Z.
		\]
		
		Pour $E$ un corps de nombres CM dont $F$ est le sous-corps totalement réel, on peut donc montrer que (\ref{itm: etendue ker}) est injectif (\cite{KMSW}, page 18 ligne 2). L'ensemble $B(F, U_{E/F}(n))_{bs}$ s'identifie avec le noyau de (\ref{itm : etendue image }). D'après la discussion pour les groupes linéaires et unitaires locaux, on a la description suivante. 
		
		Pour chaque place $v$ de $F$, on note $\Xi_v$ la classe d'isomorphisme des formes intérieures étendues de $U_{E/F}(n)$ dont $a_v$ est l'invariant. On a alors
		$a_v \in \Z$  si $v$ est finie et décomposée ;
		$a_v \in \Z / 2 \Z$ si $v$ est finie et inerte ;
		$a_v \in \Z / 2 \Z$ si $v$ est réelle.
		\begin{proposition} \phantomsection \label{itm : forme intérieure pure} (\cite{KMSW} section 0.3.3)
			La collection $(\Xi_v)_v$ est dans l'image de (\ref{itm: etendue ker}), i.e la localisation d'une forme intérieure étendue globale si et seulement si $a_v = 0$ pour presque tout $v$ et la somme des images de $a_v$ mod $2$ est égale à $0$ mod $2$.
		\end{proposition}	
	\end{exemple}

	\subsection{Formalisme des paramètres} \label{itm: centra}
	On commence par considérer le cas d'un corps local $F$. Considérons le groupe de Langlands local $L_F := W_F$ si $F$ est archimédien et $W_F \times SU(2)$ dans le cas non archimédien. On pose également $ \prescript{L}{}{G} = \widehat{G} \rtimes W_F $ comme un groupe topologique où $\widehat{G}$ est le groupe dual de Langlands de $G$.
	\begin{definition}
		Un $L$-paramètre local pour un groupe réductif connexe $G$ défini sur $F$, est un morphisme continu $\phi : L_F \longrightarrow \prescript{L}{}{G}$ qui commute avec les projections canoniques de $L_F$ et $\prescript{L}{}{G}$ sur $W_F$ et tel que $\phi$ envoie les éléments semisimples sur des éléments semisimples.
	\end{definition}
	Deux $L$-paramètres sont équivalents s'ils sont conjugués par un élément de $\widehat{G}$. On note $\Phi(G)$ l'ensemble des classes d'équivalences de $L$-paramètres.
	\begin{notation}
		Soit $\phi \in \Phi(G)$ un $L$-paramètre.
		\begin{enumerate}
			\item[-] $\phi$ est borné si son image dans $^{L}G$ se projette dans un sous-ensemble relativement compact de $\widehat{G}$.
			\item[-] $ \phi $ est discret (ou de carré intégrable) si son image n'est contenue dans aucun sous-groupe parabolique propre de $^{L}G$.
		\end{enumerate}
	\end{notation} 
	On note $\Phi_{\text{bdd}}(G)$ \big(resp. $\Phi_2(G)$\big) le sous-ensemble des $L$-paramètres bornés (resp. discrets). On considérera aussi l'ensemble $\Phi_{\text{2, bdd}} := \Phi_{\text{bdd}}(G) \cap \Phi_2(G)$. On note  $\Pi_{\text{temp}}(G)$ (resp. $\Pi_2(G)$) l'ensemble des représentations tempérées (resp. essentiellement de carré intégrable) de $G(F)$. De manière similaire on pose $ \Pi_{2, \text{temp}} := \Pi_2 (G) \cap \Pi_{\text{temp}} (G)$ l'ensemble des représentations de carré intégrable.
	
	On aura besoin de la notion de $A$-paramètres qui joueront le rôle des composantes locales dans la classification globale.
	
	\begin{definition}
		Un $A$-paramètre local pour un groupe réductif connexe $G$ défini sur $F$ est un morphisme continu $\psi : L_F \times SU(2) \longrightarrow \prescript{L}{}{G}$ tel que $\psi |_{L_F}$ est un $L$-paramètre borné.  
	\end{definition}   
	Comme pour les $L$-paramètres, la condition d'équivalence entre les $A$-paramètres est définie par $\widehat{G}$-conjugaison. On note $\Psi(G)$ l'ensemble des classes d'équivalences de $A$-paramètres. On note également $\Psi^{+}(G)$ l'ensemble des classes d'équivalences de morphismes continus $\psi$ comme ci-avant mais où $ \psi |_{L_F} $ n'est pas nécessairement borné. Un $A$-paramètre $\psi$ (ou $\psi \in \Psi^{+}(G)$) est générique si $ \psi |_{SU(2)}$ est trivial. 
	
	Donnons une description plus en détails de ces notions pour $G = GL(n)$. Afin d'alléger les notations, on notera $\Phi (n) := \Phi(GL(n))$, $\Pi (n) := \Pi(GL(n))$ et de même pour les autres ensembles de représentations. On posera $\Phi_{\text{sim}}(n) := \Phi_{2}(n)$. La correspondance de Langlands locale peut s'écrire de manière informelle comme suit.
	\begin{théorème} \cite{Lang} \cite{HT01}, \cite{Hen00}
		Il y a une unique bijection <<arithmétique>> $ \phi \longmapsto \pi $ de $\Phi(n)$ dans $\Pi (n)$ dite de Langlands locale qui respecte en particulier les sous-ensembles suivants
		\[
		\Phi_{\text{sim,bdd}} (n) \quad \subset \quad \Phi_{\text{bdd}} (n) \quad \subset \quad \Phi (n)
		\] 
		\[
		\Pi_{2, \text{temp}} (n) \quad \subset \quad \Pi_{\text{temp}}(n) \quad \subset \quad \Pi (n)
		\]
	\end{théorème}   
	\begin{remarque}
		<<Arithmétique>> signifie respecter facteurs $L$ et $\epsilon$ de paires. 
	\end{remarque}
	Pour les groupes unitaires, les $L$-paramètres et $A$-paramètres dans le cas quasi-déployé $U_{E/F} (n)$ sont reliés à ceux du groupe $GL(n)$ via un morphisme de changement de base. 
	
	Pour chaque $\kappa \in \{ \pm 1 \}$, on peut définir un ensemble $ \mathcal{Z}_E^{\kappa} $ (cf \cite{Mok} p.$8$, (2.1.5), (2.1.6)). Choisissons $\chi_{\kappa} \in \mathcal{Z}_E^{\kappa} $, on a un morphisme de changement de base (cf \cite{Mok} p.$9$) : 
	\begin{equation} \phantomsection \label{itm: changement de base}
		\eta_{\chi_{\kappa}} : \prescript{L}{}{U_{E/F}(n)} \longrightarrow  \prescript{L}{}{G_{E/F}(n)}	
	\end{equation}
	où $G_{E/F}(n) = \Res_{E/F}(GL_E(n))$.
	
	On obtient donc une application :
	\begin{align*} 
		\eta_{\chi_{\kappa}, *} : \Phi (U_{E/F}(n)) &\longrightarrow \Phi(G_{E/F}(n)) = \Phi_E (n) \\
		\phi & \longmapsto  \eta_{\chi_{\kappa}} \circ \phi
	\end{align*}
	et de même pour les $A$-paramètres
	\begin{align*}
		\eta_{\chi_{\kappa}, *} :  \Psi (U_{E/F}(n)) & \longrightarrow  \Psi(G_{E/F}(n)) = \Psi_E (n) \\
		\psi & \longmapsto  \eta_{\chi_{\kappa}} \circ \psi.
	\end{align*}
	
	Pour $\phi \in \Phi (U_{E/F}(n))$, le $L$-paramètre dans $\Phi_E (n)$ qui correspond à $\eta_{\chi_{\kappa}} \circ \phi$, est $\phi_{| L_E} \otimes \chi_{\kappa} $. En particulier, si $\kappa = 1$ et $\chi_{\kappa} = 1$ alors $\eta_{\chi_{\kappa}} \circ \phi$ est juste la restriction de $\phi$ sur $L_E$.
	
	Les applications $\eta_{\chi_{\kappa}, *}$ ci-dessus sont injectives et on cherche à décrire leurs images. Soit $c$ un morphisme dans $W_F$ relevant le morphisme non trivial de $\text{Gal(E/F)}$. Il y a une inclusion $ W_F \hookrightarrow L_F $ et on peut faire agir $c$ sur $L_E$ par conjugaison. On définit $\rho^c(g) := \rho (cgc^{-1})$ et de plus on pose $ \rho^* := (\rho^c)^{\vee} $ pour toute représentation $ \rho : L_E \longrightarrow GL(n, \C) $.
	\begin{definition}
		Un paramètre $\psi \in \Psi_E(n)$ est appelé conjugué auto-dual si $\psi = \psi^*$. Les ensembles des paramètres conjugués auto-duaux dans $\Phi_E(n)$ et $\Psi_E(n)$ sont notés $\widetilde{\Phi}_E(n)$ et $\widetilde{\Psi}_E(n)$ respectivement. 
	\end{definition} 
	Désormais nous notons simplement $\Phi(n)$ au lieu de $\Phi_E(n)$ (et similairement pour les autres ensembles de paramètres) lorsque le contexte est clair.
	
	\begin{definition}
		Un paramètre auto-dual $\rho$ est de parité $\eta$ où $\eta = \pm 1$ s'il existe une forme bilinéaire non dégénérée $B \langle \cdot | \cdot \rangle$ sur $V$ de sorte que $B\langle \rho^c(g)x, \rho(g)y \rangle = B \langle x, y \rangle $ et que $B \langle x,y \rangle = \eta B \langle y, \rho(c^2)x \rangle $.
	\end{definition}
	Il est clair que les images des applications $\eta_{\chi_{\kappa}, *}$ consistent en des paramètres conjugués auto-duaux. 
	
	\begin{proposition} (\cite{KMSW} lemme 1.2.5) \phantomsection \label{itm : changement de base}
		Les applications $\eta_{\chi_{\kappa}, *}$ induisent une bijection entre la pré image de $ \Phi_{\text{sim}}(n) $ (resp. de $\Psi_{\text{sim}} (n)$) et l'ensemble des paramètres conjugués auto-duaux avec parité $ (-1)^{n-1} \kappa $ dans $ \Phi_{\text{sim}} (n) $ (resp. dans $\Psi_{\text{sim}} (n)$).
	\end{proposition} 
	\begin{notation}
		On pose $ \Phi_{\text{sim}}(U_{E/F}(n)) := \eta_{\chi_{\kappa} *}^{-1} (\Phi_{\text{sim}}(n)) $ \Big(resp.$ \Psi_{\text{sim}}(U_{E/F}(n)) := \eta_{\chi_{\kappa} *}^{-1} (\Psi_{\text{sim}}(n)) $\Big), l'ensemble des $L$-paramètres simples \big(resp. l'ensemble des $A$-paramètres simples\big).
	\end{notation}
	Pour chaque $A$-paramètre $\psi \in \Psi^{+}(G)$ on définit des groupes centralisateurs comme ci-dessous, qui jouent un rôle important dans la classification locale et globale:
	\[
	S_{\psi} := \text{Cent}(\text{Im}\psi, \widehat{G}), \quad \overline{S}_{\psi} := S_{\psi} / Z(\widehat{G})^{\Gamma}, \quad \mathcal{S}_{\psi} := \pi_0 (S_{\psi}),
	\]  
	\[
	\overline{\mathcal{S}}_{\psi} := \pi_0 (\overline{S}_{\psi}), \quad S_{\psi}^{\text{rad}} := (S_{\psi} \cap \widehat{G}_{\text{der}})^0, \quad S_{\psi}^{\natural} := S_{\psi} / S_{\psi}^{\text{rad}}.
	\]
	
	On va maintenant donner une description explicite du centralisateur $S_{\psi} $ pour $\psi \in \Psi^{+}(U_{E/F}(n))$. Posons $ \psi^n := \eta_{\chi_{\kappa}, *} (\psi) $ qui s'écrit sous la forme
	\[
	\psi^n = \bigoplus_{k \in K} \ell_k \psi_k^{n_k}
	\]
	où les $ \psi_k^{n_k} $ sont des représentations irréductibles deux à deux non isomorphes de $ L_E $. Puisque $ \psi^n $ est auto-dual, il y a une involution $ * $ de $K$ telle que $ \psi_{k^*}^{n_{k^*}} = (\psi_k^{n_k})^* $. Alors $ \psi^n $ s'écrit sous la forme :
	\[
	\psi^n = \left( \bigoplus_{i \in I_{\psi^n}} \ell_i \psi_i^{n_i} \right) \oplus \left( \bigoplus_{j \in J_{\psi^n} } \ell_j(\psi_j^{n_j} \oplus (\psi_{j} ^{n_{j}})^* ) \right).
	\]
	où $ I_{\psi^n} $ est l'ensemble des points fixes dans $K$ de l'involution $*$. On voit alors que : 
	\[
	S_{\psi} = \prod_{i \in I^{+}_{\psi^n}} O(\ell_i, \C) \times \prod_{i \in I^{-}_{\psi^n}} Sp( \ell_i, \C ) \times \prod_{j \in J_{\psi^n}} GL (\ell_j, \C)
	\]
	où $ i \in I_{\psi^n} $ appartient à $I^{+}_{\psi^n}$ si la parité de ${\psi_i}$ est égale à $\kappa (-1)^{n-1}$ et à $I^{-}_{\psi^n}$ sinon.
	
	D'autre part le groupe $Z(\widehat{G})^{\Gamma} = \{ \pm 1 \}$ est envoyé diagonalement dans le membre de droite ci-dessus de sorte que: 
	\[
	\mathcal{S}_{\psi} \simeq S_{\psi}^{\natural} \simeq (\Z / 2 \Z)^{| I^{+}_{\psi^n} |} \quad \text{et} \quad \overline{\mathcal{S}}_{\psi} \simeq \left\lbrace \begin{array}{cc}
		(\Z / 2 \Z)^{| I^{+}_{\psi^n} |} & \quad \text{si} \ \forall i \in I^{+}_{\psi^n} \ 2 | \ell_i, \\
		(\Z / 2 \Z)^{| I^{+}_{\psi^n} | - 1} & \quad \text{sinon}.
	\end{array} \right.
	\]
	
	
	Dans le cas où $F = \R$, $E = \C$, $L_{\C} = W_{\C} = \C^{\times}$, rappelons la description des paramètres discrets. Comme précédemment on a un morphisme de changement de base 
	\begin{equation}
		\eta_{\chi_{\kappa}} : \prescript{L}{}{U}_{\C / \R} (n) \longrightarrow \prescript{L}{}{G}_{\C / \R} (n)
	\end{equation}
	via lequel on associe à $\phi \in \Phi(U_{\C / \R}(n))$ un paramètre $\phi^n = \eta_{\chi_{\kappa} *} \phi : L_{\C} \longrightarrow GL_n(\C)$.
	
	Soit $ \phi \in \Phi_2(U_{\C / \R}(n)) $ alors $\phi^n$ s'écrit sous la forme 
	$
	\phi^n = \eta_1 \oplus \cdots \oplus \eta_n
	$
	où les $\eta_i$ sont des caractères auto-duaux deux à deux disjoints de $\C^{\times}$. En général, un tel caractère est de la forme $\eta : z \longmapsto (z / \overline{z})^a$ avec $a \in \frac{1}{2} \Z$. Si $\eta_i(z) = (z/ \overline{z})^{a_i} $ alors on introduit le $n$-tuple 
	$
	\mu_{\phi^n} := (a_1, \cdots a_n)
	$
	appelé le caractère infinitésimal de $\phi^n$. Si $ \chi_{\kappa}(z) = (z / \overline{z})^c $ alors le caractère infinitésimal de $\phi$ est donné par la formule
	$
	\mu_{\phi} := (b_1, \cdots b_n)
	$
	où $b_i = a_i - c$.
	\begin{definition} \phantomsection \label{itm: suffisament régulier}
		Notons $ d(\mu_{\phi}) = \text{min} \left\lbrace \mathop{\mathrm{min}}_i(b_i), \mathop{\mathrm{min}}_{i \neq j}(| b_i - b_j |) \right\rbrace $ et de même pour $\mu_{\phi^n}$. Le caractère infinitésimal de $\phi$ est suffisamment régulier si $ d(\mu_{\phi}) > 0 $.
	\end{definition} 
	Supposons désormais que $F$ est \textbf{un corps de nombres} et $E$ une extension quadratique de $F$. On va définir de manière formelle les notions de $L$-paramètre pour le groupe unitaire quasi déployé $ U_{E/F}(n)$ en partant des représentations cuspidales des groupes linéaires généraux.
	
	On commence par définir $\Psi_{\text{glb,sim}}(n)$ comme l'ensemble des produits tensoriels formels $ \pi \boxtimes \nu $ où $\pi$ est une représentation automorphe cuspidale de $ GL_m(\mathbb{A}_E) $ et $\nu$ une représentation algébrique de $SL_2(\C)$ de dimension $d$ et on demande de plus que $n = m \cdot d$. Plus généralement on définit $\Psi_{\text{glb}}(n)$ comme l'ensemble des objets sous forme d'une somme directe formelle non ordonnée 
	\[
	\pi = l_1 (\pi_1 \boxtimes \nu_1) \boxplus \cdots \boxplus l_r (\pi_r \boxtimes \nu_r)
	\]
	où $l_i \geq 1 $ est un entier, les $\pi_i \boxtimes \nu_i \in \Psi_{\text{glb,sim}}(n_i)$ sont deux à deux distincts et $ n = l_1 \cdot n_1 + \cdots + l_r \cdot n_r$.  
	
	Un paramètre global sera dit \textit{générique} si pour chaque facteur simple $ \pi_i \boxtimes \nu_i $, le facteur $\nu_i$ est la représentation triviale de $SL_2(\C)$. On note $ \Phi_{\text{glb}} (n) $ l'ensemble des paramètres globaux génériques. 
	
	Soit $\pi$ une représentation automorphe cuspidale de $GL_m(\mathbb{A}_E)$, on pose $\pi^{*} := (\pi^c)^{\vee}$ où $\pi^c := \pi \circ c$ avec $c$ le morphisme de conjugaison de Galois de $E$ sur $F$ et $(\pi^c)^{\vee}$ est la contragrédiente de $\pi^c$. Plus généralement, si $ \pi = \pi_1 \boxtimes \nu_1 \in \Psi_{\text{glb,sim}} (n) $ alors on pose $ (\pi)^{*} := \pi_1^{*} \boxtimes \nu_1 $ et un $\pi \in \Psi_{\text{glb,sim}} (n)$ est dit autodual si $ \pi = (\pi)^{*} $.
	
	Maintenant si $ \pi = (l_1 \pi_1 \boxtimes \nu_1) \boxplus \cdots \boxplus (l_r \pi_r \boxtimes \nu_r) \in \Psi_{\text{glb}} (n) $, on dit que $\pi$ est autodual s'il existe une involution $ i \longmapsto i^{*} $ de $ \{ 1, \cdots , r \} $ telle que $(\pi_i \boxtimes \nu_i)^{*} = \pi_{i^{*}} \boxtimes \nu_{i^{*}}$ et $l_i = l_{i^*}$.
	
	L'ensemble des paramètres autoduaux (resp. autoduaux simples) de $\Psi_{\text{glb}}(n) $ (resp $ \Psi_{\text{glb, sim}}(n) $) sera noté   $\widetilde{\Psi}_{\text{glb}}(n)$ (resp $ \widetilde{\Psi}_{\text{glb,sim}}(n) $). 
	
	On a besoin de la construction suivante qui nous servira de substitution au groupe de Langlands global: pour la construction détaillée, voir \cite{Mok} (p.22-23) et \cite{KMSW} (p.68).
	\begin{Construction} \cite{Mok}, \cite{KMSW}
		Soit $ \psi \in \widetilde{\Psi}_{glb}(n) $ comme ci-dessus, il existe un groupe $\mathcal{L}_{\psi}$ et un morphisme\footnote{On rappelle que $G_{E/F}(n) = \Res_{E/F}(GL_E(n))$.} :
		\[
		\widetilde{\psi} : \mathcal{L}_{\psi} \times SL_2(\C) \longrightarrow  \prescript{L}{}{G_{E/F}(n)}. 
		\]
	\end{Construction} 
	\begin{definition}
		On définit l'ensemble des paramètres $ \Psi (U_{E/F}(n), \eta_{\chi_{\kappa}} ) $ comme l'ensemble des couples $(\psi^n, \widetilde{\psi})$ où $ \psi^n \in \widetilde{\Psi}_{\text{glb}}(n) $ et $\widetilde{\psi} : \mathcal{L}_{\psi^n} \times SL_2(\C) \longrightarrow \prescript{L}{}{(U_{E/F}(n))} $ est un morphisme tel que $\widetilde{\psi^n} = \eta_{\chi_{\kappa}} \circ \widetilde{\psi}$. Un paramètre $ \psi = (\psi^n, \widetilde{\psi})$ est générique si $\psi^n$ l'est.	   
	\end{definition}
	Notons que $\widetilde{\psi}$ est entièrement déterminé par $\eta_{\chi_{\kappa}}$ et $ \widetilde{\psi^n}$. 
	
	Comme auparavant on peut définir le centralisateur et ses variantes pour un paramètre $ \psi = (\psi^n, \widetilde{\psi}) $. On note $\epsilon_{\psi}$ le caractère d'Arthur de $\overline{\mathcal{S}}_{\psi}$, cf \cite{Arthur} p.$15$.
	
	Afin de décrire ces groupes en détails, on considère la décomposition de $ \psi^n $ sous la forme
	\begin{equation} \phantomsection \label{itm: forme normale}
		\psi^n = \Big(\op_{i \in I^{+}_{\psi^n}} l_i \psi^{n_i}_i \Big) \boxplus \Big(\op_{i \in I^{-}_{\psi^n}} l_i \psi^{n_i}_i \Big) \boxplus \Big( \op_{j \in J_{\psi^n} } l_j(\psi^{n_j}_j \boxplus \psi^{n_{j^*}}_{j^{*}} ) \Big)
	\end{equation}
	où $i \in I_{\psi^n}$ appartient à $I^{+}_{\psi^n}$ ou à $I^{-}_{\psi^n}$ selon la parité de $\psi^{n_i}_i$ (ce qui est lié à sa fonction L d'Asai). (\cite{KMSW} p.69). 
	
	On a alors :
	\[
	S_{\psi} = \prod_{i \in I^{+}_{\psi^n}} O(l_i, \C) \times \prod_{i \in I^{-}_{\psi^n}} Sp( l_i, \C ) \times \prod_{j \in J_{\psi^n}} GL (l_j, \C).
	\]
	
	D'autre part le groupe $Z(\widehat{G})^{\Gamma} = \{ \pm 1 \}$ est envoyé diagonalement dans le membre de droite ci-dessus. Alors on voit que (\cite{KMSW} page 69)
	\begin{equation} \phantomsection \label{itm : centra}
		\mathcal{S}_{\psi} \simeq S_{\psi}^{\natural} \simeq (\Z / 2 \Z)^{| I^{+}_{\psi} |} \quad \text{et} \quad \overline{\mathcal{S}}_{\psi} \simeq \left\lbrace \begin{array}{cc}
			(\Z / 2 \Z)^{| I^{+}_{\psi^n} |} & \quad \text{si} \ \forall i \in I^{+}_{\psi^n} \ 2 | l_i, \\
			(\Z / 2 \Z)^{| I^{+}_{\psi^n} | - 1} & \quad \text{sinon.}
		\end{array} \right.	
	\end{equation}
	\begin{notation}
		On pose $ \Psi_2 (U_{E/F}(n), \eta_{\chi_{\kappa}} ) $ le sous-ensemble des paramètres discrets, i.e de la forme $ \psi = (\psi^n, \widetilde{\psi}) $ où $ \forall i \in I^{+}_{\psi^n}, l_i = 1$ et $ \forall j \in I^{-}_{\psi^n} \bigcup J_{\psi^n}, l_j = 0$ \big(cf (\ref{itm: forme normale})\big).
	\end{notation}
	\textbf{Localisation d'un paramètre global} \textbf{} \\
	$\bullet$ On commence avec les groupes $GL(n)$. Considérons une représentation automorphe cuspidale $\pi$ de $GL(n)$ et $v$ une place de $F$. La correspondance de Langlands local associe à $\pi_v $ un $A$-paramètre générique $\psi_{\pi_v} \in \Psi^{+} (GL(n, F_v))$. Ce processus nous permet de définir une application de localisation 
	
	\[
	\begin{array}{ccc}
		\Psi_{\text{glb}} (n) & \longrightarrow & \Psi^{+}_v (n) \\
		l_1(\pi_1 \boxtimes \nu_1) \boxplus \cdots \boxplus l_r(\pi_r \boxtimes \nu_r)  & \longmapsto & l_1 (\psi_{(\pi_1)_v} \boxtimes  \nu_1) \boxplus \cdots \boxplus l_r (\psi_{(\pi_r)_v} \boxtimes  \nu_r).
	\end{array}
	\]
	$\bullet$ Considérons maintenant les groupes $U_{E/F}(n)$. Fixons $\kappa \in \{ \pm 1 \}$ avec un caractère $\chi_{+} \in \mathcal{Z}_E^{+}$ et $\chi_{-} \in \mathcal{Z}_E^{-}$. Soit $ \psi = (\psi^n, \widetilde{\psi}) \in \Psi(U_{E/F}(n), \eta_{\chi_{\kappa}})$ un paramètre global.
	\begin{proposition} \phantomsection \cite{KMSW} prop 1.3.3.
		Pour chaque $\psi \in \Psi(U_{E/F}(n), \eta_{\chi_{\kappa}}) $, il existe un paramètre local $ \psi_v \in \Psi^{+}(U_{E_v / F_v}(n)) $ tel que $\psi^n_v = \eta_{\chi_{\kappa}} \circ \psi_v $
		\[
		\psi^n_v =  L_{F_v} \times SU(2) \xrightarrow{\psi_v} \prescript{L}{}{U}_{E_v / F_v}(n) \xrightarrow{\eta_{\chi_{\kappa}}} \prescript{L}{}{G}_{E_v / F_v} (n).
		\]
		Le paramètre local $\psi_v$ est unique à isomorphisme près.
	\end{proposition} 
	
	On peut utiliser cette proposition pour définir le diagramme commutatif suivant
	\begin{center}
		\begin{tikzpicture}[scale = 1]
			\draw (-1,0) node {$L_{F_v} \times SU(2) $};
			\draw (3,0) node {$ \prescript{L}{}{U_{E_v/F_v}(n)} $};
			\draw (7,0) node {$ \prescript{L}{}{G_{E_v/F_v}(n)} $};
			\draw (11,0) node {$ W_{F_v} $};
			\draw (-1,-1.5) node {$ \mathcal{L}_{\psi} \times SL(2, \C)  $};
			\draw (3,-1.5) node {$\prescript{L}{}{U_{E/F}(n)}$};
			\draw (7,-1.5) node {$\prescript{L}{}{G_{E/F}(n)}$};
			\draw (11,-1.5) node {$W_F$};
			\draw [->] (-1,-0.4) -- (-1,-1.15);
			\draw [->] (3, -0.4) -- (3,-1.15);
			\draw [->] (7, -0.4) -- (7,-1.15);
			\draw [->] (11, -0.4) -- (11,-1.15);
			\draw [->] (4.2, 0) -- (6,0)node [midway, above]{$\eta_{\chi_{\kappa}}$};
			\draw [->] (4.2, -1.5) -- (6,-1.5)node [midway, above]{$\eta_{\chi_{\kappa}}$};
			\draw [->] (8.2, 0) -- (10.3,0);
			\draw [->] (8.2, -1.5) -- (10.3,-1.5);
			\draw [->] (0.2,0) -- (2,0) node [midway, above]{$\psi_v$};
			\draw [->] (0.2, -1.5) -- (2, -1.5) node [midway, above]{$\widetilde{\psi}$};
			\draw [->] [domain=-1:7] plot (\x, -0.05 * \x * \x + 6 * 0.05 *\x + 7*0.05 + 0.5 );
			\draw (3, 1.6) node {$\psi^n_v$};
			\draw [->] [domain=-1:7] plot (\x, 0.05 * \x * \x - 6 * 0.05 *\x - 7*0.05 - 2 );
			\draw (3, -2.4) node {$\widetilde{\psi^n}$};
		\end{tikzpicture}
	\end{center}
	
	Le diagramme commutatif ci-dessus nous permet de définir les morphismes de localisation pour les centralisateurs. Le second morphisme vertical envoie $\Im(\psi_v)$ dans $\Im(\widetilde{\psi})$ induisant alors 
	\[
	S_{\psi} \longrightarrow S_{\psi_v}, \quad \quad \quad \mathcal{S}_{\psi} \longrightarrow \mathcal{S}_{\psi_v}, \quad \quad \quad S_{\psi}^{\natural} \longrightarrow S_{\psi_v}^{\natural}.
	\]
	
	\begin{remarque} \phantomsection \label{itm : générique}
		\begin{enumerate}
			\item[$\bullet$] Si $\psi \in \Psi(U_{E/F}(n), \eta_{\chi_{\kappa}})$ est un paramètre global générique alors les paramètres locaux $\psi_v$ le sont pour toute place $v$.
			\item[$\bullet$] Les deux premiers morphismes de localisation sont injectifs.
		\end{enumerate}
	\end{remarque}
	
	\subsection{Formules de multiplicité pour les groupes unitaires} \label{itm : formule de multiplicité}
	Soient $F$ un corps local et $E/F$ une extension quadratique. On a un unique groupe unitaire quasi-déployé $ U^* := U_{E/F} (n)$. Considérons une forme intérieure étendue $(\varrho, z) : U^*  \longrightarrow U $. La donnée $(\varrho, z)$ définit un caractère $ \chi_z \in X^*(Z(\widehat{U})^{\Gamma}) $ via (\ref{itm: kappa}). Considérons un paramètre $ \psi \in \Psi(U^*) $ et les centralisateurs associés 
	\begin{equation} \phantomsection \label{caractère}
		Z(\widehat{U})^{\Gamma} \hookrightarrow S_{\psi} \twoheadrightarrow S^{\natural}_{\psi}
	\end{equation}  
	
	On note $ \text{Irr} (S^{\natural}_{\psi}, \chi_z) $ l'ensemble des caractères de $S^{\natural}_{\psi}$ dont le tiré en arrière via (\ref{caractère}) induit le caractère $ \chi_z $. Rappelons également que $ \Pi_{unit} (U) $ (resp. $\Pi_{temp} (U)$, resp. $\Pi_{2, temp} (U)$) est l'ensemble des représentations unitaires (resp. tempérées, resp. de carré intégrable).
	
	\begin{théorème} (\cite{KMSW} théorème 1.6.1) \phantomsection \label{itm: local}
		\begin{enumerate}
			\item Soit $\psi \in \Psi (U^*)$. Il existe un ensemble fini $\Pi_{\psi}(U,\varrho)$ muni d'un morphisme vers $\Pi_{\text{unit}}(U)$. L'ensemble $\Pi_{\psi}(U,\varrho)$ ne dépend pas de $z$ et est muni d'une application
			\[
			\Pi_{\psi}(U,\varrho) \longrightarrow \text{Irr}(S_{\psi}^{\natural}, \chi_z), \qquad \pi \mapsto \langle \pi, - \rangle_{\varrho, z}
			\]
			L'ensemble $\Pi_{\psi}(U,\varrho)$ ainsi que $\langle-| -\rangle_{\varrho ,z}$ dépendent seulement de la classe d'équivalence $\Xi$ de $(\varrho, z)$. Pour $\psi$ un paramètre générique (i.e $\psi \in \Phi_{\text{bdd}}(U^*) $), l'ensemble $\Pi_{\psi}(U,\varrho)$ est non vide si et seulement si $\psi$ est $(U, \varrho)$-relevant.
			\item Supposons que $\psi \in \Phi_{\text{bdd}}(U^*)$, i.e $\psi$ est générique, alors le morphisme $ \Pi_{\psi}(U,\Xi) \longrightarrow \Pi_{\text{unit}}(U) $ est injectif et d'image contenue dans $\Pi_{\text{temp}}(U)$. Si $F$ est non archimédien alors l'application $\Pi_{\psi}(U,\Xi) \longrightarrow \text{Irr}(S_{\psi}^{\natural}, \chi_z)$ est une bijection.   
			\item On a 
			\[
			\Pi_{\text{temp}}(U) = \coprod_{\psi \in \Phi_{\text{bdd}}(U^*)} \Pi_{\psi}(U,\Xi) \quad \quad \text{et} \quad \quad \Pi_{\text{2,temp}}(U) = \coprod_{\psi \in \Phi_{\text{2,bdd}}(U^*)} \Pi_{\psi}(U,\Xi).
			\] 
		\end{enumerate}
	\end{théorème}
	Intéressons nous à présent à
	classification globale. Soient $E/F$ une extension quadratique d'un corps global $F$ et $(U, \varrho)$ une forme intérieure du groupe unitaire quasi-déployé $ U^* := U_{E/F}(n)$. On peut choisir $z$ tel que $(U, \varrho, z)$ est une forme intérieure étendue. On aimerait associer des paquets globaux $ \Pi_{\psi}(U, \varrho) $ à chaque paramètre $\psi \in \Psi (U^*, \eta_{\chi_{\kappa}})$. Il existe un morphisme de localisation $ \psi \longmapsto \psi_v $ de $\Psi (U^*, \eta_{\chi_{\kappa}})$ à $\Psi^{+}_{\text{unit}} (U^*_v)$. D'après le théorème \ref{itm: local}, pour chaque $\psi_v \in \Psi(U^*_v) $, on a un paquet $\Pi_{\psi_v}(U_v, \varrho_v)$ muni d'une application
	\[
	\Pi_{\psi_v}(U_v,\varrho_v) \longrightarrow \text{Irr}(S_{\psi_v}^{\natural}, \chi_z), \qquad \pi_v \mapsto \langle\pi_v, - \rangle_{\varrho_v, z_v}.
	\]  
	
	Pour un $A$-paramètre $\psi_v \in \Psi^{+}_{\text{unit}} (U^*_v) \setminus \Psi(U^*_v) $, on peut définir un paquet $\Pi_{ \psi_v } (U_v, \varrho_v)$. L'idée de la construction consiste à descendre à un sous-groupe de Levi $M_v$ de $U_v$ de sorte que $(\psi_v)_{M_v}$ appartienne à $\Psi(M^*_v)$ (pas seulement à $\Psi^{+}(M^*_v)$) et on applique le théorème \ref{itm: local} pour $M_v$ et $(\psi_v)_{M_v} \in \Psi(M^*_v)$. Pour plus de détails, voir \cite{Mok} page 32, 33 ou \cite{KMSW} section 1.6.4. 
	\begin{remarque} \phantomsection \label{itm : relevant}
		\begin{enumerate}
			\item[$\bullet$] Lorsque $\psi_v$ n'est pas $(U_v, \varrho_v)$-relevant alors le paquet $\Pi_{ \psi_v } (U_v, \varrho_v)$ est vide.
			\item[$\bullet$] Lorsque $(U_v, \varrho_v) = (U_v^*, 1)$ est quasi-déployé alors $\psi_v$ est toujours $(U_v, \varrho_v)$-relevant. Si de plus $\psi_v$ est générique alors le paquet $\Pi_{ \psi_v } (U_v, \varrho_v)$ est non vide.
		\end{enumerate}
	\end{remarque}
	
	On pose
	\[
	\Pi_{\psi}(U, \varrho) := \left\lbrace  \bigotimes_v \pi_v : \pi_v \in \Pi_{\psi_v}(U_v,\varrho_v), \langle\pi_v, - \rangle_{\varrho_v, z_v} = 1 \quad \text{pour presque tout } v \right\rbrace .
	\]
	\begin{remarque}
		L'ensemble $\Pi_{\psi}(U, \varrho)$ peut être vide, notamment lorsque $U$ n'est pas quasi-déployé et $\psi$ n'est pas générique.
	\end{remarque}
	Pour chaque $ \pi = \bigotimes_v \pi_v \in \Pi_{\psi}(U, \varrho) $, on lui associe un caractère de $S_{\psi}^{\natural}$ par la formule
	\[
	\langle \pi,s \rangle_{\varrho} := \prod_v \langle \pi_v, s_v \rangle_{\varrho_v, z_v}, \quad s \in S_{\psi}^{\natural}
	\]
	où $s_v$ désigne l'image de $s$ par le morphisme naturel $S_{\psi}^{\natural} \longrightarrow S_{\psi_v}^{\natural}$.
	
	\begin{remarque} \phantomsection \label{itm : conditions sur le centre}
		Pour $ s \in Z(\widehat{U}^*)^{\Gamma} $, on a $ \langle \pi, s \rangle_{\varrho} = 1 $ (\cite{KMSW}, p. 89, premier paragraphe).
	\end{remarque}
	
	\begin{definition} \phantomsection \label{itm : paquets}
		Soit $ \Pi_{\psi}(U, \varrho, \epsilon_{\psi}) := \left\lbrace \pi \in \Pi_{\psi} ( U, \varrho ) : \langle \pi , - \rangle_{\varrho} = \epsilon_{\psi} \right\rbrace $ où $\epsilon_{\psi}$ est le caractère d'Arthur. Pour $\psi$ un paramètre générique, $\epsilon_{\psi} \equiv 1 $. 
	\end{definition} 
	\begin{théorème} (\cite{KMSW} théorème $1.7.1)$ \phantomsection \label{itm: global}
		Soient $E/F$ une extension quadratique d'un corps global $F$ et $\kappa \in \{ \pm 1 \}$ ainsi que $\chi_{\kappa} \in \mathcal{Z}_{E}^{\kappa}$. Soit $(U, \varrho)$ une forme intérieure pure de $U^*$. Il existe un isomorphisme de $U(\mathbb{A}_F)$-modules
		\[
		L^2_{\text{disc}}(U(F) \setminus U(\mathbb{A}_F) ) \simeq \bigoplus_{\psi \in \Psi(U^*, \eta_{\chi_{\kappa}})} L^2_{\text{disc}, \psi} (U(F) \setminus U(\mathbb{A}_F) ). 
		\]
		
		Si $\psi = \phi$ est générique alors 
		\begin{enumerate}
			\item[$\bullet$] $L^2_{\text{disc}, \phi} (U(F) \setminus U(\mathbb{A}_F) ) = 0$ si $\psi \notin \Psi_2(U^*, \eta_{\chi_{\kappa}})$.
			\item[$\bullet$] $L^2_{\text{disc}, \phi} (U(F) \setminus U(\mathbb{A}_F) ) \simeq \bigoplus_{\pi \in \Pi_{\phi} (U, \varrho, \epsilon_{\psi})} \pi$ si $\psi \in \Psi_2(U^*, \eta_{\chi_{\kappa}})$.
		\end{enumerate}	
		En particulier si $\pi$ est une représentation automorphe du groupe unitaire $(U, \varrho)$ appartenant à un paquet global générique alors $m_{\pi} = 1$. 
	\end{théorème}
	\begin{proof}
		Ce théorème est le résultat global principal de \cite{KMSW}. Pour la première décomposition consulter page 151 et pour la deuxième décomposition, consulter page 205-206 de loc.cit. 
	\end{proof}
	
	
	Résumons à présent les résultats principaux de Moeglin dans \cite{Moe}, en particulier la classification des représentations supercuspidales des groupes unitaires $p$-adiques.
	
	Soit $\psi$ un homomorphisme de $W_E \times SL_2(\C)$ dans $GL_n(\C)$. On suppose que $\psi$ est semi-simple borné sur $W_E$ et continu. On suppose qu'il se décompose en une somme de représentations irréductibles
	$ \psi = \bigoplus_{\rho , a} \rho \otimes \sigma_a $
	où $\rho$ est une représentation irréductible de $W_E$ et $\sigma_a$ est l'unique représentation irréductible de $SL_2(\C)$ de dimension $a$. 
	
	On dit que $\psi$ est $\theta$-discret s'il est sans multiplicité et pour toute $\rho$, la classe de conjugaison de $\rho$ est invariante sous l'action de $g \longmapsto \J_{\rho} \prescript{t}{}{\overline{g}}^{-1} \J_{\rho}^{-1} $ où $\J_{\rho}$ est la matrice antidiagonale de taille $d_{\rho} = \dim \rho$ avec à la place $i$ l'élément $(-1)^{(i+1)}$. On dit que $\psi$ est stable si elle se prolonge en un homomorphisme de $W_F \times SL_2(\C)$ dans le groupe dual de $U_{E/F}(n)$.
	
	En fait, le morphisme $\eta_{\chi_{\kappa}, *}$ (voir \ref{itm: changement de base}) réalise une bijection entre $ \Phi_{\text{2, bdd}} (U_{E/F}(n)) $ et l'ensemble des morphismes $\theta$-discrets stables. De plus, si on note $\overline{\psi} = \eta_{\chi_{\kappa}, *}^{-1}(\psi) $ alors on a $ S_{\overline{\psi}}^{\natural} = (\Cent_{GL_n(\C)}\psi)^{\theta} $. On a la classification suivante des séries discrètes des groupes unitaires $p$-adiques.
	\begin{théorème} (5.7 de \cite{Moe})
		Il y a une bijection entre l'ensemble des paquets stables de séries discrètes de $U_{E/F}(n)$ et l'ensemble des classes de conjugaison des morphismes $\theta$-discrets et stables de $W_E \times SL_2(\C)$ dans $GL_n(\C)$.
	\end{théorème} 
	
	Ensuite on voudrait déterminer quels paquets contiennent des représentations supercuspidales. On aura donc besoin de la notion d'un morphisme sans trou et la construction d'un sous-groupe $A(\psi)$ de $(\Cent_{GL_n(\C)}\psi)^{\theta}$.
	
	Soit $\psi$ un morphisme $\theta$-discret et stable. On note $\psi[a]$ la composante isotypique de la représentation $\psi$, vue comme une représentation de $SL_2(\C)$, pour la représentation irréductible de $SL_2(\C)$ de dimension $a$. Pour tout $a$, $\psi[a]$ est une représentation de $W_E$. On dit que $\psi$ est sans trou si pour tout $a > 2$, $\psi[a]$ est une sous-représentation de $\psi[a-2]$ en tant que représentation de $W_E$. 
	
	On décompose $\psi$ en représentations irréductibles et on considère un couple $(\rho, a)$ de sorte que $ \rho \otimes \sigma_a$ soit une sous-représentation de $\psi$. On suppose qu'il existe $ 0 \leq b < a $ tel que $b = 0$ si $a$ est pair et si $b \neq 0$, $\rho \otimes \sigma_b $ soit aussi une sous-représentation de $\psi$. On note alors $a_{-}$ le plus grand $b$ vérifiant cette propriété. On note $z_{\rho ,a}$ l'élément du centralisateur de $\psi$ dans $GL_n(\C)$ dont les valeurs propres $-1$ ont exactement pour espace propre la somme $\rho \otimes \sigma_a \oplus \rho \otimes \sigma_{a_{-}}$. On note $A(\psi)$ le sous-groupe du centralisateur de $\psi$ engendré par ces éléments $z_{\rho ,a}$.
	
	\begin{théorème} (8.4.4 de \cite{Moe}).
		Le paquet stable correspondant à $\psi$ contient une représentation supercuspidale si et seulement si $\psi$ est sans trou. Lorsque $\psi$ est sans trou alors il existe un unique caractère $\epsilon_{alt}$ de $A(\psi)$ tel que le nombre de représentations supercuspidales est le cardinal de l'ensemble des caractères de $(\Cent_{GL_n(\C)}\psi)^{\theta}/ \{ \Id, - \Id \}$ dont la restriction à $A(\psi)$ vaut $\epsilon_{alt}$. 
	\end{théorème}
	
	Enfin, lorsque $\psi : W_E \times SL_2(\C) \longrightarrow GL_n(\C)$ est trivial sur $SL_2(\C)$, on voit que $ \psi $ est une somme de représentations de la forme $\rho \otimes \sigma_1$, en particulier $A(\psi)$ est trivial. On en déduit le résultat suivant qui nous sera utile pour la suite.
	\begin{proposition} \phantomsection \label{itm : C.Moeglin} (C.Moeglin)
		Soit $ \phi : W_F \times SL_2(\C) \longrightarrow  \prescript{L}{}{U}_{E/F}(n) $ un $L$-paramètre discret trivial sur $SL_2(\C)$, alors le $L$-paquet $\Pi_{\phi}(U_{E/F})(n)$ ne contient que des représentations supercuspidales.
	\end{proposition}
	\subsection{Une application à la globalisation des représentations locales} Dans ce paragraphe, tous les groupes et paramètres globaux sont notés par un point $ \bullet $ au dessus. On suppose $n$ impair. Le but est de montrer l'existence de représentations automorphes satisfaisant des propriétés particulières. 
	
	Décrivons tout d'abord le passage du groupe unitaire au groupe de similitudes unitaires. Soit $F$ un corps totalement réel ainsi que $E/F$ une extension quadratique. Supposons que $V$ est un espace hermitien de dimension $n$ relatif à l'extension $E/F$. Considérons $ \overset{\bullet}{G} = GU(V)$ le groupe de similitudes unitaires défini sur $\Q$ par
	\[
	\overset{\bullet}{G}(R) = \big\{ g \in GL(V \otimes R) | <gv, gw> = \nu (g) <v,w>, \ v,w \in V \otimes R \big\}
	\] 
	pour toute $\Q$-algèbre $R$ et $\nu(g) \in R^{\times}$. Notons $\overset{\bullet}{U} = U(V)$ le groupe unitaire associé.
	
	\begin{proposition} (\cite[section. CHL.IV.C, prop. 1.1.4]{CHLN}). \phantomsection \label{itm: de U à GU}
		Supposons que $n$ est impair. Soit $\pi$ une représentation automorphe irréductible de $\overset{\bullet}{G}(\A)$ dont la restriction à $\overset{\bullet}{U}(\A)$ contient une représentation automorphe irréductible $\sigma$. Si $\sigma$ apparaît avec multiplicité $1$ dans le spectre discret de $\overset{\bullet}{U}$ alors $\pi$ apparaît avec multiplicité $1$ dans le spectre discret de $\overset{\bullet}{G}$. De plus, si $\chi$ est le caractère central de $\pi$ alors $\pi$ est la seule représentation automorphe de $\overset{\bullet}{G}(\A)$ contenant $\sigma$ avec le caractère central $\chi$.
	\end{proposition}
	
	Soit $G$ un groupe de similitudes unitaires non ramifié en $n = 2k + 1$ variables défini sur $\Q_p$ et $U \subset G$ le groupe unitaire noyau du facteur de similitudes. Soit $K/\Q_p$ l'extension quadratique non ramifiée. On a une inclusion $K^{\times} \hookrightarrow Z_G$. 
	\begin{proposition} \phantomsection \label{itm : passage local}
		On a une bijection entre les représentations irréductibles de $G(\Q_p)$ et les couples $(\pi, \chi)$ où $\pi$ est une représentation irréductible du groupe unitaire $U(\Q_p)$ et $\chi$ un caractère de $K^{\times}$ tel que $\omega_{\pi | K^{\times} \bigcap U(\Q_p)} = \chi_{| K^{\times} \bigcap U(\Q_p)}$. 
	\end{proposition}
	Soit $U^{*}_p$ le groupe unitaire $p$-adique quasi-déployé en $n$ variables associé à l'extension quadratique $E_p / \Q_p$ puis $\phi_p \in \Phi_{2}(U_p^*) $ un $L$-paramètre discret.
	\begin{proposition} (\cite{KMSW} prop. $4.4.1$) \phantomsection \label{globaliser les paramètres locaux}
		Supposons qu'il existe un corps $CM$ de la forme $ \overset{\bullet}{E} = F^{+}\mathcal{K}$ où $\mathcal{K}$ est un corps quadratique imaginaire et 
		tel que $ \overset{\bullet}{E}_p = E_p $. Alors il existe une place inerte $w$ et un paramètre global $\overset{\bullet}{\phi} \in \Phi_{2}(\overset{\bullet}{U^*})$ du groupe unitaire quasi-déployé $\overset{\bullet}{U^*}$ en $n$ variables associé à $\overset{\bullet}{E} / F^+$ tels que :
		\begin{enumerate} 
			\item [(i)] $\overset{\bullet}{\phi}_p = \phi_p$ et $\overset{\bullet}{\phi}_w \in \Phi_{\text{bdd}} (U_w^*)$
			\item [(ii)] $\overset{\bullet}{\phi}_u$ est un paramètre discret pour toute $u$ infinie.
			\item [(iii)] Les morphismes canoniques 
			$S_{\overset{\bullet}{\phi}} \longrightarrow S_{\overset{\bullet}{\phi}_p}$ et
			$
			S_{\overset{\bullet}{\phi}} \longrightarrow S_{\overset{\bullet}{\phi}_w}
			$
			sont des isomorphismes.
		\end{enumerate}	
	\end{proposition}
	On utilise la même construction que celle décrite dans la proposition 4.4.1 de \cite{KMSW}. Dans la proposition 4.4.1 et 4.3.1 de loc.cit, les auteurs globalisent également les groupes unitaires réels et ils supposent donc que le corps global possède au moins deux places infinies. Dans notre cas cette hypothèse n'est pas nécessaire. 
	\begin{proof} On donne seulement les étapes principales.
		
		Écrivons $\eta_{\chi_{\kappa} *} \phi_p = \phi_p^n$ sous la forme d'une somme de paramètres simples de groupes linéaires généraux $ \phi_p^n = \phi^{n_1}_{1,p} \oplus \cdots \oplus \phi^{n_r}_{r,p} $.
		
		À chaque $\phi^{n_i}_{i, p}$ on peut associer un groupe global $\overset{\bullet}{U}(n_i)$ et un morphisme $\overset{\bullet}{\eta}_{\chi_{\kappa_i}}$ avec un unique signe $\kappa_i$ de sorte que $\phi^{n_i}_{i, p}$ provient d'un $L$-paramètre $\phi_{i,p}$ dans $\Phi_{2, bdd}(\overset{\bullet}{U_p}(n_i))$ via le morphisme $(\overset{\bullet}{\eta}_{\chi_{\kappa_i}})_p$. 
		
		Fixons une place inerte $w$ et des paramètres simples deux à deux distincts $\phi_{1, w}, \cdots, \phi_{r, w}$ de sorte que $ \dim \phi_{i, w} =\dim \phi_{i,p} $.
		
		D'après le lemme 4.3.1 de loc.cit, pour chaque $i \in \{ 1, \cdots, r \}$ il existe un paramètre global simple $\overset{\bullet}{\phi_i}$ de sorte que
		\begin{enumerate}
			\item[$\bullet$] $(\overset{\bullet}{\phi_i})_p = \phi_{i,p}$ et $(\overset{\bullet}{\phi_i})_w = \phi_{i, w}$,
			\item[$\bullet$] $(\overset{\bullet}{\phi_i})_{u}$ est un paramètre discret régulier et suffisamment régulier dans le sens de \ref{itm: suffisament régulier} si $u$ est une place infinie.  
		\end{enumerate} 
		
		Construisons ensuite un $L$-paramètre global de $GL(n)$ de la manière suivante
		\[
		\overset{\bullet}{\ \phi^n} = \overset{\bullet}{\ \ \phi^{n_1}_1} \boxplus \cdots \boxplus \overset{\bullet}{\ \ \phi^{n_r}_r}.
		\] 	
		
		On montre que $\overset{\bullet}{\ \phi^n}$ s'étend en un $L$-paramètre $ \overset{\bullet}{\phi} $ dans $\Phi_{2}(\overset{\bullet}{U^*})$ en utilisant la compatibilité des signes $\kappa_i$ avec la parité des $\overset{\bullet}{\phi_i}$. On démontre enfin les propriétés (i), (ii) et (iii) par le même argument que celui de la proposition 4.4.1 de \cite{KMSW}. Remarquons enfin que la condition imposée sur les places infinies assure que le paramètre global est générique.
	\end{proof}
	
	Soit $\pi_p$ une représentation de carré intégrable d'un groupe unitaire $U_p$ $p$-adique\footnote{On rappelle que $n$ étant impair alors $U_p$ est quasi-déployé.} et $\widetilde{\pi}_p$ une représentation du groupe de similitudes unitaires correspondant telle que $(\widetilde{\pi}_p)_{| U_p} \simeq \pi_p $. Supposons que $\overset{\bullet}{U}$ est une forme intérieure de $ \overset{\bullet}{U^*} $ tel que $\overset{\bullet}{U}_p = U_p$ et $\overset{\bullet}{U}$ est quasi-déployé en toutes les places finies. Étendons $\overset{\bullet}{U}$ en une forme intérieure pure. D'après la proposition \ref{itm : forme intérieure pure}, il suffit d'étendre les groupes unitaires locaux en formes intérieures pures satisfaisant une certaine condition de parité.
	
	Comme $n$ est impair, d'après l'exemple \ref{itm : étendre}, pour chaque place $v$ il y a deux manières d'étendre le groupe unitaire local en une forme intérieure pure. Pour $v$ une place infinie de $ \overset{\bullet}{F^+} $, si le groupe $ \overset{\bullet}{U}_v$ est de signature $(i, n-i)$ avec $ i < n-i $, on choisit $a_v = ([\frac{n}{2}] + (n-i) ) \ \text{mod} \ 2$. Si $v$ est une place finie différente de $p$, on choisit $a_v = 0 \in \Z / 2\Z$ et finalement, on choisit $a_p$ l'unique élément de $\Z / 2\Z$ de sorte que la somme des $a_v$ s'annule.     
	\begin{proposition} \phantomsection \label{itm : globaliser auto}
		Il existe alors une représentation automorphe $\Pi$ de $\overset{\bullet}{U}$ telle que $\Pi_p \simeq \pi_p$. De plus s'il existe une représentation automorphe $\Pi'$ de $\overset{\bullet}{U}$ satisfaisant $(\Pi')^p \simeq (\Pi)^p $ alors $ \Pi \simeq \Pi' $. On a le même résultat pour les groupes de similitudes unitaires, à un caractère non ramifié près. Plus précisément, il existe une représentation automorphe $\widetilde{\Pi}$ de $\overset{\bullet}{G}$ telle que $\widetilde{\Pi}_p \simeq \widetilde{\pi}_p \otimes (\omega \circ c) $ où $ c : G(\Q_p) \longrightarrow \Q_p^{\times} $ est le caractère de similitudes et $\omega$ est un caractère non ramifié de $\Q_p^{\times}$. De plus s'il existe une représentation automorphe $\widetilde{\Pi}'$ de $\overset{\bullet}{G}$ satisfaisant $(\widetilde{\Pi}')^p \simeq (\widetilde{\Pi})^p $ alors $ \widetilde{\Pi} \simeq \widetilde{\Pi}' $. 
	\end{proposition}
	\begin{proof}
		Le cas des groupes de similitudes unitaires se déduit de celui pour les groupes unitaires auquel on se ramène donc.	Montrons l'existence de $\Pi$.
		
		On suppose que $\pi_p$ appartient à $\Pi_{\phi_p}(U^*_p, \varrho)$ pour un paramètre discret $\phi_{p}$. On note alors $\overset{\bullet}{\phi} \in \Phi_{2}(\overset{\bullet}{U^*}) $ le paramètre global qui existe d'après le lemme \ref{globaliser les paramètres locaux}. En particulier, il existe une place inerte $w$ telle que $ \overset{\bullet}{\phi}_w \in \Phi_{\text{bdd}} (U_w^*) $ et les morphismes canoniques 
		$S_{\overset{\bullet}{\phi}} \longrightarrow S_{\overset{\bullet}{\phi}_p}$ et
		$
		S_{\overset{\bullet}{\phi}} \longrightarrow S_{\overset{\bullet}{\phi}_w}
		$
		sont des isomorphismes. 
		
		On montre que les paquets $\Pi_{\overset{\bullet}{\phi_v}}(\overset{\bullet}{U^*_v}, \varrho)$ sont non vides. En effet, si $v = p$, $\Pi_{\overset{\bullet}{\phi_v}}(\overset{\bullet}{U^*_v}, \varrho)$ contient $\pi_p$. D'autre part, comme $\overset{\bullet}{\phi}_u$ est un paramètre discret si $u$ est une place infinie, on en déduit qu'il existe toujours une représentation dans le paquet correspondant. 
		
		Considérons une place finie $v \neq p$. Comme $ \overset{\bullet}{\phi}$ est générique, le paramètre localisé $ \overset{\bullet}{\phi_v}$ l'est également (remarque \ref{itm : générique}). Comme $\overset{\bullet}{U}$ est quasi-déployé en $v$ alors $\Pi_{\overset{\bullet}{\phi_v}}(\overset{\bullet}{U^*_v}, \varrho)$ est non vide d'après la remarque \ref{itm : relevant}.  
		
		Pour chaque $v \notin \{ p, w \}$, choisissons une représentation $\pi_v \in \Pi_{\overset{\bullet}{\phi_v}}(\overset{\bullet}{U^*_v}, \varrho) $. On va choisir une représentation $\pi_w \in \Pi_{\overset{\bullet}{\phi_w}}(\overset{\bullet}{U^*_q}, \varrho) $ de sorte que la représentation $ \bigotimes_v \pi_v $ soit automorphe. D'après le théorème \ref{itm: global} il faut donc choisir $\pi_w$ satisfaisant la relation :
		\begin{equation} \phantomsection \label{itm:preuve}
			\langle \pi,s \rangle_{\varrho} := \prod_v \langle \pi_v, s_v \rangle_{\varrho_v, z_v} = 1, \quad \forall s \in S_{\overset{\bullet}{\phi}}^{\natural}
		\end{equation}
		
		Comme $ S_{\overset{\bullet}{\phi}} \simeq S_{\overset{\bullet}{\phi}_w} \simeq S^{\natural}_{\overset{\bullet}{\phi}_w}$ (le dernier isomorphisme résulte du fait que $\overset{\bullet}{\phi}_w$ est un paramètre discret) et d'autre part les caractères $ \langle s_v, \pi_v \rangle $ sont fixés (pour $v \neq w$) , il existe un unique caractère dans $\text{Irr}(S_{\phi_w}^{\natural}, \chi_{z_w})$ satisfaisant l'égalité au-dessus. Le point $2$ de \ref{itm: local} affirme l'existence d'une représentation $\pi_w$ satisfaisant (\ref{itm:preuve}). \\
		
		Supposons maintenant qu'on a une représentation $\Pi'$ telle que $(\Pi')^p \simeq (\Pi)^p $. Montrons que $\Pi$ et $\Pi'$ sont dans le même paquet. Supposons que $\Pi' \in \Pi_{\overset{\bullet}{\phi'}}(\overset{\bullet}{U}, \varrho, \epsilon_{\phi})$ et on veut montrer que $ \overset{\bullet}{\phi'} = \overset{\bullet}{\phi} $.
		
		
		À presque toutes les places finies non ramifiées $ v $ où $ \pi_v $ est non ramifiée, les localisations $ \overset{\bullet}{\phi}_v $ et $ \overset{\bullet}{\phi'}_v $ sont égales. En effet, notre hypothèse suffisamment régulière implique que ces paramètres sont génériques. D'après \cite[pg 189]{Mok}, ces paramètres locaux sont factorisés via $ \prescript{L}{}{M} $ où $ M $ est le sous-groupe de Levi minimal de $ U_v $ et correspondent au même paramètre sphérique de $ M $. Cela implique que $ \overset{\bullet}{\phi^n} $ et $ (\overset{\bullet}{\phi'})^n $ ont la même "Hecke string". Alors, par \cite{JS}, \cite[Theorem 4.3]{Art4} on voit que $ \overset{\bullet}{\phi} $ et $ \overset{\bullet}{\phi'} $ sont égaux.
		
		Comme $ \Pi $ et $\Pi'$ sont des représentations automorphes, on a les égalités suivantes :
		\[
		\langle \Pi,s \rangle_{\varrho} = \prod_v \langle \Pi_v, s_v \rangle_{\varrho_v, z_v} = \langle \Pi_p, s_p \rangle_{\varrho_p, z_p}  \prod_{v \neq p} \langle \Pi_v, s_v \rangle_{\varrho_v, z_v} = 1, \quad \forall s \in S_{\overset{\bullet}{\phi}}^{\natural},
		\]
		et	
		\[
		\langle \Pi',s \rangle_{\varrho} = \prod_v \langle \Pi'_v, s_v \rangle_{\varrho_v, z_v} = \langle \Pi'_p, s_p \rangle_{\varrho_p, z_p}  \prod_{v \neq p} \langle \Pi_v, s_v \rangle_{\varrho_v, z_v} = 1, \quad \forall s \in S_{\overset{\bullet}{\phi}}^{\natural}.
		\]
		On en déduit donc que :
		\[
		\langle \Pi_p, s_p \rangle_{\varrho_p, z_p} = \langle \Pi'_p, s_p \rangle_{\varrho_p, z_p} \quad \forall s \in S_{\overset{\bullet}{\phi}}^{\natural}.
		\]
		
		D'autre part, le morphisme $ S_{\overset{\bullet}{\phi}} \longrightarrow S_{\overset{\bullet}{\phi}_p} $ est un isomorphisme et on a
		\[
		\langle \Pi_p, s_p \rangle_{\varrho_p, z_p} = \langle \Pi'_p, s_p \rangle_{\varrho_p, z_p} \quad \forall s \in S_{\overset{\bullet}{\phi}_p}^{\natural}.
		\]
		Le point $2$ de \ref{itm: local} implique que $\Pi_p \simeq \Pi'_p$, autrement dit on a $ \Pi \simeq \Pi' $.
		
		Considérons le cas des groupes de similitudes unitaires. Par le lemme 4.1.2 de \cite{LJ}, nous pouvons étendre $ \Pi $ à une représentation automorphe cuspidale algébrique $ \overline{\Pi} $ de $ \overset{\bullet}{G}(\A) $. De plus, nous pouvons supposer que $ \overline{\Pi} $ est cohomologique puisque $ \Pi $ l'est.
		
		Considérons la suite exacte
		\[
		1 \longrightarrow \overset{\bullet}{U} \longrightarrow \overset{\bullet}{G} \xrightarrow{c} \Gm \longrightarrow 1
		\]
		
		Comme $ (\overline{\Pi}_p)_{| \overset{\bullet}{U}(\Q_p)} \simeq (\widetilde{\pi}_p)_{| \overset{\bullet}{U}(\Q_p)} $, il existe donc un caractère $ \lambda : \Q^{\times}_p \to \C^{\times}$ tel que $\overline{\Pi}_p \otimes (\lambda \circ c) \simeq \widetilde{\pi}_p $.
		
		Il y a un isomorphisme de groupes topologiques
		\begin{align*}
			\Q^{\times} \times \R_{> 0} \times \prod \Z_v^{\times} & \longrightarrow \Gm(\A) \\
			(r, t, (u_v)) & \longmapsto (rt, ru_2, ru_3, \cdots).
		\end{align*}
		
		D'où il y a un caractère $ \overline {\Lambda} $ de $ \Q^{\times} \times \R_{> 0} \times \prod \Z_v^{\times} $ tel que $ \overline{\Lambda} $ est trivial sur $ \Q^{\times} \times \R_{> 0} $, $ \overline{\Lambda}_{| \Z_p^{\times}} \equiv \lambda_ {| \Z_p^{\times}} $ et $ \overline{\Lambda} (-1, 1, (-1)) = 1 $. Ce caractère descend en un caractère de Hecke $ \Lambda $ de $ \Gm(\Q) \backslash \Gm(\A) $ tel que $ \Lambda_p = \lambda \otimes \omega $ où $ \omega: \Q^{\times}_p \longrightarrow \C^{\times} $ est un caractère non ramifié et $ \Lambda_{\infty} $ est trivial. En particulier si on note $ \widetilde{\Pi}: = \overline{\Pi} \otimes (\Lambda \circ c) $ alors il est également cohomologique (puisque $ \overline{\Pi} $ l'est) et $ \widetilde{\Pi}_p \simeq \widetilde{\pi}_p \otimes (\omega \circ c) $.
		
		Grâce à la proposition \ref{itm: de U à GU}, la dernière assertion se déduit de celui pour les groupes unitaires.
	\end{proof}
	
	\section{Un cas PEL de la conjecture de Kottwitz} \label{itm: cohomologie}
	
	\subsection{Notations} \phantomsection \label{itm : section notations}
	Dans cette section, nous précisons les notations du \ref{text : thm} de l'introduction, concernant $r_{\mu}$ puis les $\tau_i$. En vue du théorème \ref{itm : thm partiel} de l'appendice, on ne suppose pas pour l'instant que $d = 1$.
	
	Rappelons que $ \mathcal{D}_{\Q_p} = (F_p, *, V, \langle \cdot | \cdot \rangle, \mu, b)$ est une donnée de Rapoport-Zink basique de type PEL unitaire non ramifiée simple de signature $(1, n-1), (0, n), \cdots (0, n)$. On note $ G $ le groupe de similitudes unitaires $p$-adique quasi-déployé (en $n$ variables) associé ainsi que $(\mathcal{M}_{K_p})_{K_p}$ la tour d'espaces de Rapoport-Zink. La cohomologie de cet espace est munie d'une action de $G(\Q_p) \times J_b(\Q_p) \times W_{E_p} $\footnote{Puisque $b$ est basique et $n$ est impair, on a $G(\Q_p) = \J_b(\Q_p)$} où on rappelle que $F_p$ est aussi le corps de définition de $\mu$, autrement dit $E_p = F_p$.
	
	La représentation $r_{\mu}$ de $\prescript{L}{}{G} = \widehat{G} \rtimes W_{F_p}$ a une action triviale de $W_{F_p}$, et $r_{\mu |_{\widehat{G}}}$ est donnée par la formule
	\[
	r_{\mu |_{\widehat{G}}}  = \bigotimes_{\tau \in \Phi} \Big( \overset{p_{\tau}}{\bigwedge} St_n \Big)^* \otimes (St_1)^{-1} = (St_n)^* \otimes (St_1)^{-1}
	\]
	où $St_i$ désigne la représentation standard de dimension $i$. En particulier, on notera que $\dim r_{\mu} = n$.
	
	Notons $U_p$ le groupe unitaire défini sur $\Q_p$ noyau du facteur de similitude. Le groupe de Langlands dual est $ \prescript{L}{}{U}_p = \big( \prod_{\tau \in \Phi} GL_n(\C) \big) \rtimes W_{\Q_p} $. On a également un morphisme de $L$-groupes
	\[
	\prescript{L}{}{G}_p = \big( \prod_{\tau \in \Phi} GL_n(\C) \times \C^{\times} \big) \rtimes W_{\Q_p} \longrightarrow \big( \prod_{\tau \in \Phi} GL_n(\C) \big) \rtimes W_{\Q_p} = \prescript{L}{}{U}_p
	\]
	
	Étant donné un $L$-paramètre discret $ \varphi : W_{\Q_p} \longrightarrow \big( \prod_{\tau \in \Phi} GL(V) \times \C^{\times} \big) \rtimes W_{\Q_p} $ où $V \simeq \C^n$, on obtient un $L$-paramètre discret $ \widetilde{\varphi} : W_{\Q_p} \longrightarrow \big( \prod_{\tau \in \Phi} GL(V) \big) \rtimes W_{\Q_p} $ de $U_p$. 
	
	Puisque $ U_p = \Res_{F^+_p/\Q_p} (U_{F_p / F^+_p}) $ (où $F_p^{+}$ est le sous-corps de $F_p$ fixé par l'involution $*$) le lemme de Shapiro
	\[
	H^1(W_{\Q_p}, \prod_{\tau \in \Phi} GL(V)) \simeq H^1 (W_{F_p^+}, GL(V))
	\]
	nous donne un $L$-paramètre de $U_{F_p / F^+_p}$ :
	\[
	\widetilde{\varphi}_{F_p^+} : W_{F_p^+} \longrightarrow  GL(V) \rtimes W_{F^+_p} = \prescript{L}{}{U_{F_p / F_p^{+}}}.
	\] 
	
	D'après la proposition \ref{itm : C.Moeglin}, le $L$-paquet $\Pi_{\varphi}(G(\Q_p))$ est un paquet cuspidal. Comme dans la section précédente, pour chaque $\kappa \in \{ \pm 1 \}$ et $\chi_{\kappa} \in \mathcal{Z}^{\kappa}_{F_p}, $ il y a un morphisme de changement de base (cf \cite{Mok}, p.9)
	\begin{align*} 
		\eta_{\chi_{\kappa}, *} : \Phi (U_{F_p/F_p^{+}}) &\longrightarrow \Phi (GL_{F_p}(n)) \\
		\phi & \longmapsto  \eta_{\chi_{\kappa}} \circ \phi
	\end{align*}
	de plus, si $\kappa = 1$ et $\chi_{\kappa} = 1$ alors $\eta_{\chi_{\kappa}} \circ \phi$ est juste la restriction de $\phi$ sur $L_{F_p} = W_{F_p} \times SU(2) $.
	
	La restriction de $\widetilde{\varphi}_{F_p^+}$ sur $W_{F_p}$ se décompose donc en une somme de $L$-paramètres discrets simples de groupes linéaires généraux $\widetilde{\varphi}_{F_p^+ | F_p} = \widetilde{\varphi}^{n_1}_{1} \oplus \cdots \oplus \widetilde{\varphi}^{n_r}_{r}$. 
	On a alors une décomposition de $W_{F_p^+ | F_p}$-représentations $ V = \bigoplus_{i=1}^{r} V_i $ où $V_i$ correspond à $\widetilde{\varphi}_i^{n_i}$.
	
	Cela implique une décomposition
	\[
	r_{\mu} \circ \widetilde{\varphi}_{F_p} = \bigoplus_{ i = 1}^r r_{\mu_{i}} \circ \widetilde{\varphi}^{n_i}_{i}
	\]
	où $\mu_i = (1, n_i - 1), (0, n_i), \cdots, (0, n_i)$.
	
	D'autre part, on a
	\[
	S_{\widetilde{\varphi}} \simeq S_{\widetilde{\varphi}_{F_p^+}} \approxeq \prod_{i = 1}^{r} O(1, \C) = \prod_{i = 1}^{r} \Z / 2\Z \approxeq S^{\natural}_{\widetilde{\varphi}_{F_p^+}}. 
	\]
	
	\begin{definition} \phantomsection \label{itm : rep}
		Pour $ 1 \leq i \leq r $, on définit un caractère $\tau_i$ de $S_{\widetilde{\varphi}}$ par la formule
		\[
		\tau_i ( \underbrace{1, \cdots, 1, -1}_{j}, 1, \cdots, 1) = \left\lbrace \begin{array}{ccc}
			-1 & si & i=j \\
			1 & si & 1 \leq i \neq j \leq r
		\end{array} \right.
		\]	
	\end{definition}
	L'action de $S_{\widetilde{\varphi}}$ sur $V_i$ se factorise par le caractère $\tau_i$, on en déduit que $ r_{\mu_i} \circ \widetilde{\varphi}^{n_i}_{i} = \hom_{ S_{\widetilde{\varphi}} } \big( \tau_i, r_{\mu} \circ \widetilde{\varphi}_{F_p} \big) $ et donc
	
	\[
	r_{\mu} \circ \widetilde{\varphi}_{F_p} = \bigoplus_{i = 1}^r \hom_{S_{\widetilde{\varphi}}} (\tau_i, r_{\mu} \circ \widetilde{\varphi}_{F_p}).
	\]
	
	Au niveau du groupe de similitudes unitaires, puisque $\C^{\times}$ est abélien, on a une décomposition de $ W_{F_p}$-module
	\[
	r_{\mu} \circ \varphi_{F_p} = \bigoplus_{i = 1}^r \hom_{S_{\varphi}} (\tau_i, r_{\mu} \circ \varphi_{F_p}) = \bigoplus_{i = 1}^r  r_{\mu_i} \circ \varphi^{n_i}_{i, F_p}.
	\]

	\subsection{Cohomologie d'intersection des variétés de Shimura d'après \cite{Mo}}
	Dans ce paragraphe, on va utiliser des résultats dans \cite{Mo}. La lettre $G$ désigne un groupe de similitudes unitaires en $n$ variables associé à une forme hermitienne définie sur une extension quadratique imaginaire $ \overset{\bullet}{F} / \Q $.
	
	
	Pour tout $K \subset G(\A_f)$ un sous-groupe compact ouvert suffisamment petit, on a une variété $\Sh_K$. Lorsque $G$ n'est pas anisotrope modulo son centre, $\Sh_K$ n'est pas compacte. Dans ce cas, on a la compactification minimale (ou de Satake-Baily-Borel) de $\Sh_K$
	\[
	j : \Sh_K \longrightarrow \Sh_K^*
	\]  
	telle que $\Sh_K^*$ est une variété projective normale dont $\Sh_K$ est un ouvert dense. En général $\Sh_K^*$ n'est pas lisse et on utilise la cohomologie d'intersection telle qu'elle est étudiée notamment dans \cite{Mo}.
	
	Rappelons les relations entre la cohomologie étale (à support compact ou non), la cohomologie $L^2$ et la cohomologie d'intersection.    
	
	Comme auparavant, pour $\xi$ une représentation algébrique irréductible de dimension finie de $G$, on peut définir un système local $\mathcal{L}_{\xi}$ sur $\Sh_K$ ainsi qu'un faisceau d'intersection $\mathcal{IC}_{\xi}$ sur $\Sh_K^*$. On dispose donc des groupes de cohomologie étale $H^i_c(\Sh_K, \mathcal{L}_{\xi})$, de $L^2$-cohomologie $H^i_{(2)}(\Sh_K, \mathcal{L}_{\xi})$ et de cohomologie d'intersection $ IH^i(\Sh_K, \mathcal{IC}_{\xi})$. En prenant la limite inductive sur les sous-groupes compacts ouverts suffisamment petits $K$, on dispose des groupes de cohomologie étale (resp. de $L^2$-cohomologie, resp. de cohomologie d'intersection) à niveau infini $ H^i_c(\Sh, \mathcal{L}_{\xi}) $ (resp. $H^i_{(2)}(\Sh, \mathcal{L}_{\xi})$, resp. $ IH^i(\Sh, \mathcal{IC}_{\xi})$).
	
	
	Il y a un diagramme commutatif entre ces groupes de cohomologie.
	\begin{center} 
		\begin{tikzpicture}[scale = 1]
			\draw (-1,0) node {$H_c^i(\Sh, \mathcal{L}_{\xi})$};
			\draw (4,0) node {$ IH^i(\Sh, \mathcal{IC}_{\xi}) $};
			\draw (4,-1.5) node {$ H^i_{(2)}(\Sh, \mathcal{L}_{\xi}) $};
			\draw [->] (4, -0.4) -- (4,-1.15)node[midway, right]{$\approx$};
			\draw [->] (0.2,0) -- (2.75,0) node [midway, above]{$(1)$};
			\draw [->] (-0.8, -0.4) -- (2.75, -1.5) node [midway, above]{$(2)$};
		\end{tikzpicture}
	\end{center}
	
	Ces morphismes sont $G(\A_f)$-équivariants et le morphisme $(1)$ est de plus $\Gamma$-équivariant (où $\overset{\bullet}{E}$ est le corps de définition et $\Gamma$ = Gal$(\overline{\overset{\bullet}{E}}/\overset{\bullet}{E})$ ). D'après la conjecture de Zucker (démontrée par Looijenga \cite{Lo}, Looijenga-Rapoport \cite{LoR} et Saper-Stein \cite{SS}), le morphisme vertical est en fait un isomorphisme. 
	
	On exploitera le morphisme $(2)$ afin de comparer $H_c^i(\Sh, \mathcal{L}_{\xi})$ et $IH^i(\Sh, \mathcal{IC}_{\xi})$.
	\begin{proposition} \label{itm : comparer}
		Supposons maintenant $\xi$ régulier. Soit $\Pi = \Pi_{\infty} \otimes \Pi_f$ une représentation automorphe cuspidale de $G(\A)$ qui est $\xi$-cohomologique et que $ \Pi_v $ est supercuspidale pour une place finie $v$, alors $H^i(\Sh, \mathcal{L}_{\xi})$ est concentré en degré moitié et
		\[
		H^d_c(\Sh, \mathcal{L}_{\xi}) [\Pi_f] \approx H^d_{(2)}(\Sh, \mathcal{L}_{\xi})[\Pi_f].
		\]
		
		Puisque le morphisme $(1)$ est $\Gamma$-équivariant, on en déduit qu'il y a un isomorphisme $\Gamma$-équivariant 
		\[
		H^d_c(\Sh, \mathcal{L}_{\xi}) [\Pi_f] \approx IH^d(\Sh, \mathcal{IC}_{\xi})[\Pi_f]
		\]
		où $d$ est la dimension de la variété de Shimura.
	\end{proposition} 
	Dans \cite{MT} proposition $1$, Mokrane et Tilouine ont montré ce résultat pour les groupes de similitudes symplectiques mais leur démonstration s'adapte au cas des groupes de similitudes unitaires. On va rappeler brièvement leur argument.
	
	\begin{proof}
		Notons $ \mathcal{C}^{\infty}(G_{\Q}\setminus G(\A), \C ) $ l'espace des fonctions localement constantes et $ \mathcal{C}^{\infty}_c(G_{\Q}\setminus G(\A), \C ) $ son sous-espace des fonctions à support compact. On note aussi $ L^2(G_{\Q}\setminus G(\A), \C ) $ l'espace des fonctions de carré intégrable et $ L^2_0(G_{\Q}\setminus G(\A), \C ) $ son sous-espace consistant des fonctions cuspidales ainsi que $\mathcal{C}^{\infty}_{cusp} = \mathcal{C}^{\infty}_{c} \bigcap L^2_0$ et $\mathcal{C}^{\infty}_{(2)} = \mathcal{C}^{\infty} \bigcap L^2 $.
		
		L'inclusion d'espaces
		\[
		\mathcal{C}^{\infty}_{cusp}(G_{\Q} \ G(\A), \C ) \subset \mathcal{C}^{\infty}_c(G_{\Q} \ G(\A), \C ) \subset \mathcal{C}^{\infty}_{(2)}(G_{\Q} \ G(\A), \C ) \subset \mathcal{C}^{\infty}(G_{\Q} \ G(\A), \C )
		\]
		implique une application
		\[
		H^{\bullet}_{cusp}(\Sh, \mathcal{L}_{\xi}) \longrightarrow H^{\bullet}_c(\Sh, \mathcal{L}_{\xi}) \longrightarrow H^{\bullet}_{(2)}(\Sh, \mathcal{L}_{\xi}) \longrightarrow H^{\bullet}(\Sh, \mathcal{L}_{\xi})
		\]	
		qui est une injection d'après \cite{Bo74}.
		
		D'autre part on a 
		\[
		H^{\bullet}_{cusp}(\Sh, \mathcal{L}_{\xi}) = \bigoplus_{\pi} \pi_f \otimes H^{\bullet} (\Lie G(\R), K_{\infty}, \pi_{\infty}^{K_{\infty}} \otimes \mathcal{L}_{\xi} ),
		\]
		où $\pi = \pi_f \otimes \pi_{\infty}$ varie dans l'ensemble des classes d'isomorphisme de représentations cuspidales.
		
		Et de plus, on a
		\[
		H^{\bullet}_{(2)}(\Sh, \mathcal{L}_{\xi}) = \bigoplus_{\pi} \pi_f \otimes H^{\bullet} (\Lie G(\R), K_{\infty}, \pi_{\infty}^{K_{\infty}} \otimes \mathcal{L}_{\xi} )
		\]
		où $\pi$ varie dans le spectre discret de $L^2(Z_{\A}G_{\Q} \ G_{\A}, \omega)$ et $\omega$ est le caractère central de $\xi^{\vee}$.
		
		Comme $\xi$ est régulier alors $\pi_{\infty}$ est une série discrète de caractère infinitésimal $\xi$, en particulier $\pi_{\infty}$ est tempérée. 
		
		Puisque la composante en $v$ de $ \Pi_f $ est supercuspidale, on en déduit que si $ \pi_{\infty} \otimes \Pi_f $ est une représentation automorphe, elle est cuspidale. On a alors l'égalité
		\[
		H^d_{cusp}(\Sh, \mathcal{L}_{\xi}) [\Pi_f] \approx H^d_c(\Sh, \mathcal{L}_{\xi}) [\Pi_f] \approx IH^d(\Sh, \mathcal{IC}_{\xi})[\Pi_f].
		\]

	\end{proof}
	
	
	\begin{notation}
		Pour $n_1, \cdots, n_r \in \N$, notons 
		\[
		H = G \big( U^*(n_1) \times \cdots \times U^*(n_r) \big) = \big\{ (g_1, \cdots, g_r) \in G^*(n_1) \times \cdots \times G^*(n_r) | c(g_1) = \cdots = c(g_r) \big\},
		\]
		où $G^*(n_i)$ désigne le groupe de similitudes unitaires quasi-déployé en $n_i$ variables. De plus on a 
		\[
		\widehat{H} = \C^{\times} \times GL_{n_1}(\C) \times \cdots \times GL_{n_r}(\C).
		\]
	\end{notation}
	Tout d'abord, on voudrait donner une description des triplets endoscopiques elliptiques $(H_1, s, \eta)$ de $H$ au sens de \cite{Kot84} 7.3, 7.4.
	
	\begin{proposition}  (\cite{Mo} prop 2.3.1)
		Pour $i \in \{ 1, \cdots, r \}$, soient $n_i^{+}, n_i^{-} \in \N$ tels que $ n_i =  n_i^{+} + n_i^{-}$. Supposons que $n_1^- + \cdots + n_r^-$ est pair. Posons
		\[
		s = \big(1, \diag(\underbrace{1,\cdots,1}_{n_1^+}, \underbrace{-1, \cdots, -1}_{n_1^-}), \cdots, \diag(\underbrace{1,\cdots,1}_{n_r^+}, \underbrace{-1, \cdots, -1}_{n_r^-}) \big) \in \widehat{H}
		\] 	
		\[
		H_1 = G \big( U^*(n_1^+) \times U^*(n_1^-) \times \cdots \times U^*(n_r^+) \times U^*(n_r^-) \big) 
		\]
		et définissons 
		\[
		\eta : \widehat{H}_1 = \C^{\times} \times GL_{n_1^+}(\C) \times GL_{n_1^-}(\C) \times \cdots \times GL_{n_r^+}(\C) \times GL_{n_r^-}(\C) \longrightarrow \widehat{H} = \C^{\times} \times GL_{n_1}(\C) \times \cdots \times GL_{n_r}(\C)
		\]
		par la formule
		\[
		\eta \big((\lambda, g_1^+,g_1^-, \cdots,g_r^+, g_r^-)\big) = \big( \lambda, \diag(g_1^+, g_1^-), \cdots, \diag(g_r^+, g_r^-) \big).
		\]
		
		Alors $\big(H_1, s, \eta \big)$ est un triplet endoscopique elliptique pour $H$. De plus, les triplets endoscopiques elliptiques de $H$ définis par $\big((n_1^+, n_1^-), \cdots, (n_r^+, n_r^-) \big)$ et $\big((m_1^+, m_1^-), \cdots, (m_r^+, m_r^-) \big)$ sont isomorphes si et seulement si pour tout $i \in \{1, \cdots, r \}$, $(n_i^+, n_i^-) = (m_i^+, m_i^-)$ ou $(n_i^+, n_i^-) = (m_i^-, m_i^+)$. Inversement, tout triplet endoscopique elliptique pour $H$ est isomorphe à un des triplets définis ci-dessus. 
	\end{proposition}
	\begin{definition}
		Posons $\mathcal{E}^0(G)$ l'ensemble des triplets endoscopiques elliptiques $(H, s, \eta_0)$ pour $G$ tels que $H$ n'est pas une forme intérieure de $G$. Pour tout $(H, s, \eta_0) \in \mathcal{E}^0(G) $, fixons un $L$-morphisme $ \eta : \prescript{L}{}{H} \longrightarrow \prescript{L}{}{G}$ étendant $\eta_0$. On définit de même l'ensemble $\mathcal{E}^0(H)$ pour $H = G\big(U^*(n_1) \times \cdots U^*(n_r)\big)$. Posons $\mathcal{F}_G$ l'ensemble des suites $(e_1, \cdots, e_r)$ avec $r \in \N^*$ où $e_1 = (H_1, s_1, \eta_{1, 0}) \in \mathcal{E}^0(G)$ et pour $ i \in \{ 2, \cdots, r \} $, $e_i = (H_i, s_i, \eta_{i, 0}) \in \mathcal{E}^0(H_{i-1})$.
		
		Finalement, on pose 
		\[
		\eta_{\e} = \eta_1 \circ \cdots \circ \eta_r : \prescript{L}{}{H}_{\e} \longrightarrow \prescript{L}{}{G}
		\]
		le morphisme de $L$-groupes correspondant.
	\end{definition} 
	
	Considérons $(\e) = (e_1, \cdots, e_r) \in \mathcal{F}_G$. Supposons ensuite que $(e_1) = (H_1, s_1, \eta_1)$ est le triplet endoscopique elliptique défini par un couple $(n^{+}, n^{-})$ avec $n^{-}$ pair. On écrit
	\[
	H_{\e} = G \Big(U^*(n_1^+) \times \cdots \times U^*(n_r^+) \times U^*(n_1^-) \times \cdots \times U^*(n_s^-) \Big)
	\]
	où l'identification est faite de sorte que $\eta_2 \circ \cdots \circ \eta_r$ envoie $\widehat{U^*(n_1^+)} \times \cdots \times \widehat{U^*(n_r^+)}$ (resp. $\widehat{U^*(n_1^-)} \times \cdots \times \widehat{U^*(n_s^-)}$) dans $\widehat{U^*(n^+)}$ (resp. $\widehat{U^*(n^-)}$).
	
	Maintenant pour $p_1^+, \cdots, p_r^+, p_1^-, \cdots, p_s^- \in \N$ tels que $1 \leq p_i^+ \leq n_i^+$ et $1 \leq p_i^- \leq n_i^-$, on définit un cocaractère de $H_{\e}$ par la formule
	\[
	\mu = (\mu_{p_1^+} \times \cdots \times \mu_{p_r^+} \times \mu_{p_1^-} \times \cdots \times \mu_{p_s^-}) : \G_{m, \C} \longrightarrow H_{\e, \C}. 
	\]
	où $\mu_{p_i^*}$ est le cocaractère de signature $(p_i^*, n_i^* - p_i^*)$.
	\begin{definition}
		Supposons $G$ de signature $(p,n-p)$ à l'infini. Pour chaque $(\e)$ on note $M_{\e}$ l'ensemble des cocaractères $\mu = \mu_{p_1^+, \cdots, p_r^{+}, p_1^-, \cdots, p_s^-}$ de sorte que $ p = p_1^+ + \cdots + p_r^+ + p_1^- + \cdots + p_s^-$. Pour un tel cocaractère on note également $s(\mu) = (-1)^{p_1^- + \cdots + p_s^-}$.
	\end{definition}
	
	Pour chaque $(\e)$, on peut définir les constantes $l(\e)$, $\iota(\e)$ ainsi que $\iota'(\e)$. Ces constantes apparaîtront dans le théorème \ref{itm :  comp iso} ci-dessus mais pour notre but, nous n'aurons pas besoin de les calculer explicitement. Pour une définition précise, le lecteur pourra consulter \cite{Mo} - p.108.   
	\begin{definition}
		Soit $\pi_f = \bigotimes_p' \pi_p$ une représentation irréductible de $G(\A_f)$ de sorte que $\pi_f^{K} \neq 0$ pour $K \subset G(\A_f)$ un niveau et soit $\e \in \mathcal{F}_G$. Notons $R_{\e}(\pi_f)$ pour l'ensemble des classes d'équivalences de représentations irréductibles $\pi_{\e ,f} = \bigotimes_p' \pi_{\e ,f}$ de $H_{\e}(\A_f)$ telles que pour presque tout $p$ où $\pi_f$ et $\pi_{\e, p}$ sont non ramifiées, le morphisme $\eta_{\e} : \prescript{L}{}{H}_{\e} \longrightarrow \prescript{L}{}{G}$ envoie un paramètre de Langlands de $\pi_{\e ,p}$ sur celui de $\pi_p$.
	\end{definition}
	
	Étant donnés $(\e)$ et $f \in C_c^{\infty}(G(\Q_v))$, on note $f^{\e} = (((f^{H_1})^{H_2})\cdots)^{H_r} \in C_c^{\infty}(H_{\e}(\Q_v))$ un transfert endoscopique. Si $f \in C_c^{\infty}(G(\A_f))$ on notera de même $f^{\e} \in C_c^{\infty}(H_{\e}(\A_f))$. \\
	
	Soient $\xi$ une représentation algébrique de $G$ ainsi que $ K \subset G(\A_f) $ un niveau. Posons $ \mathcal{H}_K = \mathcal{H}_K(G(\A_f), K)$. Considérons les groupes de cohomologie suivant ($i \geq 0$)
	\[
	W_i = H^i(\Sh_K^*, \mathcal{IC}_{\xi}).
	\] 
	
	Ce sont des $L$-espaces vectoriels  ( où $L$ est une extension finie de $ \Q_{\ell} $, $ \ell \neq p $) de dimension finie qui possèdent une action de $ \mathcal{H}_K \times Gal(\overline{\Q} \slash \overset{\bullet}{E}) $ où $\overset{\bullet}{E}$ est le corps reflex de la variété de Shimura. Soit $ \iota : L \longrightarrow \C $ une immersion. Dans la suite, par abus de langage, on note également $ W_i $ la représentation $ \iota_{*}(W_i) $. Dans le cas qui nous intéresse, \textit{la signature à l'infini du groupe $G$ est $(1, n-1)$ avec $ n > 2 $, donc le corps reflex $\overset{\bullet}{E}$ est égale à $\overset{\bullet}{F}$}. Comme $W_i$ s'annule si $i$ est assez grand, on définit alors
	\[
	W = \sum_{i \geq 0} (-1)^iW_i
	\] 
	un objet dans le groupe de Grothendieck des représentations de $ \mathcal{H}_K \times Gal(\overline{\Q} \slash \overset{\bullet}{F}) $. On a une décomposition isotypique de $W$ comme $\mathcal{H}_K$-module
	\[
	W = \sum_{\pi_f} W(\pi_f) \otimes \pi_f^K,
	\]
	où la somme est prise sur l'ensemble des classes d'isomorphismes de représentations irréductibles $\pi_f$ de $G(\A_f)$ telles que $\pi_f^K \neq 0 $ et où $W(\pi_f)$ est une $\C$-représentation virtuelle de dimension finie de $Gal(\overline{\Q} \slash \overset{\bullet}{F})$. 
	
	Pour $\xi$ une représentation algébrique fixée de $G$ et pour $\pi_f$ une représentation irréductible de $G(\A_f)$, on peut définir une constante $c_G(\pi_f)$ laquelle est liée aux multiplicités des représentations automorphes ($c_G(\pi_f)$ dépend éventuellement de $\xi$). De même, il y a une constante $c_{\e}(\pi_{\e, f})$ associée à $\pi_{\e, f}$ dont la définition est liée aux multiplicités des représentations automorphes.
	
	Pour $p$ un nombre premier non ramifié, on pose $ \eta_{ \e, p} := \eta_{\e | \widehat{H}_{\e} \rtimes W_{\Q_p} } $. On écrit $ (\eta_{ \e, p})_{simple} : \prescript{L}{}{H}_{\e} \longrightarrow \prescript{L}{}{G} $ pour le $L$-morphisme défini dans \cite{Mo}, page $109$ troisième paragraphe. On note également $ \chi_{\e, p} $ le caractère défini dans le même paragraphe de loc.cit. 
	
	Considérons un nombre premier $p$ ainsi que $\mathcal{P}$ une place de $\overset{\bullet}{F}$ au-dessus de $p$. Notons $\Fr_{\mathcal{P}}$ un relèvement du Frobenius arithmétique. Le théorème suivant, dû à Sophie Morel, nous permet de calculer la trace de $\Fr_{\mathcal{P}}$ sur $W(\pi_f)$ lorsque $p$ est assez grand. 
	
	\begin{théorème} \phantomsection \label{itm :  comp iso} (Théorème 7.2.2 de \cite{Mo})
		Soit $\pi_f$ une représentation irréductible admissible de $G(\mathbb{A}_f)$ telle que $\pi_f^{K} \neq 0$. Il existe alors une fonction $f^{\infty} \in C_c^{\infty}(G(\mathbb{A}_f)) $ de sorte que pour presque tout nombre $p$ premier et pour tout $m \in \Z$,
		\begin{align*} \label{itm : sstar}
			\Tr(\Fr^m_{\mathcal{P}}, W(\pi_f)) =& (N \mathcal{P})^{-\frac{m(n-1)}{2}} c_G(\pi_f) \dim(\pi_f^K) \Tr (r_{\mu_G} \circ \varphi_{\pi_p}(\Fr^m_{\mathcal{P}})) \\
			&+ (N\mathcal{P})^{md/2} \sum_{\underline{e} \in \mathcal{F}_G} (-1)^{l(\underline{e})} \iota(\underline{e}) \sum_{\pi_{\e, f} \in R_{\e}(\pi_f)} c_{\e} (\pi_{\e, f}) \Tr(\pi_{\e, f})((f^{\infty})^{\e}) \\
			& \quad \quad \quad \quad \quad \quad \quad \sum_{\mu \in M_{\e}} (1 - (-1)^{s(\mu)} \cdot \dfrac{\iota'(\e)}{\iota(\e)}) \Tr (r_{ \mu} \circ \varphi_{\pi_{\e ,p} \otimes \chi_{\e, p}} (\Fr^m_{\mathcal{P}}) ) \tag{$\star$}
		\end{align*}
		où la somme est prise pour $\pi_{\e ,f}$ telle que $\pi_{\e ,p} \otimes \chi_{\e, p}$ est non ramifiée.
	\end{théorème}
	\begin{remarque} \phantomsection \label{itm : remarque}
		Les nombres premiers $p$ tels que la formule (\ref{itm : sstar}) du théorème \ref{itm :  comp iso} est vérifiée, sont ceux pour lesquels il existe un ensemble $ p \notin T $ avec $f = h_Tg^T$ où
		\begin{enumerate}
			\item[$\bullet$] $ T \supset \big\{ q \ | \ G_q$ est ramifié, $K_q$ n'est pas maximal$\big\} $.
			\item[$\bullet$] $ h \in \mathcal{H}(G(\A_f), K) $ est une fonction de la forme $h = h_T 1_{K^T}$ satisfaisant les propriétés décrites dans les lignes $14, 15$, p$.111$ de \cite{Mo} (i.e. $h$ sépare les représentations dans $R'$ de loc. cit.).
			\item[$\bullet$] $ g^T \in \mathcal{H}(G(\A_f^T), K^T) $ est une fonction satisfaisant les propriétés décrites dans les lignes $25 - 28$, p$.111$ de loc. cit. (i.e. $g^T$ sépare les représentations dans $R'_{\e}$ de loc.cit.).
			\item[$\bullet$] $ g_p^T = 1_{G(\Z_p)} $.
		\end{enumerate}
	\end{remarque}
	\subsection{Détermination de $W(\pi_f)_{\mathcal{P}}$ dans un cas particulier}
	Le but de ce paragraphe est de montrer un corollaire du théorème \ref{itm :  comp iso}.
	\begin{lemme}
		Soit $p$ un nombre premier tel que $K_p$ est hyper-spécial et $G$, $\pi$ sont non ramifiés en $p$, alors il existe $f^{\infty}$ telle que la formule (\ref{itm : sstar}) est vraie pour $p$. 
	\end{lemme}
	\begin{proof}
		Tout d'abord, considérons l'ensemble $R'$ des représentations irréductibles admissibles $\pi_f'$ de $G(\A_f)$ satisfaisant les conditions suivantes
		\begin{enumerate}
			\item[$\bullet$] $\pi_f' \ncong \pi_f $,
			\item[$\bullet$] $(\pi_f')^K \neq 0$,
			\item[$\bullet$] $W_{\lambda}(\pi_f') \neq 0 $ ou $c_G(\pi_f') \neq 0$.
		\end{enumerate}
		
		Comme $R'$ est fini et $\pi_p$ est non ramifiée, on peut choisir $h \in H(G(\A_f, K))$ telle que $h_p = 1_{G(\Z_p)}$ et
		\begin{enumerate}
			\item[$\bullet$] $\Tr(\pi_f(h)) = \Tr(\pi_f(1_K))$.
			\item[$\bullet$] $\Tr(\pi_f'(h)) = 0$ pour $\pi_f' \in R'$.
		\end{enumerate}
		
		Maintenant, il existe un ensemble $T_0$ de nombres premiers ne contenant pas $p$ mais qui satisfait toutes les conditions décrites dans lignes $16-20$ p$.111$ de \cite{Mo}.	
		
		Ensuite, pour chaque $\e \in \mathcal{F}_G$, on définit un ensemble $R_{\e}'$ comme dans le paragraphe 5, page 111 de loc. cit. Posons $T = T_0 \cup \{ p \}$, le même argument que dans la page $111$ de loc.cit. montre qu'il existe une fonction $g^T \in \mathcal{H}(G(\A_f^T), K^T)$ qui sépare les représentations dans $R_{\e}'$.
		
		Considérons $f^{\infty} = h_{T_0} g^{T_0}$ où $g^{T_0} = 1_{G(\Z_p)} g^T$, on voit que $T_0$ et $f^{\infty}$ satisfont  les trois premières propriétés dans la remarque \ref{itm : remarque} ainsi que $g^{T_0}_p = 1_{G(\Z_p)}$. On en déduit que la formule (\ref{itm : sstar}) est vérifiée pour $f^{\infty}$ et $p$.   
	\end{proof}
	
	\begin{corollaire} \phantomsection \label{itm: galois}
		Supposons $n$ impair et $G$ quasi-déployé en toutes les places finies et de signature $(1, n-1)$ à l'infini. Considérons une représentation automorphe $\pi = \pi_{\infty} \otimes \bigotimes_p' \pi_p$ de $G(\A)$ satisfaisant les conditions suivantes
		\begin{enumerate}
			\item[(i)] Le $L$-paramètre de $\pi$ est de la forme $ \varPsi = (\varPsi^n, \widetilde{\varPsi})$ où $ \varPsi^n = \varPsi^{n_1}_1 \boxplus \varPsi^{n_2}_2 $ tel que $\varPsi^{n_i}_i$ sont des paramètres globaux de dimension $n_i$ correspondant à des représentations cuspidales de $GL_{n_i}(\mathbb{A}_{F})$ (avec $i = 1,2$).
			\item[(ii)] $\pi_{\infty}$ est une série discrète de poids régulier.
			\item[(iii)] Il existe une place $q$ de $\Q$ qui est décomposée dans le corps quadratique imaginaire $ \overset{\bullet}{F} $ de sorte que $\pi_q$ est non ramifiée et son $L$-paramètre local $\phi_q$ est de la forme $ \phi^n_q = \chi_1 \oplus \chi_2 \oplus \cdots \oplus \chi_n $ où les $\chi_j$ sont deux à deux distincts. 
		\end{enumerate}
		
		Notons $\rho := W(\pi_f) $ la représentation de $Gal(\overline{\Q} \slash \overset{\bullet}{F})$ associée à $\pi_f$ dans la cohomologie de la variété de Shimura correspondant à $G$. Si $\dim \rho = \dim \varPsi^{n_i}_i = n_i$ ($ i = 1$ ou $i = 2$) alors pour tout nombre premier $p$ et toute place $\mathcal{P}$ de $\overset{\bullet}{F}$ au-dessus de $p$, on a
		\[
		\rho_{\mathcal{P}} =  (r_{\mu_i} \circ (\varPsi^{n_i}_i)_p)_{| W_{\mathcal{P}}} \otimes | \cdot |^{-\frac{n-1}{2}}.
		\]
		où $\mu_i = (1, n_i - 1)$ et $W_{\mathcal{P}}$ est le groupe de Weil de $\overset{\bullet}{F}_{\mathcal{P}}$.
	\end{corollaire}
	\begin{proof}
		On applique le théorème \ref{itm :  comp iso} pour $\pi_f$. Choisissons tout d'abord un sous-groupe compact suffisamment petit $K$ de $G(\A_f)$ de sorte que $\pi_f^K \neq 0$, en particulier on peut choisir $K$ tel que $K_q = G(\Z_q)$ car $\pi_q$ est non ramifiée.
		
		Comme le $L$-paramètre global $\varPsi$ de $\pi$ satisfait $ \varPsi^n = \varPsi^{n_1}_1 \boxplus \varPsi^{n_2}_2 $ l'ensemble des $\e \in \mathcal{F}_G$ tels que $R_{\e}(\pi_f)$ soit non nul contient un seul élément $\e = (e_1)$ où $e_1$ est le triplet endoscopique correspondant à $G(U^*(n_1) \times U^*(n_2))$.
		
		Par construction de $R_{\e}(\pi_f)$, pour un nombre premier $p$ non ramifié on a $ (\eta_{ \e, p})_{simple} \circ \varphi_{\pi_{\e ,p} \otimes \chi_{\e, p}} = \varphi_{(\pi_f)_p}$ (consulter les paragraphes avant le théorème \ref{itm :  comp iso} pour la définition de $ \eta_{ \e, p})_{simple} $) et de plus le changement de base quadratique de $\varphi_{(\pi_f)_p}$ s'écrit $ \varphi_{(\pi_f)_p}^n = (\varPsi^{n_1}_1)_p \oplus (\varPsi^{n_2}_2)_p $. D'autre part, la signature à l'infini $\mu_G$ est $(1, n-1)$ on voit alors que l'ensemble $M_{\e}$ contient exactement deux éléments $ \mu_1 = (1, n_1 - 1) \times (0, n_2)$ et $ \mu_2 = (0, n_1) \times (1, n_2 - 1)$. On en déduit que  	
		\[
		r_{\mu} \circ \varphi_{\pi_{\e ,p} \otimes \chi_{\e, p}} = \left\lbrace \begin{array}{cc}
			r_{\mu_1} \circ (\varPsi^{n_1}_1)_p & \quad  \mathop{\mathrm{si}}\nolimits \quad \mu = (1, n_1 - 1) \times (0, n_2) \\
			r_{\mu_2} \circ (\varPsi^{n_2}_2)_p & \quad \mathop{\mathrm{si}}\nolimits \quad \mu = (0, n_1) \times (1, n_2 - 1)
		\end{array} \right.
		\]
		où $\mu_1 = (1, n_1 - 1)$ et $\mu_2 = (1, n_2 - 1)$.
		
		D'après la description des représentations $ r_{\mu_G}, r_{\mu_1}, r_{\mu_2} $ dans la section \ref{itm : section notations}, on a l'identité suivante :
		\[
		r_{\mu_G} \circ \varphi_{\pi_p} = r_{\mu_1} \circ (\varPsi^{n_1}_1)_p \oplus r_{\mu_2} \circ (\varPsi^{n_2}_2)_p.
		\]
		
		D'autre part, on a
		\[
		(N \mathcal{P})^{- \frac{m(n-1)}{2}} \Tr (r_{\mu_G} \circ \varphi_{\pi_p}(\Fr^m_{\mathcal{P}})) = \Tr ( |\cdot|^{-\frac{n-1}{2}} \otimes r_{\mu_G} \circ \varphi_{\pi_p} (\Fr^m_{\mathcal{P}})) ).
		\]
		
		Finalement le théorème \ref{itm :  comp iso} induit l'identité suivante pour presque tout $p$ où $\pi_p$ est non ramifiée.
		\begin{equation} \phantomsection \label{itm : trace}
			\Tr(\Fr^m_{\mathcal{P}}, W(\pi_f)) = \alpha_p \cdot \Tr ( |\cdot|^{-\frac{n-1}{2}} \otimes r_{\mu_1} \circ (\varPsi^{n_1}_1)_p (\Fr^m_{\mathcal{P}}) ) + \beta_p \cdot \Tr ( |\cdot|^{-\frac{n-1}{2}} \otimes r_{\mu_2} \circ (\varPsi^{n_2}_2)_p (\Fr^m_{\mathcal{P}})).
		\end{equation}
		
		Montrons ensuite que pour presque tous $p_1$ et $p_2$ deux nombres premiers tels que $\pi_{p_1}$ et $\pi_{p_2}$ sont non ramifiées, on a $\alpha_{p_1} = \alpha_{p_2}$ et $\beta_{p_1} = \beta_{p_2}$.
		
		En effet, pour notre représentation $\pi_f$, on a
		\[
		\alpha_{p_i} = c_G(\pi_f ) \dim(\pi_f^K) \ + \ (-1)^{l(e_1)} \cdot \iota_{(e_1)} \cdot (1 - (-1)^{s(\overline{\mu}_1)} \cdot \dfrac{\iota'(e_1)}{\iota(e_1)}) \sum_{\pi_{e_1, f} \in R_{e_1}(\pi_f)} c_{e_1} (\pi_{e_1, f}) \Tr(\pi_{e_1, f})((f^{\infty})^{e_1}),
		\]
		où la somme est prise pour $\pi_{e_1 ,f}$ telle que $\pi_{e_1 ,p_i} \otimes \chi_{e_1, p_i}$ est non ramifiée.	
		
		Or lorsque $\pi_{e_1 ,p_i} \otimes \chi_{e_1, p_i}$ est ramifiée, $\Tr(\pi_{e_1, f})((f^{\infty})^{e_1})$ s'annule, on peut récrire la formule calculant $\alpha_{p_i}$ sous forme 
		\[
		\alpha_{p_i} = c_G(\pi_f ) \dim(\pi_f^K) \ + \ (-1)^{l(e_1)} \cdot \iota_{(e_1)} \cdot (1 - (-1)^{s(\overline{\mu}_1)} \cdot \dfrac{\iota'(e_1)}{\iota(e_1)}) \sum_{\pi_{e_1, f} \in R_{e_1}(\pi_f)} c_{e_1} (\pi_{e_1, f}) \Tr(\pi_{e_1, f})((f^{\infty})^{e_1})
		\]
		où la somme est prise pour \textit{toute} $\pi_{e_1 ,f}$. En particulier la formule calculant $\alpha_{p_i}$ ne dépend pas de $p_i$. On en déduit que $\alpha_{p_1} = \alpha_{p_2}$.
		
		Par le même argument, on a $\beta_{p_1} = \beta_{p_2}$ en utilisant
		\[
		\beta_{p_i} = c_G(\pi_f ) \dim(\pi_f^K) \ + \ (-1)^{l(e_1)} \cdot \iota_{(e_1)} \cdot (1 - (-1)^{s(\overline{\mu}_2)} \cdot \dfrac{\iota'(e_1)}{\iota(e_1)}) \sum_{\pi_{e_1, f} \in R_{e_1}(\pi_f)} c_{e_1} (\pi_{e_1, f}) \Tr(\pi_{e_1, f})((f^{\infty})^{e_1}).
		\]
		
		Considérons maintenant le premier $q$ décomposé. Puisque $\pi_{\infty}$ est une série discrète de plus haut poids régulier, la cohomologie de la variété de Shimura se concentre en degré moitié, en particulier $W(\pi_f)$ est une vraie représentation (à signe près). 
		D'après le lemme précédent appliqué en $q$, la condition $(iii)$ implique que $ r_{\mu_1} \circ (\varPsi^{n_1}_1)_q \oplus r_{\mu_2} \circ (\varPsi^{n_2}_2)_q $ est sans multiplicité, ce qui implique que $\alpha_q$ et $\beta_q$ sont de même signe. On peut supposer qu'ils sont positifs.
		
		Supposons que $\alpha_q$ et $\beta_q$ sont \textit{strictement} positifs. Cela implique donc que $\dim \rho \geq \dim \varPsi^{n_1}_1 + \dim \varPsi^{n_2}_2 = n$ ce qui contredit l'hypothèse $\dim \rho = \dim \varPsi^{n_i}_i$. Il y a alors exactement un coefficient non nul. 
		
		Maintenant, le fait que $n$ est impair couplé avec l'hypothèse $\dim \rho = \dim \varPsi^{n_i}_i$ implique alors que $\beta = 0$ si $i = 1$ et $\alpha = 0$ si $i = 2$. Autrement dit, on a 
		\[
		\Tr(\Fr^m_{\mathcal{P}}, W(\pi_f)) = \Tr ( |\cdot|^{-\frac{n-1}{2}} \otimes r_{\mu_i} \circ (\varPsi^{n_i}_i)_p (\Fr^m_{\mathcal{P}}) ).
		\]
		(ici $i = 1$ si $\dim \rho = \dim \varPsi^{n_1}_1$ et $i = 2$ si $\dim \rho = \dim \varPsi^{n_2}_2$).
		
		Comme l'égalité ci-dessus est vraie pour presque toute place $\mathcal{P}$ de $\overset{\bullet}{F}$, le théorème de densité de Chebotarev implique que 
		\[
		\rho_{\mathcal{P}} =  (r_{\mu_i} \circ (\varPsi^{n_i}_i)_p)_{| W_{\mathcal{P}}} \otimes | \cdot |^{-\frac{n-1}{2}}
		\]
		pour toute place $\mathcal{P}$.
	\end{proof}																									 
	
	\begin{remarque} \phantomsection \label{itm : remar}
		Le corollaire est encore valable lorsque $ \varPsi $ est un paramètre simple générique (ce qui équivaut à dire que $\dim \varPsi^{n_1}_1 = 0$ ou $ \dim \varPsi^{n_2}_2 = 0 $).
	\end{remarque}

	\subsection{Preuve du théorème principal}
	Le but de ce paragraphe est de prouver le théorème principal de l'introduction. D'après le théorème \ref{itm : suite spectrale}, on a une suite spectrale reliant la partie supercuspidale des espaces de Rapoport-Zink à celle de la strate basique $\Sh_{\infty}(basic)$ des variétés de Shimura $\Sh_{\infty}$ du \ref{itm : Sh} laquelle est déterminée par le corollaire \ref{itm: galois}. Afin d'identifier les $\sigma_{\pi_p, \pi_p'}$ avec les $r_{\mu} \circ \varphi_{i, F_p}$, nous avons besoin d'une part de l'hypothèse $\dim \rho = \dim \varPsi_i$ dans \ref{itm: galois} et d'autre part de calculer $\dim \sigma_{\pi_p, \pi_p'}$. Dans les deux cas, ce calcul repose sur un problème de comptage de caractères du groupe centralisateur des L-paramètres à l'infini. Plus précisément, la preuve est découpée en $4$ étapes.
	\begin{enumerate}
		\item[$\bullet$] Dans la première étape, on montre les points $(i)$ et $(ii)$ du théorème et on obtient des contraintes sur la dimension de $\sigma_{\pi_p, \pi_p'}$.
		\item[$\bullet$] Dans la deuxième étape, on établit une relation entre $\sigma_{\pi_p, \pi_p'}$ et une représentation galoisienne bien choisie de la variété de Shimura puis on ramène le problème du calcul de la dimension de la représentation au problème de comptage des caractères du groupe centralisateur à l'infini.
		\item[$\bullet$] Ensuite, on utilise des contraintes sur la dimension de $\sigma_{\pi_p, \pi_p'}$ obtenues avant pour calculer les caractères du groupe centralisateur à l'infini.
		\item[$\bullet$] Enfin on utilise le corollaire \ref{itm: galois} afin de calculer $\sigma_{\pi_p, \pi_p'}$.
	\end{enumerate}

	Rappelons tout d'abord la construction des données globales à partir des données locales.
	\begin{proposition}(\cite{Far04}) \phantomsection \label{itm: globale -> locale}
		Soit une donnée locale $\mathcal{D}_{\Q_p} = (F_p, *, V, \langle \cdot | \cdot \rangle, \mu, b)$ de type PEL non ramifiée simple sur une extension finie de $\Q_p$. Supposons que $[F_p : \Q_p] = 2d$ n'est pas un multiple de $4$, alors:
		\begin{enumerate}
			\item[$\bullet$] il existe un corps CM de la forme $\overset{\bullet}{F} = K \mathcal{K} $ avec $\mathcal{K}$ un corps quadratique imaginaire de sorte que $p$ reste inerte dans $\overset{\bullet}{F}$ et $\overset{\bullet}{F}_p = F_p$.
			\item[$\bullet$] il existe une donnée globale $ \mathcal{D} = ( \overset{\bullet}{F}, B, *, V, \langle \cdot | \cdot \rangle, h, \overset{\bullet}{G}) $ de type PEL, un plongement $ \nu : \overline{\Q} \hookrightarrow \overline{\Q}_p $ tels que via $\nu$, $\mathcal{D}$ induise la donnée locale $\mathcal{D}_{\Q_p}$.
		\end{enumerate}  
		De plus  
		\begin{enumerate}
			\item[-] Pour tout $n$, on peut imposer que  $\End_B(V)$ est une algèbre à division qui est en toute place finie soit déployée, soit une algèbre à division. 
			\item[-] Pour $n$ impair ou $n \equiv 2$ modulo $4$, on peut imposer que $\End_B(V)$ est une algèbre simple qui est déployée en toutes places finies.
		\end{enumerate}
	\end{proposition}
	\begin{proof}
		On raisonne comme dans la proposition 10.1.3 de \cite{Far04}. Grâce à la proposition 10.1.1 dans loc. cit
		on peut supposer que $\End_B(V)$ satisfait les conditions annoncées ci-dessus. 
	\end{proof}
	
	\begin{lemme} \phantomsection \label{itm: dégénère} 
		Soit $\Pi_p$ une représentation supercuspidale de $\J_b(\Q_p)$. Alors pour tout $t > 0$ on a
		$$ \Ext^t_{\J_{b}(\Q_p)} \left( H^q_c(\mathcal{M}_{K_p}, \overline{\Q}_{\ell}(n-1)), \Pi_p \right) = 0 $$
	\end{lemme}
	\begin{proof}
		
		On pose $\Delta = \hom_{\Z}(X^*(G)_{\Q_p}, \Z)$. Comme $b$ est basique, $\J_b(\Q_p)$ est une forme intérieure de $G(\Q_p)$, tout $\chi \in X^*(G)_{\Q_p} $ se transfère à $\J_b(\Q_p)$ en un $\tilde{\chi} \in X^*(\J_b (\Q_p) )_{\Q_p}$. 
		
		Notons $ \J_b^1 (\Q_p) = \bigcap_{\chi \in X^*(G)_{\Q_p}} ker \vert \tilde{\chi} \vert $ où $ \vert \tilde{\chi} \vert : \J_b (\Q_p) \ \longrightarrow $, $ x \longmapsto \ v_p(\tilde{\chi}(x)) $
		; en particulier, $\J_b^1 (\Q_p)$ a un centre compact.
		
		Il y a une application $ \pi_2 : \mathcal{M} \longrightarrow \Delta $ qui est essentiellement la hauteur de la rigidification $\rho$ (\cite{RZ96}, 3.52). On note encore $\Delta' \subset \Delta $ l'image de $\pi_2$. On a donc une décomposition: $ \mathcal{M}_{K_p} = \coprod_{i \in \Delta'} \mathcal{M}_{K_p}^{(i)} $ où $ \mathcal{M}_{K_p}^{(i)} = \pi_2^{-1}(i)$. On en déduit
		$$ H^q_c(\mathcal{M}_{K_p}, \overline{\Q}_{\ell}(n-1)) = \sum_{\overline{i} \in \Delta' / \J_b (\Q_p) } c-\Ind_{\J^1_{b} (\Q_p) }^{\J_b(\Q_p)}(H_c^q (\mathcal{M}_{K_p}^{(i)}, \overline{\Q}_{\ell}(n-1))). $$
		
		En utilisant la dualité de Frobenius, on dispose alors d'un isomorphisme de foncteurs:
		$$ \hom_{\J_{b}(\Q_p)} \left( H^q_c(\mathcal{M}_{K_p}, \overline{\Q}_{\ell}(n-1)), \bullet \right) \simeq \sum_{\overline{i} \in \Delta' / \J_b (\Q_p)}  \hom_{\J^{1}_{b}} \left( H^q_c(\mathcal{M}^{(i)}_{K_p}, \overline{\Q}_{\ell}(n-1)), \Res_{\J_{b}^1(\Q_p)}^{\J_b (\Q_p) } \bullet \right). $$
		
		Il en résulte des isomorphismes:
		$$ \Ext^t_{\J_b(\Q_p)} \left( H^q_c(\mathcal{M}_{K_p}, \overline{\Q}_{\ell}(n-1)), \Pi_p \right) \simeq \sum_{\overline{i} \in \Delta' / \J_b (\Q_p)} \Ext^t_{\J^{1}_b (\Q_p)} \left( H^q_c(\mathcal{M}^{(i)}_{K_p}, \overline{\Q}_{\ell}(n-1)), \Res_{\J_b^1(\Q_p)}^{\J_b (\Q_p)} \Pi_p \right). $$
		
		D'autre part la représentation $\Pi_p$ est cuspidale, donc $\Res_{\J_b^1 (\Q_p) }^{\J_b (\Q_p) } \Pi_p$ est une représentation finie (les coefficients matriciels sont des fonctions à support compact), donc $\Res_{\J_b^1 (\Q_p) }^{\J_b (\Q_p) } \Pi_p$ est un objet projectif dans la catégorie des représentations lisse de $\J^1_b (\Q_p) $. On obtient alors que pour tout $t > 0$
		$$ \Ext^t_{\J_b (\Q_p) } \left( H^q_c(\mathcal{M}_{K_p}, \overline{\Q}_{\ell}(n-1)), \Pi_p \right) = 0. $$
	\end{proof}
	
	Soit $ c: G(\Q_p) \longrightarrow \Q_p^{\times} $ le caractère du facteur de similitudes. Pour $ b $ non basique, le groupe $ \J_b (\Q_p) $ est une forme intérieure d'un sous-groupe de Levi $ \mathrm{M}_b (\Q_p) $ de $ G(\Q_p) $. Ensuite, le caractère du facteur de similitudes $ c $ restreint à $ \mathrm{M}_b(\Q_p) $ peut être transféré à $ \J_b(\Q_p) $. Donc par abus de langage, on note aussi $ c $ le caractère correspondant de $ \J_b(\Q_p) $. Le lemme suivant nous permet de comprendre le comportement de la cohomologie des espaces de Rapoport-Zink lorsque la représentation $ \pi $ est tordue par un caractère non ramifié :
	
	\begin{lemme} \phantomsection \label {itm: torsion non ramifiée}
		Soit $ (F_p, *, V, \langle \cdot | \cdot \rangle, \mu, b) $ une donnée de Rapoport-Zink PEL unitaire non ramifiée, supposons que $ \omega: \Q_p^{\times} \longrightarrow \overline{\Q}^{\times}_{\ell} $ est un caractère non ramifié et $ \pi $ est une représentation irréductible de $ \J_b(\Q_p) $. On a alors un isomorphisme dans $ Groth (G(\Q_p) \times W_{E_p}) $ entre
		\[
		\displaystyle \sum_{t,q} (-1)^{t+q} \mathop{\mathrm{lim}}_{\overrightarrow{K}_p} \Ext^t_{\J_{b}(\Q_p)} \left( H^q_c(\mathcal{M}_{K_p}, \overline{\Q}_{\ell}(n-1)), \pi \otimes (\omega \circ c) \right)
		\]
		et 
		\[
		\displaystyle \sum_{t,q} (-1)^{t+q} \mathop{\mathrm{lim}}_{\overrightarrow{K}_p} \Ext^t_{\J_{b}(\Q_p)} \left( H^q_c(\mathcal{M}_{K_p}, \overline{\Q}_{\ell}(n-1)),\pi \right) \otimes (\omega \circ c) \otimes (\omega \circ Art^{- 1}_{E_p})
		\]
		où $E_p$ est le corps de définition de $ \mu $.
	\end{lemme}
	
	\begin{proof}
		Ce lemme est un analogue de \cite[Lemme 4.9]{Shin} et la même preuve s'applique dans notre situation. Nous donnons donc brièvement une idée de la preuve.
		
		Définissons un caractère $ \chi $ de $ \J_b (\Q_p) \times G (\Q_p) \times W_{E_p} $ par
		\[
		\chi: = (\omega \circ c) \otimes (\omega \circ c) \otimes (\omega \circ Art^{- 1}_{E_p}).
		\]
		
		Ensuite, nous prouvons qu'il existe un isomorphisme de $ \overline{\Q}_{\ell}$-espaces vectoriels
		\[
		H_c^j (\mathcal{M}_{K_p}, \overline{\Q}_{\ell}) \simeq H_c^j (\mathcal{M}_{K_p}, \overline{\Q}_{\ell}) \otimes \chi
		\]
		compatible avec l'action de $ \J_b(\Q_p) \times (K_p \backslash G (\Q_p) / K_p) \times W_{E_p} $.
		
		Notons qu'il y a une application $ \pi_2 : \mathcal {M}_ {K_p} \longrightarrow \Delta $ et de plus il existe une manière naturelle de définir une action de $ \J_b(\Q_p) \times G(\Q_p) \times W_{E_p} $ sur $ \Delta $ telle que l'application $ \pi_2 $ soit équivariante par rapport à $ \J_b(\Q_p) \times (K_p \backslash G(\Q_p) / K_p) \times W_{E_p} $ (\cite[remarque 2.6.11]{Far04}).
		
		On peut prouver le lemme en utilisant le fait que $ \chi $ agit trivialement sur $ (\J_b(\Q_p) \times (K_p \backslash G(\Q_p) / K_p) \times W_{E_p})^1 $ et
		\[
		\mathop{\mathrm{lim}}_{\overrightarrow {K_p}} H_c^j (\mathcal{M}_{K_p}, \overline{\Q}_{\ell}) \simeq c-\Ind^{\J_b (\Q_p) \times (K_p \backslash G(\Q_p) / K_p) \times W_{E_p}}_{(\J_b(\Q_p) \times (K_p \backslash G(\Q_p) / K_p ) \times W_{E_p})^1} \Big(\mathop{\mathrm{lim}}_{\overrightarrow{K_p}} H_c^j(\mathcal{M}^{(0)}_{K_p}, \overline{\Q}_{\ell}) \Big)
		\]
		où $ \mathcal{M}^{(0)}_{K_p} $ est l'image inverse de $ 0 $ par $ \pi_2 $ et $ (\J_b (\Q_p) \times (K_p \backslash G (\Q_p ) / K_p) \times W_{E_p})^1 $ est le sous-groupe de $ \J_b (\Q_p) \times (K_p \backslash G (\Q_p) / K_p) \times W_{E_p} $ qui agit trivialement sur $ \Delta $.
		
	\end{proof}
	
	\begin{demn}
		Soit $\mathcal{D} = ( \overset{\bullet}{F}, B, *, V, \langle \cdot | \cdot \rangle, h, \overset{\bullet}{G})$ une donnée globale de type PEL globalisant la donnée locale de sorte que $\End_B(V)$ est une algèbre simple qui est déployée en toutes places finies comme dans la proposition \ref{itm: globale -> locale}. Soit $\Sh$ la variété de Shimura associée: le groupe $\overset{\bullet}{G}(\Q_p) = G(\Q_p)$ est le groupe des similitudes unitaires quasi-déployé en $n$ variables. En particulier, $\Sh$ est de signature $(1, n-1)$ à l'infini. 
		
		On note $\mathcal{M} (\mathcal{D}_{\Q_p}, b)$ l'espace de Rapoport-Zink associé à la donnée locale.
		
		Soit $\phi$ une classe d'isogénie intervenant dans la strate basique et $ I := (I^{\phi})$ le groupe réductif associé. On sait que $I(\R)$ est la forme compacte modulo le centre de $\overset{\bullet}{G}(\R)$, que $I(\Q_p) = \J_b(\Q_p) $ et que $ I(\A_f^p) = \overset{\bullet}{G}(\A_f^p) $. D'après la proposition \ref{itm : suite spectrale} il y a une suite spectrale $\overset{\bullet}{G}(\A_f) \times W_{F_p}$ équivariante:
		\begin{equation} \phantomsection \label{itm : 1ss}
			E_2^{tq} = | \ker^1 (\Q, \overset{\bullet}{G}) | \sum_{\substack{\Pi \in \mathcal{A}(I) \\ \Pi_{\infty} = \breve{\xi}}} \left( \Ext^t_{\J_{b}(\Q_p)} \left( H^q_c(\mathcal{M}), \Pi_p \right)_{cusp}  \right) \otimes (\Pi^{\infty, p}) \Longrightarrow \left( H^{t+q}_c(\Sh, \mathcal{L}_{\xi}) \right)_{p-cusp}
		\end{equation}	
		où on a noté $ \Ext^t_{\J_{b}(\Q_p)}\left( H^q_c(\mathcal{M}), \Pi_p \right) := \displaystyle \mathop{\mathrm{lim}}_{\overrightarrow{K}} \Ext^t_{\J_{b}(\Q_p)} \left( H^q_c(\mathcal{M}_{K_p}, \overline{\Q}_{\ell}(n-1)), \Pi_p \right) $ pour alléger les notations et où $\text{p-cusp}$ signifie que la composante en $p$ est supercuspidale.
		
		On choisit une représentation $\xi$ de dimension finie de $I(\C)$ qui est de poids régulier et suffisamment régulier au sens de \ref{itm: suffisament régulier}. Soit $\Pi_{(\xi)}(\overset{\bullet}{G}(\R))$ le $L$-paquet de séries discrètes de $\overset{\bullet}{G}(\R)$ cohomologiques pour $\xi$.
		
		Considérons $ \varphi : W_{\Q_p} \times SL_2(\C) \longrightarrow \prescript{L}{}G $ un paramètre discret qui est trivial sur $SL_2(\C)$. D'après la proposition \ref{itm : C.Moeglin}, le paquet $\Pi_{\varphi}(G(\Q_p))$ ne contient que des représentations supercuspidales. On peut exprimer la restriction de $\varphi$ sur $W_{F_p}$ sous la forme
		\[
		\varphi_{F_p} = \varphi^{n_1}_1 \oplus \varphi^{n_2}_2 \oplus \cdots \oplus \varphi^{n_r}_r,
		\]
		où les $\varphi^{n_i}_i$ sont des paramètres simples de $GL_{n_i}(F_p)$. 
		
		Puisqu'on utilisera dans la suite le théorème de classification des représentations automorphes pour $I$ et $\overset{\bullet}{G}$ (c.f \ref{itm: global}), il faut étendre $I$ et $\overset{\bullet}{G}$ en des formes intérieures pures. Pour ce faire, utilisons la proposition \ref{itm : forme intérieure pure}.
		
		Comme $\overset{\bullet}{G}(\R)$ est de signature $(1, n-1)$, son invariant est alors $a^G_{\infty} = \left[ \dfrac{n}{2} \right] + (n-1) \ (\text{mod} \ 2)$. De même $I(\R)$ est de signature $(0,n)$, son invariant est $a^I_{\infty} = \left[ \dfrac{n}{2} \right] + n \ (\text{mod} \ 2)$.
		
		De plus pour $v \neq p$, $\overset{\bullet}{G}(\Q_v)$ est quasi-déployé. Or $n$ étant impair, il y a donc deux manières d'étendre $\overset{\bullet}{G}(\Q_v)$ en une forme intérieure pure et on choisit alors la manière dont l'invariant vérifie $ a_{v} = 0 \ \text{mod} \ 2 $. De même, il y a deux manières d'étendre $\overset{\bullet}{G}(\Q_p) = G(\Q_p)$ en une forme intérieure pure et on choisit alors l'unique manière de sorte que $a^G_p + \left[ \dfrac{n}{2} \right] + (n-1) = 0 \ (\text{mod} \ 2)$.   
		
		Pour $I$ on procède de même et en particulier on étend $\J_b(\Q_p)$ en une forme intérieure de sorte que $ a^I_p + \left[ \dfrac{n}{2} \right] + n = 0 \ (\text{mod} \ 2)$.
		\begin{remarque} \phantomsection \label{itm : infini}
			On a $ a^I_p + a^G_p = a^I_{\infty} + a^G_{\infty} = 1 \ \text{mod} \ 2 $. 
		\end{remarque}
		
		\textbf{Étape 1 : Obtenir des contraintes sur $\dim \sigma_{\pi_p, \pi_p'} $.}
		
		Considérons une représentation supercuspidale $\widetilde{\pi}_p $ dans $\Pi_{\varphi}(\J_b(\Q_p))$. D'après \cite[théo. 5.7]{Shin1}, il existe une représentation automorphe $ \widetilde{\Pi} $ de $I(\A)$ telle que $ \widetilde{\Pi}_{\infty} = \breve{\xi} $ et $ \widetilde{\Pi}_{w_0} $ est supercuspidale pour une place $w_0$ décomposée et que $ \widetilde{\Pi}_p = \widetilde{\pi}_p $ (Plus précisément $\widetilde{\Pi}_p$ et $\widetilde{\pi}_p$ sont dans la même classe d'équivalence d'inertie; pourtant, grâce au lemme \ref{itm: torsion non ramifiée}, il suffit de considérer une représentation dans chaque classe d'équivalence d'inertie. Nous pouvons donc supposer que $\widetilde{\Pi}_p = \widetilde{\pi}_p$). En particulier, le $A$-paramètre $\widetilde{\varPsi}$ de $\widetilde{\Pi}$ est simple et on obtient son groupe centralisateur
		\[
		S^{\natural}_{\widetilde{\varPsi}} = \Z/2 \Z.
		\] 
		
		Par la suite, on va essayer d'appliquer \ref{itm: galois} et \ref{itm : remar} pour $\widetilde{\varPsi}$. En prenant la partie $ \widetilde{\Pi}^{\infty, p}$-isotypique de la suite spectrale (\ref{itm : 1ss}) on a une suite spectrale $G(\Q_p) \times W_{F_p}$-équivariante (remarquons que comme $ \widetilde{\Pi}_{\infty} = \breve{\xi} $, la condition $ \Pi_{\infty} = \breve{\xi} $ est une conséquence de la condition $ \Pi^p = (\widetilde{\Pi})^p $)
		\begin{equation*}
			E_2^{tq} = | \ker^1 (\Q, \overset{\bullet}{G}) | \sum_{\substack{\Pi \in \mathcal{A}(I) \\ \Pi^{p} = (\widetilde{\Pi})^{p} }} \left( \Ext^t_{\J_{b}(\Q_p)} \left( H^q_c(\mathcal{M}), \Pi_p \right)_{cusp}  \right) \Longrightarrow \left( H^{t+q}_c(\Sh, \mathcal{L}_{\xi}) \right)_{p-cusp}[\widetilde{\Pi}^{\infty, p}].
		\end{equation*}
		
		Puisque le paquet $\Pi_{\varphi}(\J_b(\Q_p))$ ne contient que des représentations supercuspidales, le lemme \ref{itm: dégénère} implique que $\Ext^t_{\J_{b}(\Q_p)} \left( H^q_c(\mathcal{M}_{K_p}, \overline{\Q}_{\ell}(n-1)), \Pi_p \right) = 0$ si $t > 0$ et la suite spectrale ci-dessus  dégénère en $E_2$. On obtient donc des isomorphismes:
		\begin{equation} \phantomsection \label{itm: 9}
			| \ker^1 (\Q, \overset{\bullet}{G}) | \sum_{\substack{\Pi \in \mathcal{A}(I) \\ \Pi^{p} = (\widetilde{\Pi})^{p} }} \left(  \hom_{\J_b(\Q_p)} \left( H^i_c(\mathcal{M}), \Pi_p \right)_{cusp}  \right) = \left( H^{i}_c(\Sh, \mathcal{L}_{\xi}) \right)_{p-cusp}[\widetilde{\Pi}^{\infty, p}].
		\end{equation}
		
		D'autre part, la proposition \ref{itm : comparer} couplée avec la formule de Matsushima pour la $L^2$-cohomologie, généralisée par Borel et Casselman (\cite{BC}), nous donne une décomposition:
		\begin{align} \phantomsection \label{itm : 10}
			H^{n-1}_c(\Sh, \mathcal{L}_{\xi})_{\text{p-cusp}}[\widetilde{\Pi}^{\infty, p}] &= | \ker^1 (\Q, \overset{\bullet}{G}) | \sum_{\substack{\Pi \in \mathcal{A}(\overset{\bullet}{G})_{\xi} \\ \Pi_p \ \text{supercuspidale} \\ \Pi^{\infty, p} = \widetilde{\Pi}^{\infty, p} } } \Pi_p \otimes \rho (\Pi^{\infty}). \\
			H^{i}_c(\Sh, \mathcal{L}_{\xi})_{\text{p-cusp}}[\widetilde{\Pi}^{\infty, p}] &= 0 \quad \quad (i \neq n-1) .
		\end{align}
		où $\rho (\Pi^{\infty})$ est une représentation continue de dimension finie de $Gal(\overline{\overset{\bullet}{F}}/ \overset{\bullet}{F} )$ et $\mathcal{A}(\overset{\bullet}{G})_{\xi}$ est l'ensemble des représentations automorphes de $\overset{\bullet}{G}$ cohomologiques pour $\xi$. De plus on a 
		\begin{equation} \phantomsection \label{itm : dim1}
			\dim \rho (\Pi^{\infty})  = \sum_{\pi_{\infty}} m(\pi_{\infty} \otimes \Pi^{\infty}) \dim H^{n-1} (\Lie \overset{\bullet}{G}(\R), K, \pi_{\infty} \otimes \xi)	
		\end{equation}
		où $\pi_{\infty}$ varie dans le paquet cohomologique $ \Pi_{(\xi)}(\overset{\bullet}{G}(\R)) $.
		
		Puisque $\xi$ est de poids régulier et $\overset{\bullet}{G}(\R)$ est le groupe de similitudes unitaires de signature $(1, n-1)$, d'après le corollaire VI.2.7 de \cite{HT01}, on en déduit qu'il y a exactement $n$ représentations $\pi_j$ ($ j \in \{1, \cdots, n \} $) cohomologiques pour $\xi$. On a également :
		\begin{equation} \phantomsection \label{itm : dim2}
			\dim H^{n-1} (\Lie \overset{\bullet}{G}(\R), K, \pi_j \otimes \xi) = 1.
		\end{equation}
		
		On en déduit que \textit{$\sigma_{\pi_p, \pi_p'}^i = 0$ si $ i \neq n-1 $} de sorte que
		\begin{equation*} \label{itm: 1ss galois}
			\sum_{\substack{\Pi \in \mathcal{A}(I) \\ \Pi^p = \widetilde{\Pi}^p}} \left(\hom_{\J_b(\Q_p)} \left( H^{n-1}_c(\mathcal{M}), \Pi_p \right)_{p-cusp}  \right) = \sum_{\substack{\Pi' \in \mathcal{A}(\overset{\bullet}{G}) \\ (\Pi')^{\infty, p} = \widetilde{\Pi}^{\infty, p} } } \Pi'_p \otimes \rho((\Pi')^{\infty})_p 	
		\end{equation*}
		
		En écrivant  
		$ \displaystyle \hom_{\J_b(\Q_p)} \left( H^{n-1}_c(\mathcal{M}), \Pi_p \right)_{cusp}   = \sum_{\widetilde{\Pi}_p'} \widetilde{\Pi}_p' \otimes \sigma_{\Pi_p, \widetilde{\Pi}_p'} $	où $\widetilde{\Pi}_p'$ parcourt l'ensemble des classes d'équivalences de représentations supercuspidales de $G(\Q_p)$, l'égalité ci-dessus s'écrit sous la forme
		\begin{equation} \phantomsection \label{itm : ...}
			\sum_{\substack{\Pi \in \mathcal{A}(I) \\ \Pi^p = \widetilde{\Pi}^p}} \left( \sum_{\widetilde{\Pi}_p'} \widetilde{\Pi}_p' \otimes \sigma_{\Pi_p, \widetilde{\Pi}_p'} \right) = \sum_{\substack{\Pi' \in \mathcal{A}(\overset{\bullet}{G}) \\ (\Pi')^{\infty, p} = \widetilde{\Pi}^{\infty, p} } } \Pi'_p \otimes \rho((\Pi')^{\infty})_p 	
		\end{equation}
		
		Or si $\Pi' \in \mathcal{A}(\overset{\bullet}{G})$ et $ (\Pi')^{\infty, p} = \widetilde{\Pi}^{\infty,p} $, on en déduit que $\Pi'_p$ et $\widetilde{\Pi}_p$ sont dans le même paquet, à savoir le paquet $\Pi_{ \varphi }(G(\Q_p))$. De même on voit que $\Pi_p \in \Pi_{ \varphi }(\J_b(\Q_p))$. On en déduit que $\sigma_{\Pi_p, \widetilde{\Pi}_p'} = 0$ si $\widetilde{\Pi}_p' \notin \Pi_{ \varphi } (G(\Q_p))$. 
		
		Or, le groupe centralisateur global $S_{\widetilde{\varPsi}}^{\natural}$ est $\Z / 2\Z$, en particulier $ S_{\widetilde{\varPsi}}^{\natural} \simeq Z(\widehat{\overset{\bullet}{U^*}})^{\Gamma} $ où $ \Gamma $ est le groupe de Galois, $ \overset{\bullet}{U} $ est le groupe unitaire noyau du facteur de similitude de $\overset{\bullet}{G}$ et $\overset{\bullet}{U^*}$ est sa forme intérieure quasi-déployée. On en déduit que $ \Pi_{\varPsi} (I, \varrho, \epsilon)  = \Pi_{\varPsi}(I, \varrho)$ car la condition sur $\epsilon$ disparaît (consulter \ref{itm : conditions sur le centre} et \ref{itm : paquets}). D'après le théorème \ref{itm: global}, on voit que $\breve{\xi} \otimes \pi_p \otimes \widetilde{\Pi}^{\infty, p}$ est toujours une forme automorphe lorsque $\pi_p$ varie dans $\Pi_{\varphi}(\J_b(\Q_p))$. L'égalité (\ref{itm : ...}) se réécrit donc sous la forme
		
		\[
		\sum_{\pi_p \in \Pi_{\varphi}(\J_b(\Q_p))} \sum_{\pi_p' \in \Pi_{\varphi}(G(\Q_p))} \pi_p' \otimes \sigma_{\pi_p, \pi_p'} = \sum_{\pi_p' \in \Pi_{\varphi}(G(\Q_p))} \pi_p' \otimes \rho(\pi_p' \otimes \widetilde{\Pi}^{\infty, p})_p.
		\]
		
		En particulier pour $\pi_p' \in \Pi_{\varphi}(G(\Q_p))$ fixée, on a 
		\[
		\sum_{\pi_p \in \Pi_{\varphi}(\J_b(\Q_p))} \sigma_{\pi_p, \pi_p'} = \rho(\pi_p' \otimes \widetilde{\Pi}^{\infty, p})_p.
		\]
		
		D'après le résultat de multiplicité $1$ (théorème \ref{itm: global} et proposition \ref{itm: de U à GU}), on voit que $m(\pi_j \otimes \pi_p' \otimes \widetilde{\Pi}^{\infty, p})$ est soit nulle soit égale à $1$. 
		Comme au-dessus, le théorème \ref{itm: global} implique que $\pi_i \otimes \pi_p' \otimes \widetilde{\Pi}^{\infty, p}$ est toujours une forme automorphe lorsque $\pi_p'$ varie dans $\Pi_{\varphi}(G(\Q_p))$ et $i$ varie dans $\{ 1, \cdots, n \}$.
		
		Cela implique que $\dim \rho(\pi_p' \otimes \widetilde{\Pi}^{\infty, p}) = n$. D'après le corollaire \ref{itm: galois} et la remarque \ref{itm : remar}, on a alors
		\[
		\rho(\pi_p' \otimes \widetilde{\Pi}^{\infty, p})_p = \left( r_{\mu} \circ \varphi_{|F_p} \right)  \otimes | \cdot |^{-\frac{n-1}{2}}.
		\]
		
		Autrement dit on a
		\begin{equation} \phantomsection \label{itm:dim galois}
			\sum_{\pi_p \in \Pi_{\varphi}(J_b(\Q_p))} \sigma_{\pi_p, \pi_p'} = \left( r_{\mu} \circ \varphi_{|F_p} \right)  \otimes | \cdot |^{-\frac{n-1}{2}}.
		\end{equation}
	\end{demn}
	
	\textbf{Étape 2 : Ramener au problème de comptage à l'infini.}
	
	Fixons une représentation $\pi_p \in \Pi_{\varphi}(\J_b(\Q_p))$ et  appliquons la proposition \ref{itm : globaliser auto} pour $I$ et la représentation $\pi_p$, on trouve un $L$-paramètre global discret générique $\Omega$ dont le $L$-paquet contient une représentation automorphe $\overline{\Pi} \in \mathcal{A}(I(\A))$ telle que 
	\begin{enumerate}
		\item[-] $\overline{\Pi}_p$ et $ \pi_p $ sont dans la même classe d'équivalence d'inertie,
		\item[-] $ \overline{\Pi}_{\infty} = \breve{\xi} $,
		\item[-] $ \overline{\Pi} \cong \Pi'$ dès lors que $ (\overline{\Pi})^p \cong (\Pi')^p $.
	\end{enumerate}	 
	
	De plus, grâce au lemme \ref{itm: torsion non ramifiée}, nous pouvons supposer $\overline{\Pi}_p = \pi_p, $.
	
	En prenant la partie $ \overline{\Pi}^{\infty, p}$-isotypique de la suite spectrale (\ref{itm : 1ss}) on a : 
	\begin{equation*} \phantomsection 
		E_2^{tq} = | \ker^1 (\Q, \overset{\bullet}{G}) |\left(\Ext^t_{\J_b(\Q_p)} \left( H^q_c(\mathcal{M}), \pi_p \right)_{cusp}  \right) \Longrightarrow \left( H^{t+q}_c(\Sh, \mathcal{L}_{\xi}) \right)_{p-cusp}[\overline{\Pi}^{\infty, p}]
	\end{equation*}
	
	Comme dans la première partie, le lemme \ref{itm: dégénère} simplifie la suite spectrale, on obtient donc l'égalité
	\begin{equation*} \label{itm:1ss galois}
		\sum_{\substack{\Pi \in \mathcal{A}(I) \\ \Pi^p = \overline{\Pi}^p}} \left(\hom_{\J_b(\Q_p)} \left( H^{n-1}_c(\mathcal{M}), \overline{\Pi}_p \right)_{cusp}  \right) = \sum_{\substack{\Pi \in \mathcal{A}(\overset{\bullet}{G}) \\ \Pi^{\infty, p} = \overline{\Pi}^{\infty, p} } } \Pi_p \otimes \rho(\Pi^{\infty})_p 	
	\end{equation*}
	
	Or $ \overline{\Pi} \cong \Pi'$ dès lors que $ (\overline{\Pi})^p \cong (\Pi')^p $, l'égalité au-dessus se récrit donc sous la forme
	\[
	\sum_{\widetilde{\pi}_p' \in \Pi_{\varphi}(G(\Q_p))} \widetilde{\pi}_p' \otimes \sigma_{\pi_p, \widetilde{\pi}_p'} = \sum_{\widetilde{\pi}_p' \in \Pi_{\varphi}(G(\Q_p))} \widetilde{\pi}_p' \otimes \rho(\widetilde{\pi}_p' \otimes \overline{\Pi}^{\infty, p})_p.
	\]
	
	En prenant la partie $[\pi_p']$-isotypique, on en déduit en particulier que $ \pi_p' \otimes \sigma_{\pi_p, \pi_p'} = \pi_p' \otimes \rho(\pi_p' \otimes \overline{\Pi}^{\infty, p})_p $ et donc que $\dim \rho (\pi_p' \otimes \overline{\Pi}^{\infty, p}) = \dim \sigma_{\pi_p, \pi_p'}$.
	
	D'après (\ref{itm : dim1}) et (\ref{itm : dim2}), on a $\dim \rho (\pi_p' \otimes \overline{\Pi}^{\infty, p}) = \sum_{i =1}^n m(\pi_i \otimes \pi_p' \otimes \overline{\Pi}^{\infty, p}) $. La formule de multiplicité (théorème \ref{itm: global} et \ref{itm: de U à GU}) implique alors que $\dim \sigma_{\pi_p, \pi_p'}$ est égale au nombre de représentations $\pi_i$ telles que $\pi_i \otimes \pi_p' \otimes \overline{\Pi}^{\infty, p}$ est une forme automorphe. D'après le théorème \ref{itm: global} cela \textit{équivaut à demander} l'égalité suivante :
	\[
	\langle \pi_i, s_{\infty} \rangle_{\varrho_{\infty}, z_{\infty}} \cdot \langle \pi_p', s_p \rangle_{\varrho_v, z_v} \prod_{v \neq \infty, p} \langle \overline{\Pi}_v, s_v \rangle_{\varrho_v, z_v} = 1, \quad \forall s \in S_{\Omega}^{\natural}
	\]
	
	Or $ \breve{\xi} \otimes \pi_p \otimes \overline{\Pi}^{\infty, p}$ étant une forme automorphe dans $\mathcal{A}(I(\A))$, on a l'égalité suivante :
	\[
	\langle \breve{\xi}, s_{\infty} \rangle_{\varrho_{\infty}, z_{\infty}} \cdot \langle \pi_p, s_p \rangle_{\varrho_v, z_v} \prod_{v \neq \infty, p} \langle \overline{\Pi}_v, s_v \rangle_{\varrho_v, z_v} = 1, \quad \forall s \in S_{\Omega}^{\natural}.
	\]
	
	On en déduit que $\pi_i \otimes \pi_p' \otimes \overline{\Pi}^{\infty, p}$ est une forme automorphe si et seulement si
	\begin{equation} \phantomsection \label{itm : equa}
		\langle \pi_i, s_{\infty} \rangle_{\varrho_{\infty}, z_{\infty}} \cdot \langle \pi_p', s_p \rangle_{\varrho_v, z_v} \cdot \langle \breve{\xi}, s_{\infty} \rangle_{\varrho_{\infty}, z_{\infty}} \cdot \langle \pi_p, s_p \rangle_{\varrho_v, z_v} = 1, \quad \forall s \in S_{\Omega}^{\natural}.
	\end{equation}
	
	Donc, afin de calculer la dimension des représentations galoisiennes, on doit calculer les caractères du groupe centralisateur à l'infini.
	
	\textbf{Étape 3 : Calcul de caractères à l'infini}
	
	Considérons les  $L$-paquets $\Pi_{(\xi)}(I(\R))$ et $\Pi_{(\xi)}(\overset{\bullet}{G}(\R))$. Comme $\xi$ est un $L$-paramètre discret, on peut écrire $\xi^n := \eta_{\chi_{\kappa} *} \xi$ sous la forme $ \xi^n = \xi_1 \oplus \cdots \oplus \xi_n $ où les $\xi_i$ sont deux à deux distincts. Le groupe centralisateur est donné par 
	\[
	S_{\xi}^{\natural} \cong \prod_{i = 1}^n O(1, \C)  \cong \prod_{i = 1}^n (\Z / 2\Z).
	\]
	
	Le groupe $Z(\widehat{G})^{\Gamma} = \{ \pm 1 \}$ est envoyé diagonalement dans $S_{\xi}^{\natural}$.
	
	Le $L$-paquet $\Pi_{(\xi)}(\overset{\bullet}{G}(\R))$ est constitué par $n$ représentations $\pi_i$ qui correspondent à $n$ caractères $\tau_{\pi_i}$ de $S_{\xi}^{\natural}$ dont la restriction sur $Z(\widehat{G})^{\Gamma}$ est un caractère de $Z(\widehat{G})^{\Gamma}$ calculé en fonction de $a^G_{\infty}$. De même le $L$-paquet $\Pi_{(\xi)}(I(\R))$ est constitué par la représentation $\breve{\xi}$ qui correspondent à un caractère $\tau_{\breve{\xi}}$ de $S_{\xi}^{\natural}$ dont la restriction sur $Z(\widehat{G})^{\Gamma}$ est calculé par $a^I_{\infty}$. On va calculer les $n$ caractères $\tau_{\breve{\xi}} \cdot \tau_{\pi_i}$ où $i \in \{ 1, \cdots n \}$. 
	
	\begin{remarque} \phantomsection \label{itm: produit}
		Puisque $a^G_{\infty} + a^I_{\infty} \equiv 1 \text{ mod } 2$, la restriction de $\tau_{\breve{\xi}} \cdot \tau_{\pi_i}$ sur $Z(\widehat{G})^{\Gamma}$ est le caractère non trivial. On en déduit que pour tout $i$, on a
		\[
		\prod_{j = 1}^{n} \tau_{\breve{\xi}} \cdot \tau_{\pi_i} \big( 1, \cdots,1, \underbrace{-1}_{j}, 1, \cdots, 1 \big) = \tau_{\breve{\xi}} \cdot \tau_{\pi_i} \big( -1, \cdots, -1 \big) = -1.
		\]
	\end{remarque}
	
	Choisissons un $L$-paramètre supercuspidal $\psi$ de $G(\Q_p)$ de sorte que $ \eta_{\chi_{\kappa} *} \psi = \psi_{F_p} = \psi^{n_1}_1 \oplus \psi^{n_2}_2 $ avec $\dim \psi^{n_1}_1 = 1$ et $\dim \psi^{n_2}_2 = n-1$. On a 
	\[
	S_{\psi} \cong \prod_{i = 1}^2 O(1, \C)  \cong \prod_{i = 1}^2 (\Z / 2\Z) \cong S_{\psi}^{\natural}.
	\]	
	
	Fixons une représentation $\pi_p \in \Pi_{\psi}(\J_b(\Q_p))$ et choisissons une représentation $\pi_p' \in \Pi_{\psi}(G(\Q_p))$ de sorte que le caractère $\tau_{\pi_p} \cdot \tau_{\pi_p'}$ de $S^{\natural}_{\psi}$ est donné par
	\[
	\tau_{\pi_p} \cdot \tau_{\pi_p'} (-1,1) = -1 \quad \quad \quad \tau_{\pi_p} \cdot \tau_{\pi_p'} (1,-1) = 1.
	\]
	
	En utilisant la même construction que celle dans la proposition \ref{globaliser les paramètres locaux} (et aussi le lemme \ref{itm: torsion non ramifiée}), il existe, pour tout $i_0$ fixé, un paramètre global discret générique $ \Phi(i_0)$ de la forme $ \Phi(i_0)^n = \Phi^{n_1}_1(i_0) \boxplus \Phi^{n_2}_2(i_0) $ de sorte que
	\begin{enumerate}
		\item[$\bullet$] $\Phi^{n_j}_j(i_0)$ est un $L$-paramètre global discret générique pour $j = 1, 2$,
		\item[$\bullet$] $(\Phi(i_0))_{\infty} = \xi$ et $(\Phi^{n_1}_1(i_0))_{\infty} = \xi_{i_0}$,
		\item[$\bullet$] $(\Phi^{n_1}_1(i_0))_p = \psi^{n_1}_1 $ et $(\Phi^{n_2}_2(i_0))_p = \psi^{n_2}_2$,
		\item [$\bullet$] Le morphisme canonique 
		\[
		S^{\natural}_{\Phi(i_0)} \longrightarrow S^{\natural}_{\psi}
		\] 
		est un isomorphisme.
	\end{enumerate}
	
	Détaillons maintenant les groupe centralisateurs ainsi que les morphismes de localisation 
	\[
	S_{\Phi(i_0)} \cong \prod_{i = 1}^2 O(1, \C)  \cong \prod_{i = 1}^2 (\Z / 2\Z) \cong S_{\Phi(i_0)}^{\natural}.
	\]
	
	Les morphismes de localisation $ S_{\Phi(i_0)} \longrightarrow S_{\psi} $ et $ S_{\Phi(i_0)} \longrightarrow S_{\xi} $ sont donnés par les formules ci-dessous
	
	\begin{align*}
		S_{\Phi(i_0)}^{\natural} \quad \quad & \longrightarrow \quad \quad S_{\psi}^{\natural} & S_{\Phi(i_0)}^{\natural} \quad \quad & \longrightarrow \quad \quad \quad S^{\natural}_{\xi} & \\
		(x_1, x_2) \quad & \longmapsto  (x_1, x_2) & (x_1, x_2) \quad & \longmapsto  (\underbrace{x_2, \cdots x_2}_{i_0-1}, x_1, \underbrace{x_2 \cdots x_2}_{n - i_0}) &
	\end{align*}
	
	\begin{notation}
		Si on a une représentation irréductible $\tau$ de $S_{\psi}^{\natural}$ ou de $S_{\xi}^{\natural}$, on note $\overline{\tau}^{i_0}$ la représentation induite sur $S^{\natural}_{\Phi(i_0)}$ par les morphismes de localisation.	
	\end{notation}
	
	Utilisons la proposition \ref{itm : globaliser auto} et le lemme \ref{itm: torsion non ramifiée} pour $I$ et la représentation $\pi_p$ et le $L$-paramètre $\Phi(i_0)$, on trouve une représentation automorphe $\overline{\Pi}(i_0) \in \Pi_{\Phi(i_0)}(I(\A))$ telle que 
	\begin{enumerate}
		\item [-] $\overline{\Pi}(i_0)_p = \pi_p, \ \overline{\Pi}(i_0)_{\infty} = \breve{\xi} $, 
		\item[-] $ \overline{\Pi}(i_0) \cong \Pi'$ dès lors que $ (\overline{\Pi}(i_0))^p \cong (\Pi')^p $.
	\end{enumerate}	
	
	En prenant la partie $ \overline{\Pi}(i_0)^{\infty, p}$-isotypique de la suite spectrale (\ref{itm : 1ss}) on a : 
	\begin{equation} \phantomsection \label{itm : Adam}
		E_2^{tq} = | \ker^1 (\Q, \overset{\bullet}{G}) |\left( \Ext^t_{\J_b(\Q_p)} \left( H^q_c(\mathcal{M}), \pi_p \right)_{cusp}  \right) \Longrightarrow \left( H^{t+q}_c(\Sh, \mathcal{L}_{\xi}) \right)_{p-cusp}[\overline{\Pi}(i_0)^{\infty, p}]
	\end{equation}
	
	Comme dans l'étape $2$, on obtient l'égalité $ \pi_p' \otimes \sigma_{\pi_p, \pi_p'} = \pi_p' \otimes \rho(\pi_p' \otimes \overline{\Pi}(i_0)^{\infty, p})_p $ et donc $\dim \rho (\pi_p' \otimes \overline{\Pi}(i_0)^{\infty, p}) = \dim \sigma_{\pi_p, \pi_p'}$.
	
	D'une part, l'égalité (\ref{itm:dim galois}) couplée avec le fait que $\psi_{F_p} = \psi^{n_1}_1 \oplus \psi^{n_2}_2$ avec $\dim \psi^{n_1}_1 = 1$ et $\dim \psi^{n_2}_2 = n-1$ implique que $\dim \sigma_{\pi_p, \pi_p'} \in \{ 0, 1, n-1, n \}$.
	
	De plus $ \dim \rho (\pi_p' \otimes \overline{\Pi}(i_0)^{\infty, p}) = \sum_{i =1}^n m(\pi_i \otimes \pi_p' \otimes \overline{\Pi}(i_0)^{\infty, p}) $. Comme dans la fin de l'étape 2, la formule de multiplicité implique que $\dim \sigma_{\pi_p, \pi_p'}$ est égale au nombre de représentations $\pi_i$ telles que  $\overline{\tau}^{i_0}_{\breve{\xi}} \cdot \overline{\tau}^{i_0}_{\pi_i} = \overline{\tau}^{i_0}_{\pi_p} \cdot \overline{\tau}^{i_0}_{\pi_p'} $.
	
	Montrons que $\dim \sigma_{\pi_p, \pi_p'} = 1$. Supposons le contraire, il y a $3$ possibilités.
	
	\textbf{Cas 1 :} $\dim \sigma_{\pi_p, \pi_p'} = 0$.
	
	Dans ce cas on a $\overline{\tau}^{i_0}_{\breve{\xi}} \cdot \overline{\tau}^{i_0}_{\pi_i} \neq \overline{\tau}^{i_0}_{\pi_p} \cdot \overline{\tau}^{i_0}_{\pi_p'} $ pour tout $i$. La description des morphismes de localisation implique que pour tout $ i \in \{ 1, \cdots, n \} $ on a 
	\[
	\tau_{\breve{\xi}} \cdot \tau_{\pi_i} \big(1,\cdots,1, \underbrace{-1}_{i_0}, 1, \cdots, 1 \big) = 1.
	\] 
	
	En faisant varier $i_0$, on en déduit que tous les $ \tau_{\breve{\xi}} \cdot \tau_{\pi_i} $ coïncident, contradiction.
	
	\textbf{Cas 2 :} $\dim \sigma_{\pi_p, \pi_p'} = n$.
	
	La description des morphismes de localisation implique que pour tout $ i \in \{ 1, \cdots, n \} $ on a 
	\[
	\tau_{\breve{\xi}} \cdot \tau_{\pi_i} \big(1,\cdots,1, \underbrace{-1}_{i_0}, 1, \cdots, 1 \big) = -1.
	\] 
	
	Comme précédemment en faisant varier $i_0$ , on en déduit que les $ \tau_{\breve{\xi}} \cdot \tau_{\pi_i} $ sont égaux pour tout $i$, contradiction.
	
	\textbf{Cas 3 :} $\dim \sigma_{\pi_p, \pi_p'} = n-1$.
	
	La description des morphismes de localisation implique qu'il y a exactement $n-1$ indices $ i \in \{ 1, \cdots, n \} $ tels que 
	\[
	\tau_{\breve{\xi}} \cdot \tau_{\pi_i} \big(1,\cdots,1, \underbrace{-1}_{i_0}, 1, \cdots, 1 \big) = -1.
	\] 
	
	Comme $i_0$ est arbitraire, on en déduit que 
	\[
	\prod_{i_0=1}^{n} \prod_{i = 1}^{n} \tau_{\breve{\xi}} \cdot \tau_{\pi_i} \big( 1, \cdots, 1, \underbrace{-1}_{i_0}, 1, \cdots,1 \big) = (-1)^{(n-1) \cdot n} = 1.
	\]
	
	Or d'après la remarque \ref{itm: produit} et le fait que $n$ est impair, on en déduit que
	\[
	\prod_{i_0=1}^{n} \prod_{i = 1}^{n} \tau_{\breve{\xi}} \cdot \tau_{\pi_i} \big( 1, \cdots, 1, \underbrace{-1}_{i_0}, 1, \cdots,1 \big) = \prod_{i = 1}^{n} \tau_{\breve{\xi}} \cdot \tau_{\pi_i} \big( -1, \cdots, -1 \big) = (-1)^n = (-1).
	\]
	
	Ceci est une contradiction. 
	
	On en déduit que $ \dim \sigma_{\pi_p, \pi_p'} = 1 $, autrement dit pour tout $i_0$, il y a exactement un indice $i \in \{ 1, \cdots, n \}$ tel que 
	\[
	\tau_{\breve{\xi}} \cdot \tau_{\pi_i} \big(1,\cdots,1, \underbrace{-1}_{i_0}, 1, \cdots, 1 \big) = -1.
	\]
	
	En utilisant la remarque \ref{itm: produit}, on vérifie aisément que les $n$ caractères $\tau_{\breve{\xi}} \cdot \tau_{\pi_i}$ de $(\Z / 2\Z)^n$ avec $i \in \{ 1, \cdots, n \}$ sont les caractères $\lambda_i$ où
	\begin{equation} \phantomsection \label{itm : desc}
		\lambda_i \big(1, \cdots, \underbrace{-1}_{\text{j}}, \cdots, 1\big) = \left\lbrace \begin{array}{ccc}
			-1 & si & i=j \\
			1 & si & 1 \leq i \neq j \leq n.
		\end{array} \right.
	\end{equation}
	
	\textbf{Étape 4 : Fin de la démonstration}

	\begin{lemme}
		Il existe un paramètre global générique  $ \varPsi = (\varPsi^n, \widetilde{\varPsi})$ où $ \varPsi^n = \varPsi^{n_1}_1 \boxplus \varPsi^{n_2}_2 $ avec $\varPsi_i^{n_i}$ des paramètres simples génériques de sorte que
		\begin{enumerate}
			\item[i)] $(\varPsi^{n_1}_1)_p = \varphi^{n_1}_1$
			\item[ii)] $(\varPsi^{n_2}_2)_p = \varphi^{n_2}_2 \oplus \cdots \oplus \varphi^{n_r}_r$.
			\item[iii)] Il y a une place $q$ de $\Q$ décomposée dans le corps quadratique imaginaire $ \overset{\bullet}{F} $ telle que 
			\[
			(\varPsi^{n_1}_1)_q = \chi_1 \oplus \cdots \oplus \chi_{n_1} \quad \quad \quad (\varPsi^{n_2}_2)_q = \chi_{n_1 + 1} \oplus \cdots \oplus \chi_n
			\]
			où les $\chi_i$ sont non ramifiées et deux à deux distincts.
			\item[iv)] $ \varPsi_{\infty} = \xi$. 
		\end{enumerate}
	\end{lemme}
	\begin{proof}
		On utilise la même construction que celle dans la proposition \ref{globaliser les paramètres locaux}.
	\end{proof}  
	
	
	On fixe une représentation $\pi_p \in \Pi_{\varphi}(\J_b(\Q_p))$. Comme dans la démonstration de la proposition \ref{itm : globaliser auto} (et aussi le lemme \ref{itm: torsion non ramifiée}), il existe une représentation automorphe $\widetilde{\Pi} $ de $I(\A)$ dans le paquet $\Pi_{\varPsi}(I)$ de sorte que $\widetilde{\Pi}_p = \pi_p$ et $\widetilde{\Pi}_{\infty} = \breve{\xi}$ (mais $\Pi^p = \widetilde{\Pi}^p$ \textit{n'implique pas forcément} $\Pi = \widetilde{\Pi}$). 
	
	Prenons la partie $[\widetilde{\Pi}^{\infty,p}]$ de la suite spectrale (\ref{itm : 1ss}). Grâce au lemme \ref{itm: dégénère}, la suite spectrale dégénère, on obtient donc des isomorphismes
	\[
	| \ker^1 (\Q, \overset{\bullet}{G}) | \sum_{\substack{\Pi \in \mathcal{A}(I) \\ \Pi^p = \widetilde{\Pi}^p}} \left( \hom_{\J_b(\Q_p)} \left( H^q_c(\mathcal{M}), \Pi_p \right)_{cusp}  \right) \otimes [\widetilde{\Pi}^{\infty, p}] = \left( H^q_c(\Sh, \mathcal{L}_{\xi}) \right)_{p-cusp}[\widetilde{\Pi}^{\infty, p}].
	\] 
	
	En utilisant la formule de Matsushima couplée avec la proposition \ref{itm : comparer}, on en déduit que 
	
	
	\begin{equation} \label{itm : 1ss galois}
		\sum_{\substack{\overline{\Pi} \in \mathcal{A}(I) \\ \overline{\Pi}^p = \widetilde{\Pi}^p}} \left(\hom_{\J_b(\Q_p)} \left( H^q_c(\mathcal{M}), \overline{\Pi}_p \right)_{cusp}  \right) = \sum_{\substack{\Pi \in \mathcal{A}(\overset{\bullet}{G}) \\ \Pi^{\infty, p} = \widetilde{\Pi}^{\infty, p} } } \Pi_p \otimes \rho(\Pi)_p.
	\end{equation}
	
	D'après le calcul dans la section \ref{itm: centra}, on a
	\[
	S_{\varPsi} \cong \prod_{i = 1}^2 O(1, \C)  \cong \prod_{i = 1}^2 (\Z / 2\Z) \cong S_{\varPsi}^{\natural}.
	\]
	\[
	S_{\varphi} \cong \prod_{i = 1}^r O(1, \C)  \cong \prod_{i = 1}^r (\Z / 2\Z) \cong S_{\varphi}^{\natural} \quad \quad S_{\varPsi_{\infty}} \cong \prod_{i = 1}^n O(1, \C)  \cong \prod_{i = 1}^n (\Z / 2\Z) \cong S_{\varPsi_{\infty}}^{\natural}.
	\]
	
	Les morphismes de localisation $ S_{\varPsi} \longrightarrow S_{\varphi} $ et $ S_{\varPsi} \longrightarrow S_{\varPsi_{\infty}} $ sont donnés comme ci-dessous
	
	\begin{align*}
		S_{\varPsi}^{\natural} \quad \quad & \longrightarrow \quad \quad S_{\varphi}^{\natural} & S_{\varPsi}^{\natural} \quad \quad & \longrightarrow \quad \quad \quad S_{\varPsi_{\infty}} & \\
		(x_1, x_2) \quad & \longmapsto  (x_1, \underbrace{x_2, \cdots, x_2}_{r-1}) & (x_1, x_2) \quad & \longmapsto  (\underbrace{x_1, \cdots x_1}_{n_1}, \underbrace{x_2, \cdots, x_2}_{n_2}) &
	\end{align*}
	
	Si on a une représentation irréductible $\tau$ de $S_{\varphi}^{\natural}$ ou de $S_{\varPsi_{\infty}}^{\natural}$, on note $\overline{\tau}$ la représentation induite sur $S^{\natural}_{\varPsi}$ par les morphismes de localisation. 
	
	Pour $ 1 \leq i \leq r$, on note $\tau_i$ le caractère de $S^{\natural}_{\varphi}$ défini par \ref{itm : rep}.

	
	Le groupe $Z(\widehat{G})^{\Gamma} = \{ \pm 1 \}$ est envoyé diagonalement dans $S_{\varphi}^{\natural}$. Les représentations dans le paquet $\Pi_{\varphi}(\J_b(\Q_p))$ sont paramétrées par les éléments de $\text{Irr}(S_{\psi}^{\natural}, \chi_{I})$ où $\chi_{I}$ est calculé en fonction de $a_p^I$ (voir \ref{itm : formule de multiplicité} et exemple \ref{itm : étendre}) et celles dans le paquet $\Pi_{\varphi}(G(\Q_p))$ sont paramétrées par les éléments de $\text{Irr}(S_{\psi}^{\natural}, \chi_{G})$ où $\chi_G$ est calculé en fonction de $a_p^G$. 
	
	Puisque les rôles des $\tau_i $ pour $ 1 \leq i \leq r$ sont identiques, il suffit de montrer que $\sigma_{\pi_p, \pi_p'} = r_{\mu_1} \circ (\varphi^{n_1}_{1}) \otimes | \cdot |^{-\frac{n-1}{2}}$ pour $\tau_{\pi_p} \cdot \tau_{\pi_p'} = \tau_1$.
	
	Choisissons $\pi_p' \in \Pi_{\varphi}(G(\Q_p))$ de sorte que $\tau_{\pi_p} \cdot \tau_{\pi_p'} = \tau_1$. D'après (\ref{itm : equa}) et la description explicite des caractères $\tau_{\breve{\xi}} \cdot \tau_{\pi_i}$ dans (\ref{itm : desc}), on en déduit que $\dim \sigma_{\pi_p, \pi_p'} = n_1$.	
	
	Comme $\widetilde{\Pi}$ est une représentation automorphe de $I(\A)$, on en déduit que $ \pi_i \otimes \pi_p' \otimes \widetilde{\Pi}^{\infty, p} $ est une représentation automorphe de $\overset{\bullet}{G}(\A)$ si et seulement si 
	
	\[
	\overline{\tau}_{\pi_i} \cdot \overline{\tau}_{\pi_p'} = \overline{\tau}_{\breve{\xi}} \cdot \overline{\tau}_{\pi_p}.
	\] 
	
	Puisque $\tau_{\pi_p} \cdot \tau_{\pi_p'} = \tau_1$, l'égalité ci-dessus équivaut à  
	\[
	\overline{\tau}_{\breve{\xi}} \cdot \overline{\tau}_{\pi_i} = \overline{\tau}_1.
	\]
	
	D'après l'étape $3$, on a $\overline{\tau}_{\breve{\xi}} \cdot \overline{\tau}_{\pi_i} = \overline{\lambda}_i$, donc $ \pi_i \otimes \pi_p' \otimes \widetilde{\Pi}^{\infty, p} $ est une représentation automorphe de $\overset{\bullet}{G}(\A)$ si et seulement si $\overline{\lambda}_i = \overline{\tau}_1$. En utilisant la description des morphismes de localisation ainsi que la définition de $\tau_1$ (c.f. \ref{itm : rep}) et de $\lambda_i$ (c.f. \ref{itm : desc}), on voit que $\overline{\lambda}_i = \overline{\tau}_1$ si et seulement si $ i \in \{ 1, \cdots n_1 \}$. On en déduit que $ \dim \rho(\pi_p' \otimes \widetilde{\Pi}^{\infty, p}) = n_1 $. Le corollaire \ref{itm: galois} implique que  
	\[ \rho(\pi_p' \otimes \widetilde{\Pi}^{\infty, p})_p =r_{\mu_1} \circ (\varphi^{n_1}_{1}) \otimes | \cdot |^{-\frac{n-1}{2}}. 
	\]
	
	Maintenant en prenant la partie $[\pi_{p}']$-isotypique dans (\ref{itm : 1ss galois}), on voit que
	\[
	\sum_{\substack{\overline{\pi}_p \in \Pi_{\varphi}(\J_b(\Q_p)) \\  \overline{\pi}_p \otimes \widetilde{\Pi}^p \in \mathcal{A}(I) }} \sigma_{\overline{\pi}_p, \pi_p'} \otimes \pi_p' = \pi_p' \otimes r_{\mu_1} \circ (\varphi^{n_1}_{1}) \otimes | \cdot |^{-\frac{n-1}{2}} 
	. 
	\]
	
	Comme $ \dim \sigma_{\pi_p, \pi_p'} = n_1 = \dim r_{\mu_1} \circ (\varphi^{n_1}_{1}) $, on en déduit que
	\[
	\sigma_{\pi_p, \pi_p'} = r_{\mu_1} \circ (\varphi^{n_1}_{1}) \otimes | \cdot |^{-\frac{n-1}{2}}.
	\]
	
	\section{Appendice} \phantomsection \label{itm : appendice}
	Dans cet appendice on démontre un analogue plus faible du \ref{text : thm} pour $F^{+}$ un corps totalement réel de degré impair strictement plus grand que $1$. Le principe de la démonstration repose sur l'étude des variétés de Shimura de type Kottwitz-Harris-Taylor $\Sh_{ / \overset{\bullet}{E}}$ définies sur leur corps reflex $\overset{\bullet}{E} = \overset{\bullet}{F}$ où $\overset{\bullet}{F} = \mathcal{K} F^{+}$ avec $\mathcal{K}$ un corps quadratique imaginaire et plus particulièrement sur la géométrie de la fibre spéciale en une place $p$ inerte dans $\overset{\bullet}{F}$ d'un modèle $\Sh_{/ \mathcal{O}_p}$ où $\mathcal{O}_p$ est l'anneau des entiers de $\overset{\bullet}{F}_p$.
	\begin{Propositionn} \phantomsection \label{itm : résultat partiel}
		Soit $ \varphi : W_{\Q_p} \longrightarrow \big( \prod_{\tau \in \Phi} GL_n(\C) \times \C^{\times} \big) \rtimes W_{\Q_p} $ un $L$-paramètre discret de $G(\Q_p)$. Notons $\Pi_{\varphi}(G(\Q_p))$ et $\Pi_{\varphi}(J_b(\Q_p))$ le $L$-paquet (supercuspidal) respectivement de $G(\Q_p)$ et de $J_b(\Q_p)$ correspondant à $\varphi$. Alors pour $\pi_p'$ une représentation cuspidale dans $\Pi_{\varphi}(G(\Q_p))$ on a
		\begin{enumerate}
			\item[i)] $\sigma^i_{\pi_p, \pi_p'} = 0$ si $ i \neq n-1 $.
			\item[ii)] $\sigma_{\pi_p, \pi_p'} = 0$ si $\pi_p \notin \Pi_{\varphi}(\J_b(\Q_p))$,
			\item[iii)] 
			\[
			\dim \sigma_{\pi_p, \pi_p'} = \dim \hom_{S_{\varphi}}( \tau_{\pi_p'} \otimes \tau_{\pi_p} ,r_{\mu} \circ \varphi_{F_p} ) \otimes | \cdot |^{-\frac{n-1}{2}}.
			\]
		\end{enumerate}
		
	\end{Propositionn}
	\begin{proof} 
		Soit $\mathcal{D} = ( \overset{\bullet}{F}, B, *, V, \langle \cdot | \cdot \rangle, h, \overset{\bullet}{G})$ une donnée globale de type PEL globalisant la donnée locale de sorte que $\End_B(V)$ est une algèbre simple qui est déployée en toutes places finies comme dans la proposition \ref{itm: globale -> locale}. Soit $\Sh$ la variété de Shimura associée, le groupe $\overset{\bullet}{G}(\Q_p) = G(\Q_p)$ est le groupe des similitudes unitaires quasi-déployé en $n$ variables. En particulier $\Sh$ est de signature $(1, n-1), (0, n), \cdots, (0, n)$ à l'infini et $\Sh$ est donc \textit{compacte}.
		
		On note $\mathcal{M} (\mathcal{D}_{\Q_p}, b)$ l'espace de Rapoport-Zink associé à la donnée locale.
		
		Soit $\phi$ une classe d'isogénie intervenant dans la strate basique et $ I := (I^{\phi})$ le groupe réductif associé. On sait que $I(\R)$ est la forme compacte modulo le centre de $\overset{\bullet}{G}(\R)$, que $I(\Q_p) = \J_b(\Q_p) $ et que $ I(\A_f^p) = \overset{\bullet}{G}(\A_f^p) $. D'après la proposition \ref{itm : suite spectrale} il y a une suite spectrale $\overset{\bullet}{G}(\A_f) \times W_{F_p}$-équivariante:
		\begin{equation} \phantomsection \label{itm : 1sss}
			E_2^{tq} = | \ker^1 (\Q, \overset{\bullet}{G}) | \sum_{\substack{\Pi \in \mathcal{A}(I) \\ \Pi_{\infty} = \breve{\xi}}} \left( \Ext^t_{\J_{b}(\Q_p)} \left( H^q_c(\mathcal{M}), \Pi_p \right)_{cusp}  \right) \otimes (\Pi^{\infty, p}) \Longrightarrow \left( H^{t+q}(\Sh, \mathcal{L}_{\xi}) \right)_{p-cusp}. 
		\end{equation}	
		
		Fixons une représentation $\pi_p \in \Pi_{\varphi}(\J_b(\Q_p))$ et  appliquons la proposition \ref{itm : globaliser auto} (et aussi le lemme \ref{itm: torsion non ramifiée}) pour $I$ et la représentation $\pi_p$; on trouve une représentation automorphe $\overline{\Pi} \in \mathcal{A}(I(\A))$ telle que 
		\begin{enumerate}
			\item [-] $\overline{\Pi}_p = \pi_p, \ \overline{\Pi}_{\infty} = \breve{\xi} $, 
			\item[-] $ \overline{\Pi} \cong \Pi'$ dès lors que $ (\overline{\Pi})^p \cong (\Pi')^p $.
		\end{enumerate}	
		
		En utilisant le même argument que dans l'étape $2$ de la démonstration du théorème principal on en déduit que 
		\[
		\dim \rho (\pi_p' \otimes \overline{\Pi}^{\infty, p}) = \dim \sigma_{\pi_p, \pi_p'},
		\]
		où $\rho (\pi_p' \otimes \overline{\Pi}^{\infty, p})$ est la partie $[\pi_p' \otimes \overline{\Pi}^{\infty, p}]$-isotypique dans la cohomologie de $\Sh$.
		
		D'autre part on a
		\[
		\dim \rho (\pi_p' \otimes \overline{\Pi}^{\infty, p}) = \sum_{i =1}^n m(\pi_i \otimes \pi_p' \otimes \widetilde{\Pi}^{\infty, p}).
		\]
		
		La formule de multiplicité \ref{itm: global} couplée avec \ref{itm: de U à GU} implique alors que $m(\pi_i \otimes \pi_p' \otimes \widetilde{\Pi}^{\infty, p})$ est soit nulle soit égale à $1$. On en déduit donc que $\dim \sigma_{\pi_p, \pi_p'}$ est égale au nombre de représentations $\pi_i$ telle que $\overline{\tau}_{\breve{\xi}} \cdot \overline{\tau}_{\pi_i} = \overline{\tau}_{\pi_p} \cdot \overline{\tau}_{\pi_p'} $. Nous utilisons le calcul fait dans l'étape $3$ de la démonstration du théorème principal pour conclure.
	\end{proof}
	\begin{thmm} \phantomsection \label{itm : thm partiel}
		Soit $ \varphi : W_{\Q_p} \longrightarrow \big( \prod_{\tau \in \Phi} GL_n(\C) \times \C^{\times} \big) \rtimes W_{\Q_p} $ un $L$-paramètre discret de $G(\Q_p)$. Notons $\Pi_{\varphi}(G(\Q_p))$ et $\Pi_{\varphi}(J_b(\Q_p))$ le $L$-paquet (supercuspidal) respectivement de $G(\Q_p)$ et de $J_b(\Q_p)$ correspondant à $\varphi$. Alors pour $\pi_p'$ une représentation cuspidale dans $\Pi_{\varphi}(G(\Q_p))$ on a 
		\[
		\sum_{\pi_p \in \Pi_{\varphi}(J_b(\Q_p))} \sigma_{\pi_p, \pi_p'} = \left( r_{\mu} \circ \varphi_{F_p} \right)  \otimes | \cdot |^{-\frac{n-1}{2}}.
		\] 
	\end{thmm}
	\begin{proof}  
		
		Soit $\mathcal{D}$ une donnée globale de type PEL globalisant la donnée locale comme dans la proposition \ref{itm: globale -> locale} de sorte que $\End_B(V)$ est une \textit{algèbre à division} qui est en toute place finie soit déployée, soit une algèbre à division. Soit $\Sh$ la variété associée, le groupe $\overset{\bullet}{G}(\Q_p) = GU^*(n)$ est le groupe de similitudes unitaires quasi-déployé en $n$ variables. En particulier $\Sh$ est de signature $(1, n-1), (0, n) \cdots (0, n)$ à l'infini et $\Sh$ est \textit{compacte} car $\End_B(V)$ est une algèbre à division. 
		
		On note $\mathcal{M} (\mathcal{D}_{\Q_p}, b)$ l'espace de Rapoport-Zink associé à la donnée locale.
		
		Soit $\phi$ une classe d'isogénie intervenant dans la strate basique et $I^{\phi}$ le groupe réductif associé. On sait que $I^{\phi}(\R)$ est la forme compacte modulo le centre de $\overset{\bullet}{G}(\R)$, que $I^{\phi}(\Q_p) = \J_b(\Q_p) $ et que $ I^{\phi}(\A_f^p) = \overset{\bullet}{G}(\A_f^p) $. D'après la proposition \ref{itm : suite spectrale}, il y a une suite spectrale $\overset{\bullet}{G}(\A_f) \times W_{F_p}$-équivariante
		\begin{equation} \phantomsection \label{itm : ss}
			E_2^{tq} = | \ker^1 (\Q, \overset{\bullet}{G}) | \sum_{\substack{\Pi \in \mathcal{A}(I^{\phi}) \\ \Pi_{\infty} = \breve{\xi}}} \left( \Ext^t_{\J_{b}(\Q_p)} \left( H^q_c(\mathcal{M}), \Pi_p \right)_{cusp}  \right) \otimes (\Pi^{\infty, p}) \Longrightarrow \left( H^{t+q}(\Sh, \mathcal{L}_{\xi}) \right)_{p-cusp} 
		\end{equation}		
		
		On choisit une représentation $\xi$ de dimension finie de $I^{\phi}(\C)$. Soit $\Pi(\xi)$ le $L$-paquet des représentations de séries discrètes de $\overset{\bullet}{G}(\R)$ cohomologiques pour $\xi$.
		
		Considérons une représentation supercuspidale $\widetilde{\pi}_p $ dans $\Pi_{\varphi}(G(\Q_p))$. D'après \cite{Clo86}, il existe une représentation automorphe $ \widetilde{\Pi} $ de $\overset{\bullet}{G}$ telle que $ \widetilde{\Pi}_p = \widetilde{\pi}_p $ et $ \widetilde{\Pi}_{\infty} \in \Pi(\xi) $ et que $ \widetilde{\Pi}_{w_0} $ est supercuspidale pour une place $w_0$ inerte. On peut supposer de plus que pour toute place finie décomposée $v$ telle que $\End_B(V)$ est une algèbre à division, la composante $\widetilde{\Pi}_v$ et $JL(\widetilde{\Pi}_v)$ soient supercuspidales, où $JL$ est l'application de Jacquet-Langlands. 
		
		Maintenant, en appliquant le cas $(A)$ du théorème $3.1.6$  de \cite{HL04} pour les groupes de similitudes unitaires globaux $\overset{\bullet}{G}$ et $ I^{\phi} $, on obtient une représentation automorphe $\widetilde{\Pi}^*$ de $I^{\phi}$ telle que $ (\widetilde{\Pi}^*)_w = (\widetilde{\Pi})_w $ pour toute place finie $w \neq p$ et $(\widetilde{\Pi}^*)_{\infty} = \xi $.
		
		En prenant la partie $ \widetilde{\Pi}^{\infty, p}$-isotypique de la suite spectrale (\ref{itm : ss}), on a 
		\begin{equation*}
			E_2^{tq} = | \ker^1 (\Q, \overset{\bullet}{G}) | \sum_{\substack{\Pi \in \mathcal{A}(I^{\phi}) \\ \Pi^{p} = (\widetilde{\Pi}^*)^{p} }} \left(\Ext^t_{\J_{b}(\Q_p)} \left( H^q_c(\mathcal{M}), \Pi_p \right)_{cusp}  \right) \Longrightarrow \left( H^{t+q}(\Sh, \mathcal{L}_{\xi}) \right)_{p-cusp}[\widetilde{\Pi}^{\infty, p}]
		\end{equation*}
		\begin{lemmee} \phantomsection \label{itm : même paquet}
			Soit $\Pi \in \mathcal{A}(I^{\phi})$ telle que $ \Pi^p = (\widetilde{\Pi}^*)^p $ alors $\Pi_p$ est dans le paquet $\Pi_{\varphi}(J_b(\Q_p))$.
		\end{lemmee}
		\begin{proof}
			On pose $S_0$ l'ensemble des places finies décomposées telles que $End_B(V)$ est une algèbre à division et $S = S_0 \cup \{ w_0 \}$ où $w_0$ est la place inerte qu'on fixée auparavant. On peut supposer que $|S_0| \geq 2 $.
			
			D'après la proposition 10.1.1 de \cite{Far04}, il existe un groupe de similitudes unitaires global $GU^1$ tel que $(GU^1)_v$ est quasi-déployé si $v$ est décomposée et $(GU^1)_w = (I^{\phi})_w$ pour toute $w \notin S$. 
			
			D'après le cas $(B)$ du théorème $3.1.6$ de \cite{HL04}, il existe une représentation automorphe $\Pi_1$ de $GU^1(\A)$ telle que $ (\Pi_1)_w = \Pi_w $ pour toute $ w \notin S $. On obtient en particulier $ (\Pi_1)_p = \Pi_p $.
			
			Par le même argument, on obtient un groupe de similitudes unitaires $GU^2$ tel que $(GU^2)_v$ est quasi-déployé si $v$ est décomposée et $(GU^2)_w = (\overset{\bullet}{G})_w$ pour toute $w \notin S$ ainsi qu'une représentation automorphe $\Pi_2$ de $GU^2(\A)$ satisfaisant $ (\Pi_2)_w = \widetilde{\Pi}_w $ pour toute $w \notin S$. En particulier on a $(\Pi_2)_p = \widetilde{\pi}_p$ et $(\Pi_2)_w = (\Pi_1)_w$ pour presque toute $w$.
			
			Puisque $GU^1$ et $GU^2$ sont des formes intérieures pures, le théorème \ref{itm: global} implique que le $L$-paramètre global $\Psi_1$ de $\Pi_1$ et le $L$-paramètre global $\Psi_2$ de $\Pi_2$ sont identiques. Comme $(\Pi_2)_p = \widetilde{\pi}_p$ est dans le paquet $\Pi_{\varphi}(G(\Q_p))$, on en déduit alors que la représentation $ (\Pi_1)_p = \Pi_p $ est dans le paquet $\Pi_{\varphi}(J_b(\Q_p))$.	
		\end{proof}
		
		Puisque le paquet $\Pi_{\varphi}(J_b(\Q_p))$ ne contient que des représentations supercuspidales, d'après le lemme \ref{itm : même paquet} on voit que $\Pi_p$ est supercuspidale dès que $\Pi^p = (\widetilde{\Pi}^*)^p$. D'autre part, le lemme \ref{itm: dégénère} implique que $\Ext^t_{\J_{b}(\Q_p)} \left( H^q_c(\mathcal{M}_{U_p}, \overline{\Q}_{\ell}(n-1)), \Pi_p \right) = 0$ si $t > 0$. Donc la suite spectrale ci-dessus  dégénère. On obtient donc des isomorphismes:
		\begin{equation} \phantomsection \label{itm: 9.0}
			| \ker^1 (\Q, \overset{\bullet}{G}) | \sum_{\substack{\Pi \in \mathcal{A}(I^{\phi}) \\ \Pi^{p} = (\widetilde{\Pi}^*)^{p} }} \left(  \hom_{\J_b(\Q_p)} \left( H^i_c(\mathcal{M}), \Pi_p \right)_{cusp}  \right) = \left( H^{i}(\Sh, \mathcal{L}_{\xi}) \right)_{p-cusp}[\widetilde{\Pi}^{\infty, p}].
		\end{equation} 
		
		D'autre part, la formule de Matsushima nous donne une décomposition:
		\begin{equation} \phantomsection \label{itm : 10.0}
			H^i(\Sh, \mathcal{L}_{\xi})_{\text{p-cusp}} = | \ker^1 (\Q, \overset{\bullet}{G}) | \sum_{\substack{\Pi \in \mathcal{A}(\overset{\bullet}{G})_{\xi} \\ \Pi_p \text{supercuspidale}} } \Pi^{\infty} \otimes \rho_i(\Pi^{\infty}),
		\end{equation}
		où $\rho_i(\Pi^{\infty})$ est une représentation continue de dimension finie de $Gal(\overline{\overset{\bullet}{F}}/\overset{\bullet}{F})$ et $\mathcal{A}(\overset{\bullet}{G})_{\xi}$ est l'ensemble des représentations automorphes de $\overset{\bullet}{G}$ cohomologiques pour $\xi$.
		
		Puisque l'application $\eta_{\chi_{\kappa}, *}$ est injective (voir \ref{itm : changement de base}) et $\Pi_v$ est supercuspidale, le même argument que celui dans $A.7.8$ de \cite{Far04} montre le résultat suivant :
		\[
		\rho_i (\pi'_p \otimes \Pi^{\infty, p}) = \left\lbrace \begin{array}{ccc}
			0 & si & i \neq n-1 = \dim \Sh \\
			\left( r_{\mu} \circ \varphi_{F_p} \right)^{m(\Pi)}  \otimes | \cdot |^{-\frac{n-1}{2}}  & si & i = n-1.
		\end{array} \right.
		\]
		
		En comparant les égalités (\ref{itm: 9.0}) et (\ref{itm : 10.0}) avec le calcul de $\rho_i (\pi'_p \otimes \Pi^{\infty, p})$, on en déduit que :
		\begin{equation} \phantomsection \label{itm : 1er résultat}
			\sum_{\substack{\overline{\pi}_p \in \Pi_{\varphi}(J_b(\Q_p))}} \left(\hom_{\J_b(\Q_p)} \left( H^{n-1}_c(\mathcal{M}), \overline{\pi}_p \right)_{cusp} \right)^{a_{I^{\phi}}(\overline{\pi}_p)} = \sum_{\substack{\pi'_p \in \Pi_{\varphi}(G(\Q_p))}} \pi'_p \otimes \left( r_{\mu} \circ \varphi_{F_p} \right)^{a_{\overset{\bullet}{G}}(\pi_p')}  \otimes | \cdot |^{-\frac{n-1}{2}},
		\end{equation}
		où $a_{I^{\phi}}(\overline{\pi}_p)$ est la multiplicité de $ \breve{\xi} \otimes \overline{\pi}_p \otimes \widetilde{\Pi}^{\infty, p} $ dans l'espace des formes automorphes de $I^{\phi}$ et $a_{\overset{\bullet}{G}}(\pi_p')$ est la multiplicité de $ \pi_{\xi} \otimes \pi_p' \otimes \widetilde{\Pi}^{\infty, p} $ dans l'espace des formes automorphes de $\overset{\bullet}{G}$ (d'après le théorème 3.1.7 de \cite{HL04}, la multiplicité $a_{\overset{\bullet}{G}}(\pi_p')$ ne dépend pas de $\pi_{\xi} \in \Pi (\rho)$).
		
		Pour $\widetilde{\pi}_p$ la représentation supercuspidale dans $ \Pi_{\varphi}(G(\Q_p)) $ fixée, on a en particulier l'égalité :
		\[
		\sum_{\pi_p \in \Pi_{\varphi}(J_b(\Q_p))} (\sigma_{\pi_p, \pi_p'})^{a_{I^{\phi}}(\pi_p)} = \left( r_{\mu} \circ \varphi_{F_p} \right)^{a_{\overset{\bullet}{G}}(\widetilde{\pi}_p)}  \otimes | \cdot |^{-\frac{n-1}{2}}.
		\]
		
		De plus comme $\varphi$ est un paramètre discret et $\mu$ est de signature $(1, n-1), (0, n), \cdots ,(0, n)$ on en déduit que $ (r_{\mu} \circ \varphi_{ F_p}) $ est une représentation de dimension $n$ et est une somme (sans multiplicité) des représentations irréductibles deux à deux non isomorphes. D'autre part on a $ a_{\overset{\bullet}{G}}(\widetilde{\pi}_p) > 0 $ et d'après la proposition \ref{itm : résultat partiel} on a 
		$  \sum_{\pi_p \in (\Pi_{\varphi}(\J_b(\Q_p)))} \dim \sigma_{\pi_p, \pi_p'} = n, $	
		ce qui implique	
		$$\sum_{\pi_p \in \Pi_{\varphi}(J_b(\Q_p))} \sigma_{\pi_p, \pi_p'} = \left( r_{\mu} \circ \varphi_{F_p} \right)  \otimes | \cdot |^{-\frac{n-1}{2}}.$$ 
	\end{proof}
	\bibliography{Bib}

\begin{thebibliography}{KMSW14}

\bibitem[Art13a]{Art4}
James Arthur.
\newblock Classifying automorphic representations.
\newblock In {\em Current developments in mathematics 2012}, pages 1--58. Int.
  Press, Somerville, MA, 2013.

\bibitem[Art13b]{Arthur}
James Arthur.
\newblock {\em The endoscopic classification of representations}, volume~61 of
  {\em American Mathematical Society Colloquium Publications}.
\newblock American Mathematical Society, Providence, RI, 2013.
\newblock Orthogonal and symplectic groups.

\bibitem[BC83]{BC}
A.~Borel and W.~Casselman.
\newblock {$L\sp{2}$}-cohomology of locally symmetric manifolds of finite
  volume.
\newblock {\em Duke Math. J.}, 50(3):625--647, 1983.

\bibitem[Bor74]{Bo74}
Armand Borel.
\newblock Stable real cohomology of arithmetic groups.
\newblock {\em Ann. Sci. \'{E}cole Norm. Sup. (4)}, 7:235--272 (1975), 1974.

\bibitem[Boy99]{Boyer99}
P.~Boyer.
\newblock Mauvaise r\'{e}duction des vari\'{e}t\'{e}s de {D}rinfeld et
  correspondance de {L}anglands locale.
\newblock {\em Invent. Math.}, 138(3):573--629, 1999.

\bibitem[Boy09]{Boyer09}
Pascal Boyer.
\newblock Monodromie du faisceau pervers des cycles \'{e}vanescents de quelques
  vari\'{e}t\'{e}s de {S}himura simples.
\newblock {\em Invent. Math.}, 177(2):239--280, 2009.

\bibitem[CHLN11]{CHLN}
Laurent Clozel, Michael Harris, Jean-Pierre Labesse, and Bao-Ch\^{a}u Ng\^{o},
  editors.
\newblock {\em On the stabilization of the trace formula}, volume~1 of {\em
  Stabilization of the Trace Formula, Shimura Varieties, and Arithmetic
  Applications}.
\newblock International Press, Somerville, MA, 2011.

\bibitem[Fal02]{Fal}
Gerd Faltings.
\newblock A relation between two moduli spaces studied by {V}. {G}. {D}rinfeld.
\newblock In {\em Algebraic number theory and algebraic geometry}, volume 300
  of {\em Contemp. Math.}, pages 115--129. Amer. Math. Soc., Providence, RI,
  2002.

\bibitem[Far04]{Far04}
Laurent Fargues.
\newblock Cohomologie des espaces de modules de groupes {$p$}-divisibles et
  correspondances de {L}anglands locales.
\newblock Number 291, pages 1--199. 2004.
\newblock Vari\'{e}t\'{e}s de Shimura, espaces de Rapoport-Zink et
  correspondances de Langlands locales.

\bibitem[FGL08]{FGL}
Laurent Fargues, Alain Genestier, and Vincent Lafforgue.
\newblock {\em L'isomorphisme entre les tours de {L}ubin-{T}ate et de
  {D}rinfeld}, volume 262 of {\em Progress in Mathematics}.
\newblock Birkh\"{a}user Verlag, Basel, 2008.

\bibitem[GL16]{LJ}
L.~Guerberoff and J.~Lin.
\newblock Galois equivariance of critical values of l-functions for unitary
  groups.
\newblock 2016.

\bibitem[Hen00]{Hen00}
Guy Henniart.
\newblock Une preuve simple des conjectures de {L}anglands pour {${\rm GL}(n)$}
  sur un corps {$p$}-adique.
\newblock {\em Invent. Math.}, 139(2):439--455, 2000.

\bibitem[HL04]{HL04}
Michael Harris and Jean-Pierre Labesse.
\newblock Conditional base change for unitary groups.
\newblock {\em Asian J. Math.}, 8(4):653--683, 2004.

\bibitem[HT01]{HT01}
Michael Harris and Richard Taylor.
\newblock {\em The geometry and cohomology of some simple {S}himura varieties},
  volume 151 of {\em Annals of Mathematics Studies}.
\newblock Princeton University Press, Princeton, NJ, 2001.
\newblock With an appendix by Vladimir G. Berkovich.

\bibitem[JS81]{JS}
H.~Jacquet and J.~A. Shalika.
\newblock On euler products and the classification of automorphic
  representations. {II}.
\newblock {\em Amer. J. Math.}, 103:777--815, 1981.

\bibitem[KMSW14]{KMSW}
Tasho Kaletha, Alberto Minguez, Sug~Woo Shin, and Paul-James White.
\newblock Endoscopic classification of representations: Inner forms of unitary
  groups, 2014.

\bibitem[Kot84]{Kot84}
Robert~E. Kottwitz.
\newblock Stable trace formula: cuspidal tempered terms.
\newblock {\em Duke Math. J.}, 51(3):611--650, 1984.

\bibitem[Kot92]{Kot92}
Robert~E. Kottwitz.
\newblock Points on some {S}himura varieties over finite fields.
\newblock {\em J. Amer. Math. Soc.}, 5(2):373--444, 1992.

\bibitem[Kot97]{Kot97}
Robert~E. Kottwitz.
\newblock Isocrystals with additional structure. {II}.
\newblock {\em Compositio Math.}, 109(3):255--339, 1997.

\bibitem[Kot14]{Kot14}
Robert~E. Kottwitz.
\newblock B(g) for all local and global fields.
\newblock 2014.

\bibitem[Lan79]{Lang79}
R.~P. Langlands.
\newblock Automorphic representations, {S}himura varieties, and motives. {E}in
  {M}\"{a}rchen.
\newblock In {\em Automorphic forms, representations and {$L$}-functions
  ({P}roc. {S}ympos. {P}ure {M}ath., {O}regon {S}tate {U}niv., {C}orvallis,
  {O}re., 1977), {P}art 2}, Proc. Sympos. Pure Math., XXXIII, pages 205--246.
  Amer. Math. Soc., Providence, R.I., 1979.

\bibitem[Lan89]{Lang}
R.~P. Langlands.
\newblock On the classification of irreducible representations of real
  algebraic groups.
\newblock In {\em Representation theory and harmonic analysis on semisimple
  {L}ie groups}, volume~31 of {\em Math. Surveys Monogr.}, pages 101--170.
  Amer. Math. Soc., Providence, RI, 1989.

\bibitem[Lan16]{Lan}
Kai-Wen Lan.
\newblock Compactifications of {PEL}-type {S}himura varieties in ramified
  characteristics.
\newblock {\em Forum Math. Sigma}, 4:e1, 98, 2016.

\bibitem[Loo88]{Lo}
Eduard Looijenga.
\newblock {$L^2$}-cohomology of locally symmetric varieties.
\newblock {\em Compositio Math.}, 67(1):3--20, 1988.

\bibitem[LR91]{LoR}
E.~Looijenga and M.~Rapoport.
\newblock Weights in the local cohomology of a {B}aily-{B}orel
  compactification.
\newblock In {\em Complex geometry and {L}ie theory ({S}undance, {UT}, 1989)},
  volume~53 of {\em Proc. Sympos. Pure Math.}, pages 223--260. Amer. Math.
  Soc., Providence, RI, 1991.

\bibitem[LS18]{LS}
Kai-Wen Lan and Beno\^{\i}t Stroh.
\newblock Nearby cycles of automorphic \'{e}tale sheaves.
\newblock {\em Compos. Math.}, 154(1):80--119, 2018.

\bibitem[Man05]{Man05}
Elena Mantovan.
\newblock On the cohomology of certain {PEL}-type {S}himura varieties.
\newblock {\em Duke Math. J.}, 129(3):573--610, 2005.

\bibitem[Man08]{Man08}
Elena Mantovan.
\newblock On non-basic {R}apoport-{Z}ink spaces.
\newblock {\em Ann. Sci. \'{E}c. Norm. Sup\'{e}r. (4)}, 41(5):671--716, 2008.

\bibitem[Man11]{Man11}
Elena Mantovan.
\newblock {$l$}-adic \'{e}tale cohomology of {PEL} type {S}himura varieties
  with non-trivial coefficients.
\newblock In {\em W{IN}---women in numbers}, volume~60 of {\em Fields Inst.
  Commun.}, pages 61--83. Amer. Math. Soc., Providence, RI, 2011.

\bibitem[Mg07]{Moe}
Colette M\oe~glin.
\newblock Classification et changement de base pour les s\'{e}ries discr\`etes
  des groupes unitaires {$p$}-adiques.
\newblock {\em Pacific J. Math.}, 233(1):159--204, 2007.

\bibitem[Mok15]{Mok}
Chung~Pang Mok.
\newblock Endoscopic classification of representations of quasi-split unitary
  groups.
\newblock {\em Mem. Amer. Math. Soc.}, 235(1108):vi+248, 2015.

\bibitem[Mor10]{Mo}
Sophie Morel.
\newblock {\em On the cohomology of certain noncompact {S}himura varieties},
  volume 173 of {\em Annals of Mathematics Studies}.
\newblock Princeton University Press, Princeton, NJ, 2010.
\newblock With an appendix by Robert Kottwitz.

\bibitem[MT02]{MT}
Abdellah Mokrane and Jacques Tilouine.
\newblock Cohomology of {S}iegel varieties with {$p$}-adic integral
  coefficients and applications.
\newblock Number 280, pages 1--95. 2002.
\newblock Cohomology of Siegel varieties.

\bibitem[Rap95]{Rap94}
Michael Rapoport.
\newblock Non-{A}rchimedean period domains.
\newblock In {\em Proceedings of the {I}nternational {C}ongress of
  {M}athematicians, {V}ol. 1, 2 ({Z}\"{u}rich, 1994)}, pages 423--434.
  Birkh\"{a}user, Basel, 1995.

\bibitem[RZ96]{RZ96}
M.~Rapoport and Th. Zink.
\newblock {\em Period spaces for {$p$}-divisible groups}, volume 141 of {\em
  Annals of Mathematics Studies}.
\newblock Princeton University Press, Princeton, NJ, 1996.

\bibitem[She14]{Shen}
Xu~Shen.
\newblock On the {H}odge-{N}ewton filtration for {$p$}-divisible groups with
  additional structures.
\newblock {\em Int. Math. Res. Not. IMRN}, (13):3582--3631, 2014.

\bibitem[Shi12a]{Shin1}
Sug~Woo Shin.
\newblock Automorphic {P}lancherel density theorem.
\newblock {\em Israel J. Math.}, 192(1):83--120, 2012.

\bibitem[Shi12b]{Shin}
Sug~Woo Shin.
\newblock On the cohomology of {R}apoport-{Z}ink spaces of {EL}-type.
\newblock {\em Amer. J. Math.}, 134(2):407--452, 2012.

\bibitem[SS90]{SS}
Leslie Saper and Mark Stern.
\newblock {$L_2$}-cohomology of arithmetic varieties.
\newblock {\em Ann. of Math. (2)}, 132(1):1--69, 1990.

\bibitem[SW20]{SW17}
Peter Scholze and Jared Weinstein.
\newblock Berkeley lectures on $p$-adic geometry.
\newblock 2020.

\end{thebibliography}
	\bibliographystyle{alpha}
\end{document}